\documentclass[11pt]{amsart}
\usepackage{amsmath,cite}
\usepackage{epsfig,amssymb,amsthm,tocvsec2}
\usepackage{amsmath}
\usepackage{sidecap}
\usepackage{subfigure}
\usepackage{wrapfig}
\usepackage{graphicx}
\usepackage{epstopdf}
\usepackage{pict2e}
\usepackage{threeparttable}
\usepackage{graphicx}
\usepackage{setspace}
\usepackage{float}
\usepackage{color}
\usepackage{mathrsfs}
\usepackage{epsfig,epsf,latexsym,subfig,version}
\usepackage{mdwlist}
\usepackage{paralist}
\usepackage{setspace}
\usepackage{enumitem}
\setlist{nolistsep}
\usepackage{indentfirst}

\catcode`\@=11
\@addtoreset{equation}{section} 
\theoremstyle{plain}
\newtheorem{thm}{Theorem}[section]
\newtheorem{lem}{Lemma}[section]

\newtheorem{remark}{\textbf{Remark}}[section]
\newtheorem{algorithm}{\textbf{Scheme}}

\newcommand{\eps}{\epsilon}

\newcommand{\bm}{\boldsymbol}
\newcommand{\bu}{{\mathbf u}}
\newcommand{\bv}{{\mathbf v}}

\newcommand{\Grad}[1]{\nabla #1}

\newcommand{\be}{\begin{equation}}
\newcommand{\ee}{\end{equation}}

\newcommand{\bse}{\begin{subequations}}
\newcommand{\ese}{\end{subequations}}
\def\benl{\begin{eqnarray*}}
\def\eenl{\end{eqnarray*}}

\def\bu{\bm{u}}

\def\be{\bm{e}}

\def\bx{\bm{x}}
\def\bX{\bm{X}}
\def\by{\bm{y}}

\def\bmu1{\bm{\mu_1}}

\newcommand{\ben}{\begin{eqnarray}}
\newcommand{\een}{\end{eqnarray}}
\newcommand{\beq}{\begin{equation}}
\newcommand{\eeq}{\end{equation}}
\newcommand{\bea}{\begin{array}}
\newcommand{\eea}{\end{array}}
\newcommand{\bef}{\begin{figure}[H]}
\newcommand{\eef}{\end{figure}}

\begin{document}
\bibliographystyle{plain}
\graphicspath{ {Figures/} }

\title[A new interface capturing method for Allen-Cahn type equations]{A new  interface capturing method for Allen-Cahn type equations based on a flow dynamic  approach in Lagrangian coordinates, I. One-dimensional case}

\author[Q. Cheng,  Chun liu and J. Shen]{Qing Cheng $^{\dagger}$, Chun Liu and  Jie Shen $^{\S}$}

\date{}

\maketitle

\begin{abstract}
We develop a new Lagrangian approach ---  flow dynamic  approach  to effectively capture the interface in the Allen-Cahn type equations. The underlying principle of this approach   is the Energetic Variational Approach (EnVarA), motivated by Rayleigh and Onsager \cite{onsager1931reciprocal,onsager1931reciprocal2}.  Its main advantage,    comparing with numerical methods in Eulerian coordinates, is that thin interfaces can be effectively captured with few points in the Lagrangian coordinate.  We concentrate in the one-dimensional case  and construct  numerical schemes for the trajectory equation in Lagrangian coordinate that  obey the variational structures, and as a consequence, are  energy dissipative. Ample numerical results are provided to
show that only a fewer points  are enough to resolve very thin interfaces by using our Lagrangian approach.
\end{abstract}

\section{Introduction}
Diffuse interface methods have been widely used in many applications in science and engineering, especially in describing phase transitions  \cite{allen1979microscopic}, microstructure coarsening \cite{li2003thin}, porous medium \cite{vazquez2007porous}, liquid crystals \cite{leslie1979theory} or vesicle membrane \cite{DuLiWa04,DuLiWa06}.
In this paper we will explore the  Allen-Cahn model  which is  related to
the  studies of the dynamic behavior of sharp interface. The standard Allen-Cahn model  in an isothermal closed system, following  the First and Second Laws of Thermodynamics,  yields an energy dissipative law \cite{de2013non,greven2014entropy,giga2017variational}:
\begin{equation}\label{normal:law}
\frac{d}{dt}\mathcal{E}(f)=-\bf \Delta,
\end{equation}
Where  $\mathcal{E}(f)$ is the total free energy and $\bf \Delta$ is  attributed to entropy production of measuring energy dissipative rate.  The Allen-Cahn model, with $\mathcal{E}(f)=\int_{\Omega}\frac 12|\Grad f|^2+\frac{1}{4\eps^2}(f^2-1)^2d\bx$, can also be viewed as the $L^2$ gradient flow of the Ginzburg-Landau functional $\mathcal{E}(f)$, i.e.,
  its equation can be derived by taking variational derivative of the free energy
with respect to the order parameter in $L^2$-topology
\begin{eqnarray}\label{allen:0}
\begin{aligned}
f_t = -\frac{\delta \mathcal{E}(f)}{\delta f}.
\end{aligned}
\end{eqnarray}
 In this formulation, the solution will be able to capture the 
free interface motion by mean curvature \cite{chen2011mass,evans1992phase,bronsard1991motion,katsoulakis1995generalized,ilmanen1993convergence}. 
It is  well-known that  solutions of Allen-Cahn equation will develop interfaces with thickness  $O(\eps)$, which renders its numerical simulation difficult as resolving  thin interfaces will require expensive computational efforts.     How to effectively capture thin interfacial layers  has been an active research topic. 
 
Many efforts have been devoted to design efficient numerical schemes to capture the  interface of transient phenomena by using the
energy dissipative law \eqref{normal:law} and the underlying variational structure, such as spectral method \cite{liu2003phase,shen2010numerical},  moving mesh method \cite{huang1994moving,mackenzie2002moving,cao2002moving,feng2006spectral,shen2009efficient,li2001moving},  
adaptive time stepping method and adaptive spatial finite element methods which have been considered in \cite{zhang2009numerical,feng2004analysis}.  We refer to \cite{Du.F19} for a up-to-date review on this subject.

 Traditional methods for interface capturing are mainly developed in Eulerian coordinate based on various moving mesh strategies.
The objective of this paper is to develop a  new Lagrangian approach for  interface capturing by using the Energetic Variational Approach \cite{eisenberg2010energy,xu2014energetic}, since the energy dissipative law with kinematic relations of variables employed in the system describes all the physical and
mechanical phenomenon for  mathematical models. To be specific,  for the Allen-Cahn model \eqref{allen:0}, we introduce a  transport equation which connects the Eulerian coordinate and Lagrangian coordinate
under a suitably defined flow map, and  derive the  trajectory equation  for Allen-Cahn model  following the Least Action Principle and Maximum Dissipative Principle by using the flow map. The main feature of this approach is that the solution of the trajectory equation will be {\em free of thin interfaces} if the flow map is suitably defined, so that it can be solved with a resolution which is {\em independent} of $\epsilon$, the interfacial thickness in the Eulerian coordinates. The is due to the fact that  we target the mesh velocity 
 by using the trajectory equation   which is consistent with the original Allen-Cahn equation, rather than  adding  moving mesh PDEs used in Eulerian approaches \cite{huang1994moving,cao2002moving,li2001moving}.

Unlike the Allen-Cahn equation which takes a simple form in the Eulerian coordinates, the trajectory equation is a non-standard, highly nonlinear parabolic equation, which also possesses an energy
dissipative law. 
  We develop efficient numerical schemes for the trajectory equation    which  preserve the variational structure and satisfy the energy
dissipative law. Furthermore, they can be interpreted as the Euler-Lagrange equations of convex functionals so that they can be effectively solved by using a Newton type iteration. 
Our Lagrangian approach has a distinct advantage for interface problems.  Meshes, in the Eulerian coordinate through the flow map,  will automatically move to the region of thin interfaces without using any adaptive mesh movement strategy, and consequently thin interfaces can be well resolved with only a few points. In fact, as the interfacial width $\epsilon$ decreases,   our numerical results show that lesser points are needed to resolve the  interfaces with our Lagrangian approach.  

The reminder of this paper is structured as follows. In Section $2$ we introduce the flow dynamic  approach for Allen-Cahn type equations. 
In Section $3$ we develop semi-discrete and fully discrete numerical schemes for trajectory equations   in Lagrangian coordinates.  In Section $4$, we consider the two-dimensional   axi-symmetric case.  In Section $5$ we present numerical results to demonstrate the efficiency of our new approach.  Some concluding remarks are given in Section $6$, followed by an appendix on the energetic variational interpretation of our approach.

\section{Flow dynamic  approach}
In this section, we introduce the flow dynamic approach to  capture the diffusive  interface  in the Allen-Cahn equation. 

Let  $\Omega_{\bx} \in R^d \, (d=1,2,3)$ be an open bounded domain.  To fix  the idea,  we consider the following Allen-Cahn equation with Dirichlet boundary condition in $\Omega_{\bx}$:
\begin{eqnarray}
&&f_t-\Delta f + F'(f)=0;\quad f(\bx,t)|_{\partial\Omega}=0; \label{G_AC_Var:1}\\
&&f(\bx,0)=f_0(\bx), \label{G_AC_Var:2}
\end{eqnarray}
where $F(f)$ is a nonlinear potential, a typical example is the double well potential $F(f)=\frac{1}{4\eps^2}(f^2-1)^2$.

It is easy to see that  the system \eqref{G_AC_Var:1}-\eqref{G_AC_Var:2}   satisfies the following energy dissipative law
\begin{eqnarray}\label{en:diss}
\begin{aligned}
\frac{d}{dt}\int_{\Omega_x}\frac 12|\Grad f|^2&+ F(f) d\bx
=-\int_{\Omega_x}|f_t|^2d\bx.
\end{aligned}
\end{eqnarray}

\subsection{\bf Flow map and deformation tensor}
Given an initial position or a reference configuration $\bX$, and a velocity field $\bu$, we define a flow map $\bx(\bX,t)$ by
\begin{equation}
\begin{aligned}\label{flow_map_AC}
&\frac{d\bx(\bX,t)}{dt}=\bu(\bx(\bX,t),t),\\
&\bx(\bX,0)=\bX.
\end{aligned}
\end{equation}

\begin{figure}[htbp]
\centering
\includegraphics[width=0.45\textwidth,clip==]{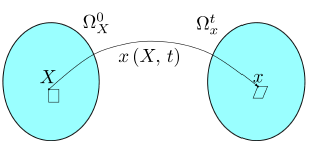}
\caption{A schematic illustration of a flow map $\bx(\bX,t)$ at a fixed time $t$: $\bx(\bX,t)$ maps $\Omega_{\bX}^0$ to $\Omega_{\bx}^t$. $\bX$ is the lagrangian coordinate while $\bx$ is  the Eulerian coordinate, and  $F(\bX,t)=\frac{\partial \bx(\bX,t)}{\partial \bX}$ represents the deformation associated with the flow map. }\label{flow_map}
\end{figure}
The flow map defined by \eqref{flow_map_AC} describes a particle moving from a initial
configuration $\bX$  to a instantaneous configuration $\bx(\bX,t)$ with velocity $\bu$, i.e., 
 $\bx(\bX,t) $ is the Eulerian coordinate and $\bX$ represents Lagrangian coordinate, with  the deformation tensor or Jacobian $F=\frac{\partial \bx(\bX,t)}{\partial \bX}$ \cite{gurtin2010mechanics}.

\begin{remark}Let $f$ be the solution of the Allen-Cahn equation \eqref{G_AC_Var:1}-\eqref{G_AC_Var:2}. We assume that $\bu$ is the velocity such that 
\begin{equation}\label{transport}
f_t+(\bu\cdot\Grad_{\bx})f=0.
\end{equation}
Then, the above transport equation and 
the flow map  defined in \eqref{flow_map_AC} determine the following  kinematic relationship between Eulerian coordinate and Lagrangian coordinate:
\begin{eqnarray}
&&\frac{d}{dt}f(\bx(\bX,t),t)=f_t+(\bu\cdot\Grad_{\bx})f=0, \label{map:1}\\
\end{eqnarray}
which leads to 
\begin{equation}\label{essential}
\hat f_0(\bX)= f(\bx(\bX,t),t)=f(\bx,0)=f_0(\bx) \quad\forall t, 
 \end{equation}
 where $\hat f_0(\bX)$ is the  initial condition in the Lagrangian coordinate. Since $\bx(\bX,0)=\bX$, we have $\hat f_0(\cdot)=f_0(\cdot)$.
 
 Once we have the flow map $\bx(\bX,t)$, we set $\phi(\bX;t)=\bx(\bX,t)$ for each $t$. Then, we derive from \eqref{map:1} that  the solution of \eqref{G_AC_Var:1}-\eqref{G_AC_Var:2} is given by
\begin{eqnarray}\label{comp}
f(\bx,t)= f_0(\phi^{-1}(\bx;t)). 
\end{eqnarray}
\end{remark}

Assuming again the transport equation $f_t+\bu\cdot\Grad f=0$ is satisfied,  we can rewrite the Allen-Cahn equation \eqref{G_AC_Var:1} as
\begin{equation}\label{All:High}
\begin{split}
\bu\cdot\Grad_{\bx} f=-\Delta_{\bx} f + F'(f).
\end{split}
\end{equation}
Just as the Allen-Cahn system \eqref{G_AC_Var:1}-\eqref{G_AC_Var:2}, we have the new energy dissipative law for \eqref{All:High}
\begin{equation}\label{en:diss:flow1}
\begin{aligned}
&\frac{d}{dt}\int_{\Omega_{\bx}}\frac 12|\Grad_{\bx} f|^2+F(f) d\bx
=-\int_{\Omega_{\bx}}|\bu\cdot \Grad_{\bx} f|^2d\bx,
\end{aligned}
\end{equation}
which is obtained by taking the inner product  of the  \eqref{All:High} with $f_t$ and  using the transport equation \eqref{transport}.

The equation \eqref{All:High} can also be interpreted as a force balance relation which can be derived from the energetic variational approach. For the reader's convenience, we provide the detail in the Appendix.

\subsection{\bf Lagrangian Formulation}
Since the formulation of Allen-Cahn equation for multi-dimensions in Lagrangian coordinate are more complicated, we shall consider first the one dimension case.


 Thanks to \eqref{essential}, we have  the  1-D chain rule $\partial_x f=f'_0(X)(\frac{\partial x}{\partial X})^{-1}$. Then, setting $\bx=\bx(\bX,t)$ in \eqref{All:High}, we can rewrite the  equation \eqref{All:High} in Lagrangian coordinate in 1-D  as 
\begin{eqnarray}\label{AC_V-Lag}
&& x_t(X,t)f'_0(X)(\frac{\partial x}{\partial X})^{-1}=
-\partial_X\left(f'_0(X)(\frac{\partial x}{\partial X})^{-1}\right)(\frac{\partial x}{\partial X})^{-1} + F'(f_0(X)),\label{AC_V-Lag:1}\\
&&x|_{\partial \Omega}=X|_{\partial \Omega}, ~~ \quad x(X,0)=X,~~X\in \Omega_X.\label{AC_V-Lag:2}
\end{eqnarray}

We observe  that the last term in \eqref{AC_V-Lag:1}  is just a forcing term for the nonlinear parabolic equation  in the Lagrangian coordinate. Hence, its  solution $x(X,t)$ should not involve thin interfacial layers as does the solution of \eqref{G_AC_Var:1}-\eqref{G_AC_Var:2} in the Eulerian coordinate.

\begin{thm}
The system \eqref{AC_V-Lag:1}-\eqref{AC_V-Lag:2} satisfies the following energy dissipative law
\begin{equation}\label{ener:diss:2:Lag_V}
\begin{split}
&\frac{d}{dt}\int_{\Omega_X}\frac{\partial x}{\partial X}\{\frac 12|f'_0(X)(\frac{\partial x}{\partial X})^{-1}|^2+F(f_0(X))\}dX
\\&=-\int_{\Omega_X}\frac{\partial x}{\partial X}|x_t(X,t)f'_0(X)(\frac{\partial x}{\partial X})^{-1}|^2dX.
\end{split}
\end{equation}
\end{thm}
\begin{proof}

 Taking the inner product of equation \eqref{AC_V-Lag:1} with $-x_tf'_0(X)$,  since $x_t|_{\partial\Omega}=0$ due to the boundary condition \eqref{AC_V-Lag:2},  we derive by  integration by parts that
\begin{eqnarray*}\label{en:decay:pr}
\begin{split}
-\int_{\Omega_X}&\frac{\partial x}{\partial X}|x_t(X,t)f'_0(X)(\frac{\partial x}{\partial X})^{-1}|^2dX\\
&=(\partial_X\left(f'_0(X)(\frac{\partial x}{\partial X})^{-1}\right)(\frac{\partial x}{\partial X})^{-1},x_tf'_0(X))-(F'(f_0(X)),x_tf'_0(X))
\\&=(\frac 12\partial_X|f'_0(X)(\frac{\partial x}{\partial X})^{-1}|^2,x_t)
-(\partial_X F(f_0(X)),x_t)
\\&=-(\frac 12|f'_0(X)(\frac{\partial x}{\partial X})^{-1}|^2,\partial_Xx_t)
+(F(f_0(X)),\partial_Xx_t)\\
&=\frac{d}{dt}\int_{\Omega_X}\frac{\partial x}{\partial X}\{\frac 12|f'_0(X)(\frac{\partial x}{\partial X})^{-1}|^2+F(f_0(X))\}dX.
\end{split}
\end{eqnarray*}
The last equality is true since
\begin{equation*}
 \frac{d}{dt}\{\frac{\partial x}{\partial X}|f'_0(X)(\frac{\partial x}{\partial X})^{-1}|^2\}=
  |f'_0(X)|^2 \frac{d}{dt} (\frac{\partial x}{\partial X})^{-1}=- |f'_0(X)|^2 (\frac{\partial x}{\partial X})^{-2}  \frac{d}{dt} \frac{\partial x}{\partial X}.
\end{equation*}
\end{proof}

\begin{remark}\label{energy:coor}
We note that the energy in Eulerian coordinate $$E(f)=\int_{\Omega_x}\frac 12|\partial_x f|^2+F(f) d\bx,$$ is  equal to the
energy $E(x(X,t))$ in Lagrangian coordinate,
$$E(x(X,t))=\int_{\Omega_X}\frac{\partial x}{\partial X}\{\frac 12|f'_0(X)(\frac{\partial x}{\partial X})^{-1}|^2+F(f_0(X))\}dX.$$
This can be easily verified using the chain rule $\partial_x f=f'_0(X)(\frac{\partial x}{\partial X})^{-1}$ and the  identity $dx=\frac{\partial x}{\partial X}dX$.
\end{remark}

\begin{remark}
Instead of solving \eqref{G_AC_Var:1}-\eqref{G_AC_Var:2} in the Eulerian coordinate $\bx$ with potentially thin interfacial layers, such as the case if $F(f)=\frac{1}{4\epsilon^2}(f^2-1)^2$ with $\epsilon\ll 1$, which need to be resolved with high spatial resolution, we can solve it in the Lagrangian coordinate $\bX$  free of   thin interfacial layers as follows:
\begin{itemize}
\item  Solve the flow map $\bx(\bX,t)$ from the trajectory equation \eqref{AC_V-Lag:1}-\eqref{AC_V-Lag:2};
\item Once we have the flow map $\bx(\bX,t)$,  the solution of \eqref{G_AC_Var:1}-\eqref{G_AC_Var:2} is given by \eqref{comp}.
\end{itemize}

\end{remark}

\section{Numerical Schemes}
In this section, we  construct energy stable time discretization schemes for the trajectory equation \eqref{AC_V-Lag:1}-\eqref{AC_V-Lag:2} in 1D.
\subsection{ Semi-discrete-in-time schemes}
We start by  constructing  a first order  scheme for Allen-Cahn system \eqref{G_AC_Var:1}-\eqref{G_AC_Var:2} in Lagrangian coordinate.

 Given $\delta t>0$, let $t_n=n\delta t$, $n=0,1,2 \cdots \frac{T}{\delta t}$.  For any function $S(\cdot,t)$,  $S^n$ denotes a numerical  approximation to $S(\cdot,t_n)$.
 
\begin{algorithm} (a first order scheme)
\begin{eqnarray}
&&\frac{x^{n+1}-x^n}{\delta t}f'_0(X)(\frac{\partial x^{n}}{\partial X})^{-1}=-\partial_X\left(f'_0(X)(\frac{\partial x^{n+1}}{\partial X})^{-1}\right)(\frac{\partial x^{n+1}}{\partial X})^{-1} + F'(f_0(X)),\label{G_AC_V:num:1} \\
&& x^{n+1}|_{\partial \Omega}=X|_{\partial \Omega},~~ \quad x^0(X)=X,~~X\in \Omega_X.\label{G_AC_V:num:2}
\end{eqnarray}
\end{algorithm}

\begin{remark} 
Once we solve $x^{n+1}$ from  \eqref{G_AC_V:num:1}-\eqref{G_AC_V:num:2}, the approximate solution to the original equation can be obtained as  $f^{n+1}(x)=f(x^{n+1}(X))=f_0(X)=f_0(\phi^{-1}(x^{n+1}))$. The above relation also indicates that the scheme \eqref{G_AC_V:num:1}-\eqref{G_AC_V:num:2}  preserves maximum principle  since  $\max_{\forall n}|f(x^{n+1}(X))|=\max| f_0(X)|$ where $f_0(X)$ is  the initial condition  in Lagrangian coordinate.
\end{remark}

\begin{thm}\label{first:thm}
Let $x^n$ be the solution of scheme \eqref{G_AC_V:num:1}-\eqref{G_AC_V:num:2} at time $t^n$ with $\frac{\partial x^{n}}{\partial X}>0$.  Then the scheme \eqref{G_AC_V:num:1}-\eqref{G_AC_V:num:2}  admits  a unique solution $x^{n+1}$ with $\frac{\partial x^{n+1}}{\partial X}>0$,  and satisfies the following discrete energy
 law holds:
\begin{equation}\label{energy1}
\begin{split}
&\frac{E(x^{n+1})-E(x^{n})}{\delta t}\leq -\big<\frac{(f'_0(X))^2}{\frac{\partial x^n}{\partial X}}\frac{x^{n+1}-x^n}{\delta t},\frac{x^{n+1}-x^n}{\delta t}\big>,
\end{split}
\end{equation}
where $E(x)=\int_{\Omega_X}\{\frac 12|f'_0(X)(\frac{\partial x}{\partial X})^{-1}|^2+ F(f_0(X))\} \frac{\partial x}{\partial X}\, dX$.
\end{thm}
\begin{proof}
We first prove the existence and uniqueness of the solution of the scheme \eqref{G_AC_V:num:1}-\eqref{G_AC_V:num:2}. To this end, we define a nonlinear functional
\begin{equation}
\begin{split}
J(\phi)&=\int_{\Omega_X}\big\{\frac{1}{2\delta t} \frac{(f'_0(X))^2}{\frac{\partial x^n}{\partial X}}|\phi|^2 +  
\frac 12(f'_0(X))^2({\frac{\partial \phi}{\partial X}})^{-1}-g(X)\phi\big\} dX,
\end{split}
\end{equation}
with $g(X)=\frac{1}{\delta t}\frac{(f'_0(X))^2}{\frac{\partial x^n}{\partial X}} +f'_0(X) F(f_0(X))$.
One can  check that  \eqref{G_AC_V:num:1}-\eqref{G_AC_V:num:2} is the
Euler-Lagrange equation
\begin{equation*}
 \frac{\delta J(\phi)}{\delta \phi}|_{\phi=x^{n+1}}=0,
\end{equation*}
 and that $J(\phi)$ is a convex functional with respect to $\phi$ with $\frac{\partial \phi}{\partial X}>0$,  because of  
\begin{equation*}\label{convex:E}
\begin{split}
\frac{\partial^2}{\partial^2 \eps}\{\frac 12(f'_0(X))^2({\frac{\partial (\phi+\eps \psi)}{\partial X}})^{-1}\}=(f'_0(X))^2(\frac{\partial \phi}{\partial X})^{-3}(\frac{\partial \psi}{\partial X})^2 \geq 0\quad\forall \psi.
\end{split}
\end{equation*} 
  
 Hence, the scheme \eqref{G_AC_V:num:1}-\eqref{G_AC_V:num:2} admits a unique solution $x^{n+1}$ with $\frac{\partial x^{n+1}}{\partial X}>0$.
 
Next, we take the inner product of  \eqref{G_AC_V:num:1} with $-\frac{x^{n+1}-x^n}{\delta t}f'_0(X)$ to obtain
\begin{equation}\label{est1}
\begin{split}
&\int_{\Omega_X}\big\{\partial_X(f'_0(X)(\frac{\partial x^{n+1}}{\partial X})^{-1})(\frac{\partial x^{n+1}}{\partial X})^{-1} - F'(f_0(X))\big\}f'_0(X)\frac{x^{n+1}-x^n}{\delta t}dX
\\&=\int_{\Omega_X}\partial_X(f'_0(X)(\frac{\partial x^{n+1}}{\partial X})^{-1})f'_0(X)(\frac{\partial x^{n+1}}{\partial X})^{-1}\frac{x^{n+1}-x^n}{\delta t}dX
\\&-\int_{\Omega_X}\partial_XF(f_0(X))\frac{x^{n+1}-x^n}{\delta t}dX.
\end{split}
\end{equation}
Due to the convexity of $\frac{1}{y}$ with respect to  $y$ with  $y>0$, we have
$$(\frac{\partial x^{n}}{\partial X})^{-1} -(\frac{\partial x^{n+1}}{\partial X})^{-1}\ge -(\frac{\partial x^{n+1}}{\partial X})^{-2}(\frac{\partial x^{n}}{\partial X}-\frac{\partial x^{n+1}}{\partial X}),$$
which implies
\begin{equation}\label{est2}
\begin{split}
&\int_{\Omega_X}\partial_X(f'_0(X)(\frac{\partial x^{n+1}}{\partial X})^{-1})f'_0(X)(\frac{\partial x^{n+1}}{\partial X})^{-1}\frac{x^{n+1}-x^n}{\delta t}dX
\\&=-\int_{\Omega_X}\frac 12|f'_0(X)(\frac{\partial x^{n+1}}{\partial X})^{-1}|^2\frac{\frac{\partial x^{n+1}}{\partial X}-\frac{\partial x^{n}}{\partial X}}{\delta t}dX
\\&\geq \frac{1}{2\delta t}\int_{\Omega_X}\big((f'_0(X))^2(\frac{\partial x^{n+1}}{\partial X})^{-1}- \int_{\Omega_X}(f'_0(X))^2(\frac{\partial x^{n}}{\partial X})^{-1}\big)dX.
\end{split}
\end{equation}
On the other hand, we have
\begin{equation}\label{est3}
\begin{split}
-\int_{\Omega_X}\partial_XF(f_0(X))\frac{x^{n+1}-x^n}{\delta t}dX
=\int_{\Omega_X}F(f_0(X))\frac{\frac{\partial x^{n+1}}{\partial X}-\frac{\partial x^n}{\partial X}}{\delta t}dX.
\end{split}
\end{equation}
We then derive \eqref{energy1} from the above two relations. 

\end{proof}

\begin{algorithm} (a second-order scheme)

{\bf Step 1:} Compute a second-order extrapolation for $\frac{\partial x^{n+1}}{\partial X}$.

We set
\begin{equation}
\frac{\partial x^{n+1}_{\star}}{\partial X}=\begin{cases}\frac{\partial (2x^{n}-x^{n-1})}{\partial X},& \text{if}\; \frac{\partial x^{n}}{\partial X}\ge  \frac{\partial x^{n-1}}{\partial X},\\
\\
\frac{1}{2/{\frac{\partial x^{n}}{\partial X}} -1/{\frac{\partial x^{n-1}}{\partial X}}} ,& \text{if}\; \frac{\partial x^{n}}{\partial X}<  \frac{\partial x^{n-1}}{\partial X}.\end{cases}
\end{equation}

{\bf Step 2:}
\begin{eqnarray}
&&\frac{3x^{n+1}-4x^n+x^{n-1}}{2\delta t}f'_0(X)(\frac{\partial x^{n+1}_{\star}}{\partial X})^{-1}=-\partial_X(f'_0(X)(\frac{\partial x^{n+1}}{\partial X})^{-1})(\frac{\partial x^{n+1}}{\partial X})^{-1} + F'(f_0(X)),\label{G_AC_V:num:2:3}\\
&& x^{n+1}|_{\partial \Omega}=X|_{\partial \Omega},~~ \quad x^0(X)=X,~~X\in \Omega_X.\label{G_AC_V:num:2:4}
\end{eqnarray}
\end{algorithm}
\begin{thm}\label{second:thm}
Given $x^k$, $k=1,2,..,n$   with $\frac{\partial x^{k}}{\partial X}>0$, the numerical scheme \eqref{G_AC_V:num:2:3}-\eqref{G_AC_V:num:2:4}  admits a unique solution $x^{n+1}$  with $\frac{\partial x^{n+1}}{\partial X}>0$, and the following  energy
dissipative law is satisfied:
\begin{equation}
\begin{split}
\frac{E(x^{n+1})-E(x^{n})}{\delta t}\leq &-\big<\frac{(f'_0(X))^2}{\frac{\partial x^{n+1}_{\star}}{\partial X}}\frac{x^{n+1}-x^n}{\delta t},\frac{x^{n+1}-x^n}{\delta t}\big>
\\&-\big<\frac{(f'_0(X))^2}{\frac{\partial x^{n+1}_{\star}}{\partial X}}\frac{x^{n+1}-2x^n+x^{n-1}}{2\delta t},\frac{x^{n+1}-2x^n+x^{n-1}}{2\delta t}\big>,
\end{split}
\end{equation}
where
\begin{equation}
\begin{split}
E(x^{n+1})&=\int_{\Omega_X}\frac{\partial x^{n+1}}{\partial X}\{\frac 12|f'_0(X)(\frac{\partial x^{n+1}}{\partial X})^{-1}|^2+F(f_0(X))\}dX\\&+\frac{1}{4\delta t}\int_{\Omega_X}(f'_0(X))^2(\frac{\partial x^{n+1}_{\star}}{\partial X})^{-1}|x^{n+1}-x^n|^2dX.
\end{split}
\end{equation}
\end{thm}
\begin{proof}
As in the proof of Theorem \ref{first:thm},  one can construct a convex functional such that its Euler Lagrange equation is equivalent to
the scheme \eqref{G_AC_V:num:2:3}-\eqref{G_AC_V:num:2:4}. Hence, the scheme admits  a unique solution $x^{n+1}$  with $\frac{\partial x^{n+1}}{\partial X}>0$.  

Next, taking the inner product of equation \eqref{G_AC_V:num:2:3} with $-f_0'(X)\frac{x^{n+1}-x^n}{\delta t}$ and using the equality,
\begin{equation}
(3a-4b+c,2(a-b))=5|a-b|^2-|b-c|^2+|a-2b+c|,
\end{equation}
 the left hand side becomes
\begin{equation}
\begin{split}
-(\frac{3x^{n+1}-4x^n+x^{n-1}}{2\delta t}&(f'_0(X))^2(\frac{\partial x^{n+1}_{\star}}{\partial X})^{-1},\frac{x^{n+1}-x^n}{\delta t})
\\=&-5\int_{\Omega_X}\frac{1}{4\delta t^2}(f'_0(X))^2(\frac{\partial x^{n+1}_{\star}}{\partial X})^{-1}|x^{n+1}-x^n|^2dX\\&+\int_{\Omega_X}\frac{1}{4\delta t^2}(f'_0(X))^2(\frac{\partial x^{n+1}_{\star}}{\partial X})^{-1}|x^{n}-x^{n-1}|^2dX\\&-\int_{\Omega_X}\frac{1}{4\delta t^2}(f'_0(X))^2(\frac{\partial x^{n+1}_{\star}}{\partial X})^{-1}|x^{n+1}-2x^{n}+x^{n-1}|^2dX.
\end{split}
\end{equation}
The right hand side can be treated exactly the same way as 
in the proof of Theorem \ref{first:thm}, see \eqref{est1}-\eqref{est2}. Combining these results, we derive the following energy dissipative law
\begin{equation*}
\begin{split}
\frac{E(x^{n+1})-E(x^{n})}{\delta t}&\leq -\frac{1}{\delta t^2}\int_{\Omega_X}(f'_0(X))^2(\frac{\partial x^{n+1}_{\star}}{\partial X})^{-1}|x^{n+1}-x^n|^2dX\\&-\frac{1}{4\delta t^2}\int_{\Omega_X}(f'_0(X))^2(\frac{\partial x^{n+1}_{\star}}{\partial X})^{-1}|x^{n+1}-2x^{n}+x^{n-1}|^2dX.
\end{split}
\end{equation*}

\end{proof}

\begin{remark}
If we consider logarithmic free energy function  $F(f)=\frac{\theta}{2}[(1+f)\log(1+f)+(1-u)\log(1-u)]-\frac{\theta_c}{2}f^2$, where $\theta,\theta_c$ are two positive constants.  Since $F'(f_0(X))$ is known in  {\bf Scheme 1} and {\bf Scheme 2},  so  the positive property of solution $0<1-f^{n+1}, 0<f^{n+1}+1$ is preserved naturally.  Then comparing with numerical methods in Eulerian coordinate, it is more convenient to solve Allen-Cahn equation with logarithmic free energy by using flow dynamic approach.
\end{remark}

\subsection{Fully discrete schemes}
We now describe fully discrete schemes with a Galerkin approximation in space. For the sake of brevity, we only consider fully discretization for {\bf Scheme 1}.  Fully discretization for {\bf Scheme 2} can be constructed similarly.

Let $V_h\in H^1(\Omega_X)$ be a finite dimensional approximation space and $V^0_h=V_h\cup H^1_0(\Omega_X)$, a fully discrete version of {\bf Scheme 1} is:
Find $x_h^{n+1} \in V_h$  such that
\begin{eqnarray}
&&(\frac{1}{2}|f_0'(X)(\frac{\partial x_h^{n+1}}{\partial X})^{-1}|^2,\partial_Xy_h)+( F'(f_0(X)),y_hf'_0(X))\nonumber \\
&&\hskip 1in= (\frac{x_h^{n+1}-x_h^n}{\delta t}f'_0(X)(\frac{\partial x_h^{n}}{\partial X})^{-1},y_hf'_0(X)),\quad\forall y_h\in V_h^0,\label{fem:AC:num:1}\\
&& x_h^{n+1}|_{\partial \Omega}=X|_{\partial \Omega}, \quad x_h^0(X)= X,~~X\in \Omega.\label{fem:AC:num:2}
\end{eqnarray}

In our numerical tests, we set the domain to be $\Omega_x=\Omega_X=(-1,1)$, and  use two different spatial discretizations. The first is the Legendre-Galerkin method \cite{Shen94b} with
\begin{equation}
 V_h:=V_N=\text{span}\{L_j(x): j=0,1,\cdots,N\},
\end{equation}
where  $L_j(x)$ is the Legendre polynomial of $j-$th degree, and
  \begin{equation}
 V^0_h:=V^0_N=\text{span}\{\phi_j(x):=L_j(x)-L_{j+2}(x): j=0,1,\cdots,N-2\}.
\end{equation}
The other  is the piecewise linear finite-element method.

The  scheme  \eqref{fem:AC:num:1}-\eqref{fem:AC:num:2}  leads to a nonlinear system: $G(x_h^{n+1})=0$ at each time step, which can be effectively solved by using, for example,  
 a damped Newton's iteration \cite{nesterov1994interior}:\
$$x_h^{n+1,k+1}=x^{n+1,k}_h-\alpha(\delta_x)(\nabla G(x^{n+1,k}_h))^{-1}G(x^{n+1,k}_h),$$  with $\alpha=O(\eps^2)$ as the damped coefficient.

Using exactly the same arguments as in the proof of Theorem \ref{first:thm}, we can establish the following:
\begin{thm}
 Given $x_h^n\in V_h$ with $\frac{\partial x_h^{n}}{\partial X}>0$.  Then the scheme \eqref{fem:AC:num:1}-\eqref{fem:AC:num:2}  admits  a unique solution $x_h^{n+1}$ with $\frac{\partial x_h^{n+1}}{\partial X}>0$,  and satisfies the following discrete energy
 law holds:
\begin{equation}\label{energy1}
\begin{split}
&\frac{E(x_h^{n+1})-E(x_h^{n})}{\delta t}\leq -\big<\frac{(f'_0(X))^2}{\frac{\partial x_h^n}{\partial X}}\frac{x_h^{n+1}-x_h^n}{\delta t},\frac{x_h^{n+1}-x_h^n}{\delta t}\big>,
\end{split}
\end{equation}
where $E(x)=\int_{\Omega_X}\{\frac 12|f'_0(X)(\frac{\partial x}{\partial X})^{-1}|^2+ F(f_0(X))\} \frac{\partial x}{\partial X}\, dX$.
\end{thm}

\section{Some extensions}
We consider in the section two immediate extensions of our flow dynamic approach.
\subsection{Allen-Cahn equation with advection}
We consider here a generalized Allen-cahn equation \eqref{move:sharp:1} with an advection term:  
\begin{eqnarray}
&&f_t+\bv\cdot\Grad_{\bx} f=(\Delta_{\bx} f-\frac{1}{\eps^2}f(f^2-1)), \label{move:sharp:1}
\end{eqnarray}
where $\bv$ is a given velocity field.
We still assume that there exists a velocity field $\bu$ satisfying the  kinematic equation 
\begin{eqnarray}
&&f_t+\bu\cdot\Grad_{\bx} f=0,\label{move:sharp:2}
\end{eqnarray}
so we can define the flow map  \eqref{flow_map_AC}. Using \eqref{move:sharp:2}, we can rewrite \eqref{move:sharp:1} as
\begin{equation}\label{vel:adve}
(\bv-\bu)\cdot\Grad_{\bx} f=\Delta_{\bx} f-\frac{1}{\eps^2}f(f^2-1).
\end{equation}

Let us consider now the 1-D case. By using the flow map
$\frac{dx(X,t)}{dt}=\bu$ and the  chain rule $\partial_x f=f'_0(X)(\frac{\partial x}{\partial X})^{-1}$, we can derive from  \eqref{vel:adve} in Eulerian coordinate the  trajectory equation in Lagrangian coordinate:
\begin{eqnarray} 
&& (x_t(X,t)-\bv)f'_0(X)(\frac{\partial x}{\partial X})^{-1}=-\partial_X(f'_0(X)(\frac{\partial x}{\partial X})^{-1})(\frac{\partial x}{\partial X})^{-1} + F'(f_0(X)),\label{mv:1}\\
&&x|_{\partial \Omega}=X|_{\partial \Omega}  ~~  \mbox{and} ~~  \quad x(X,0)=X,~~X\in \Omega.\label{GAC_V-Lag}
\end{eqnarray}
Similar to the trajectory equation \eqref{AC_V-Lag} for the Allen-Cahn equation, we can construct first- and second-order schemes for \eqref{GAC_V-Lag} as in the last section. We leave the detail to the interested readers.

\subsection{Two dimensional axis-symmetric case}
We  consider the Allen-Cahn equation  \eqref{G_AC_Var:2} in a two dimensional axis-symmetric domain $\Omega$. To fix the idea, we set $\Omega=\{x^2+y^2< h^2\}$. Using the polar transform 
$x=rcos(\theta), y=rsin(\theta)$, we can rewrite \eqref{G_AC_Var:2} in polar coordinates for the axis-symmetric case as 
\begin{equation}\label{AC-axi}
\begin{split}
&f_t-\frac 1r\partial_r (r\partial_r f)+F'(f)=0; \quad f(h,t)=0\\
&f(r,0)=f_0(r),
\end{split}
\end{equation}
and the associated  flow map \eqref{flow_map_AC}  for the axis-symmetric case as
\begin{equation}\label{map:2D:r}
\begin{split}
&\frac{dr(R,t)}{dt}=u,\\
&r(R,0)=R, 
\end{split}
\end{equation}
where $R$ is the Lagrangian coordinate and $r$ is Eulerian coordinate. Then, the assumed transport equation \eqref{transport} takes the form
\begin{eqnarray}\label{transport2}
&&f_t+u f_r=f_t+\frac{dr(R,t)}{dt}f_r=0,\\
&&f|_{t=0}=f_0(r), 
\end{eqnarray}
which is equivalent to $f(r(R,t),t)=f(r,0)=f_0(R)$ because of flow map  \eqref{map:2D:r} (cf. Remark 2.1). We  then derive from \eqref{transport2} and \eqref{AC-axi} the following force balance equation
\begin{equation}\label{AC-axi2}
\frac{dr(R,t)}{dt}f_r=-\frac 1r\partial_r (r\partial_r f)+ F'(f).
\end{equation}
Using the chain rule $\partial_r f=f_0'(R) (\frac{\partial r}{\partial R})^{-1}$, we arrive at
the trajectory equation in polar coordinate:
\begin{equation}\label{All:polar}
\begin{split}
&r_t(\frac{\partial r}{\partial R})^{-1}\partial_Rf_0(R)=-\frac{1}{r(R)}(\frac{\partial r}{\partial R})^{-1}\partial_R(r(R)(\frac{\partial r}{\partial R})^{-1}\partial_Rf_0(R))+ F'(f_0(R)),\\
&r(h,t)=h,\;r(R,0)=R.
\end{split}
\end{equation}
\begin{thm}\label{dissipative:polar}
The Allen-Cahn equation in Lagrangian coordiante  \eqref{All:polar}  satisfies the following energy
dissipative law
\begin{equation}
\begin{split}
&\frac{d}{dt}\int_{\Omega_R}\{\frac 12 |(\frac{\partial r}{\partial R})^{-1}\partial_Rf_0(R)|^2
+F(f_0(R))\}det\frac{\partial r}{\partial R} RdR\\&=-\int_{\Omega_R}|r_t\partial_Rf_0(R))(\frac{\partial r}{\partial R})^{-1}|^2 det\frac{\partial r}{\partial R} RdR.
\end{split}
\end{equation}

\end{thm}
\begin{proof}
Taking inner product of equation \eqref{All:polar} with $-rr_t\partial_Rf_0(R)$,  we obtain
\begin{equation}
\begin{split}
&(r_t(\frac{\partial r}{\partial R})^{-1}\partial_Rf_0(R),-rr_t\partial_Rf_0(R))=(-\frac{1}{r(R)}(\frac{\partial r}{\partial R})^{-1}\partial_R(r(R)(\frac{\partial r}{\partial R})^{-1}\partial_Rf_0(R)) , -rr_t\partial_Rf_0(R))\\& \hskip 4cm +( F'(f_0(R)),-rr_t\partial_Rf_0(R)).
\end{split}
\end{equation}
Notice  that $\int_{\Omega_r}rdr=\int_{\Omega_R}det\frac{\partial r}{\partial R}RdR=\int_{\Omega_R}rr'(R)dR$, we derive the equality
\begin{equation}
det\frac{\partial r}{\partial R}=\frac{r\frac{\partial r}{\partial R}}{R}.
\end{equation}
We obtain
\begin{equation}\label{polar:term0}
\begin{split}
&(r_t(\frac{\partial r}{\partial R})^{-1}\partial_Rf_0(R),-rr_t\partial_Rf_0(R))=-\int_{\Omega_R}|r_t\partial_Rf_0(R))(\frac{\partial r}{\partial R})^{-1}|^2r\frac{\partial r}{\partial R} dR\\
&=-\int_{\Omega_R}|r_t\partial_Rf_0(R))(\frac{\partial r}{\partial R})^{-1}|^2 det\frac{\partial r}{\partial R} RdR.
\end{split}
\end{equation}
Taking integration by part, we derive
\begin{equation}\label{polar:term1}
\begin{split}
&( F'(f_0(R)),-rr_t\partial_Rf_0(R))=(\partial_R F(f_0(R)),-rr_t )=(F(f_0(R)), \frac{\partial r}{\partial R}r_t+r\frac{\partial r_t}{\partial R})\\&
=\frac{d}{dt}\int_{\Omega_R}F(f_0(R)) r\frac{\partial r}{\partial R}dR=\frac{d}{dt}\int_{\Omega_R}F(f_0(R)) det\frac{\partial r}{\partial R} RdR.
\end{split}
\end{equation}
We consider  
\begin{equation}\label{polar:term2}
\begin{split}
&(-\frac{1}{r(R)}(\frac{\partial r}{\partial R})^{-1}\partial_R(r(R)(\frac{\partial r}{\partial R})^{-1}\partial_Rf_0(R)) , -rr_t\partial_Rf_0(R))=
(\frac 12\partial_R|r(R)(\frac{\partial r}{\partial R})^{-1}\partial_Rf_0(R)|^2,\frac{r_t}{r})
\\&=-(\frac 12|r(\frac{\partial r}{\partial R})^{-1}\partial_Rf_0(R)|^2,\partial_R(\frac{r_t}{r}))=
-(\frac 12|(\frac{\partial r}{\partial R})^{-1}\partial_Rf_0(R)|^2,rr_{tR}-r_t\frac{\partial r}{\partial R})\\&=
\frac{d}{dt}\int_{\Omega_R}\frac 12 |\partial_Rf_0(R)|^2r(\frac{\partial r}{\partial R})^{-1}dR=
\frac{d}{dt}\int_{\Omega_R}\frac 12 |(\frac{\partial r}{\partial R})^{-1}\partial_Rf_0(R)|^2r\frac{\partial r}{\partial R}dR\\&=
\frac{d}{dt}\int_{\Omega_R}\frac 12 |(\frac{\partial r}{\partial R})^{-1}\partial_Rf_0(R)|^2det\frac{\partial r}{\partial R} RdR.
\end{split}
\end{equation}
Finally, combining equations \eqref{polar:term0}-\eqref{polar:term2}, we obtain the energy
dissipative law.

\end{proof}

\begin{remark}
Similar with Remark \ref{energy:coor},  the energy dissipative law in Theorem \ref{dissipative:polar} is equivalent with energy dissipative law in Eulerian coordinate by using the chain rule $\partial_r f=f_0'(R) (\frac{\partial r}{\partial R})^{-1}$,
\begin{eqnarray}
\frac{d}{dt}\int_{\Omega_{r}}\{\frac 12|\partial_r f|^2+\frac{1}{4\eps^2}(f^2-1)^2\}rdr
=-\int_{\Omega_{r}}|uf_r|^2rdr.
\end{eqnarray}
\end{remark}

Similarly, we can construct first and  second schemes for the above equation. For example, a  first-order  scheme   for  \eqref{All:polar} is as follows:
\begin{equation}\label{All:polar:BDF1}
\begin{split}
&\frac{r^{n+1}-r^n}{\delta t}(\frac{\partial r^{n}}{\partial R})^{-1}\partial_Rf_0(R)=-(\frac{\partial r^{n+1}}{\partial R})^{-1}\partial_R((\frac{\partial r^{n+1}}{\partial R})^{-1}\partial_Rf_0(R))\\&-\frac{1}{r^{n+1}(R)}(\frac{\partial r^{n+1}}{\partial R})^{-1}\partial_Rf_0(R) + F'(f_0(R)),\\
&r^{n+1}|_{r=h}=h,\,r(R,0)=R.
\end{split}
\end{equation}

\section{Numerical experiments}
In this section, we  present some numerical tests to show the efficiency, stability and accuracy of the numerical schemes  \eqref{fem:AC:num:1}-\eqref{fem:AC:num:2} and its second-order version for the Allen-Cahn equation \eqref{G_AC_Var:1}-\eqref{G_AC_Var:2} with $F(f)=\frac1{4\epsilon^2}(f^2-1)^2$. In the following, we set $\Omega_x=\Omega_X=(-1,1)$ and use, as spatial discretization, the Legendre-Galerkin method \cite{Shen94b} and the  piecewise linear finite-element method.

\subsection{Accuracy test}

 We first perform an accuracy test.  We used Legendre-Galerkin method in space so that the spatial error is negligible compared with the temporal error.
 We start with a smooth initial condition $f_0(x)=x$ and using solution computed by  the second-order scheme  with  $\delta t=10^{-5}$   as the reference solution.
   In Fig.\,\ref{order} we plot the $L^{\infty}$ error between numerical solution and reference solution at time $t=0.1$. We observe that  the first-order scheme BDF1  achieves first-order convergence while  the second-order scheme BDF2 achieves second-order convergence.

\begin{figure}
\centering
\includegraphics[width=0.55\textwidth,clip==]{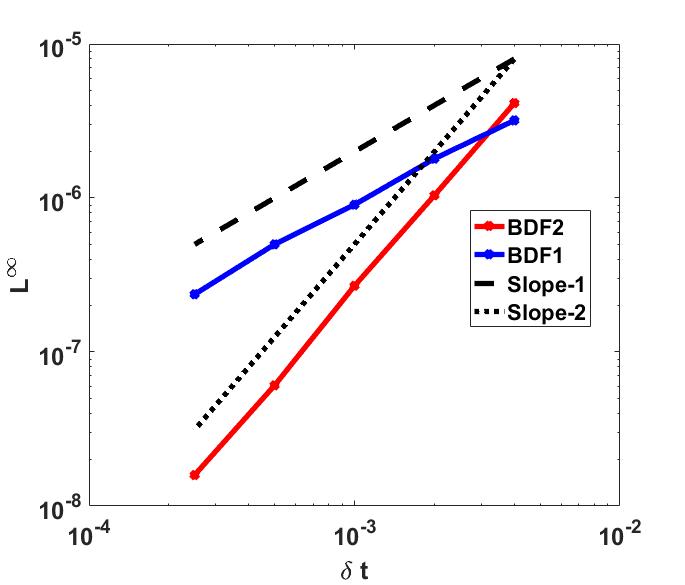}
\caption{Accuracy test for the Allen-Cahn equation \eqref{G_AC_Var:1}-\eqref{G_AC_Var:2}.}\label{order}
\end{figure}

\subsection{Interface capturing}
 
We now present  numerical simulations to demonstrate the effectiveness of our new Lagrangian approach for interface capturing.  In Fig.\,\ref{interface_0}  we choose  interface width parameter as $\eps^2=0.001$ and initial condition as $f_0(x)=1-x^2$.   We depict  profiles of interface at various time in Fig.\,\ref{interface_0}.(a) and in Fig.\,\ref{interface_0}.(b) using the second-order new Lagrangian scheme  with spectral method and finite element method in space, and in Fig.\,\ref{interface_0}.(c) using the second-order semi-implicit  method in Eulerian coordinate  with spectral method in space. We   observe that the profiles of interface can be well captured  with mesh resolution of $N=64$ by the Lagrangian  method, as compared with $N=256$ by the Eulerian method.  We also plot in Fig.\,\ref{interface_0}.(d), the mesh distribution of the Lagrangian method in Eulerian coordinate. We observe that as interface getting steeper,  more points will move closer to the  interface area.

\begin{figure}
\centering
\subfigure[Flow dynamic approach with Spectral method:$N=64$ and $\eps^2=0.001$.]{
\includegraphics[width=0.45\textwidth,clip==]{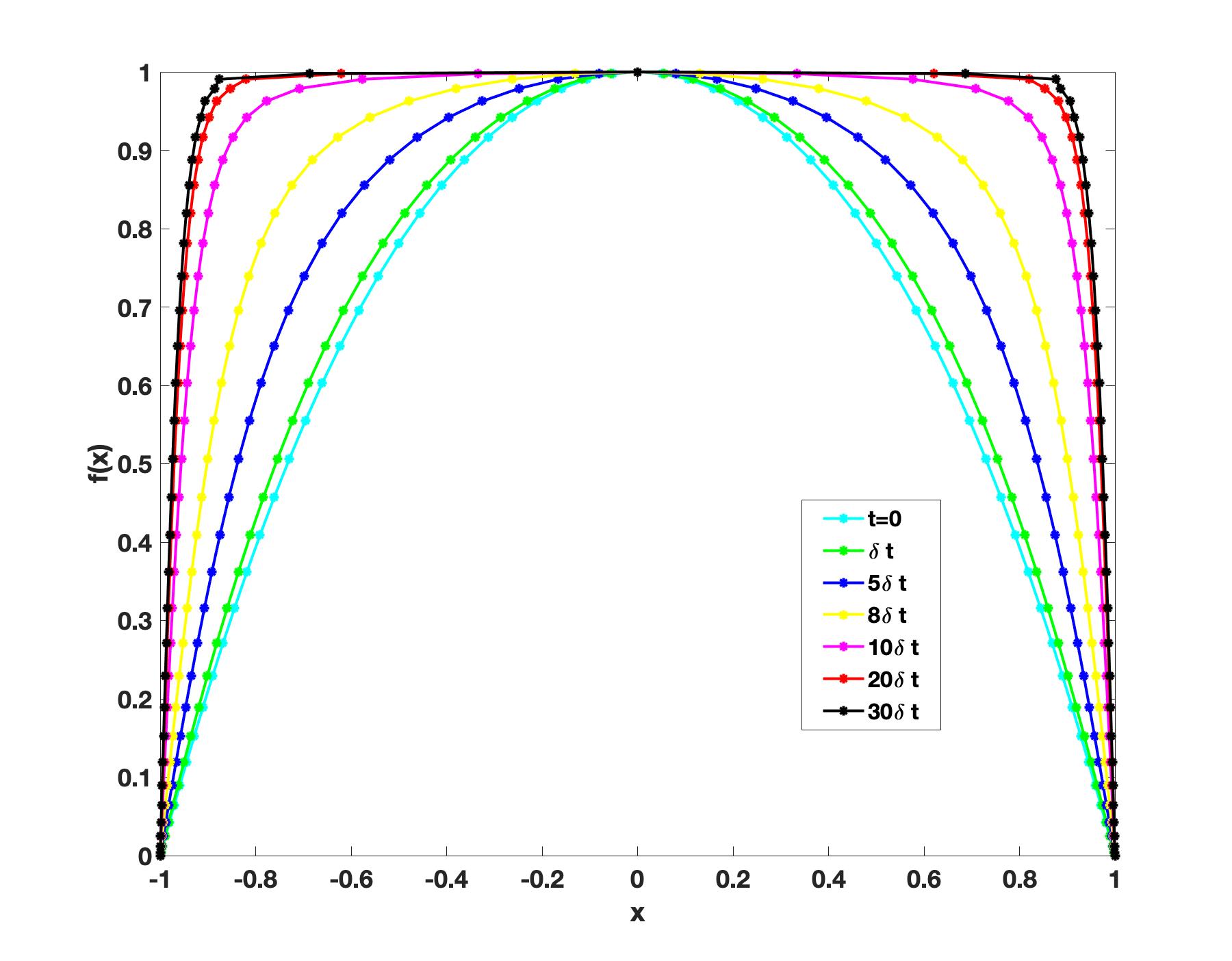}}
\subfigure[Flow dynamic approach with Finite element method: $N=64$ and $\eps^2=0.001$.]{
\includegraphics[width=0.45\textwidth,clip==]{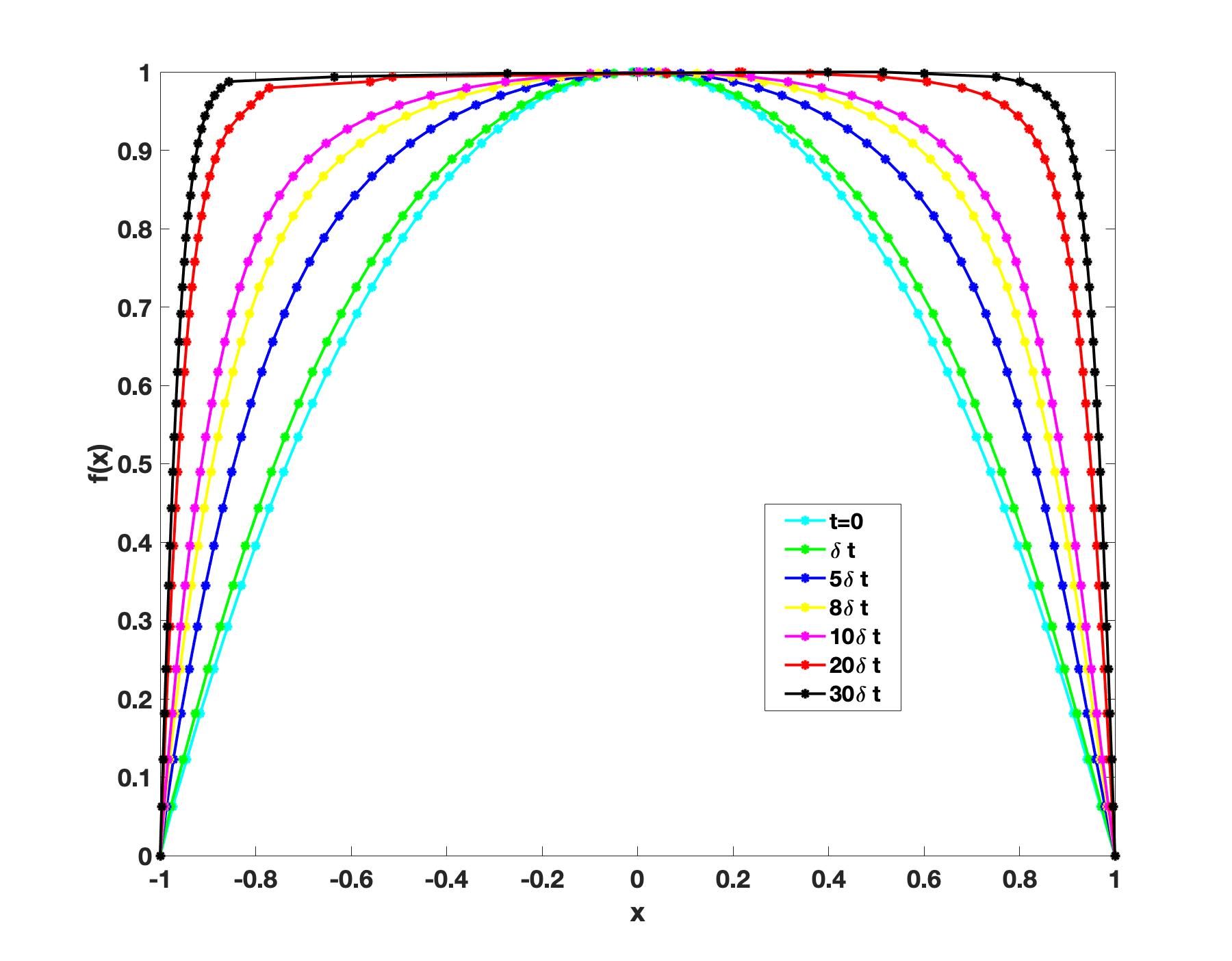}}
\subfigure[Numerical method in Eulerian coordinate $N=256$ and $\eps^2=0.001$.]{
\includegraphics[width=0.45\textwidth,clip==]{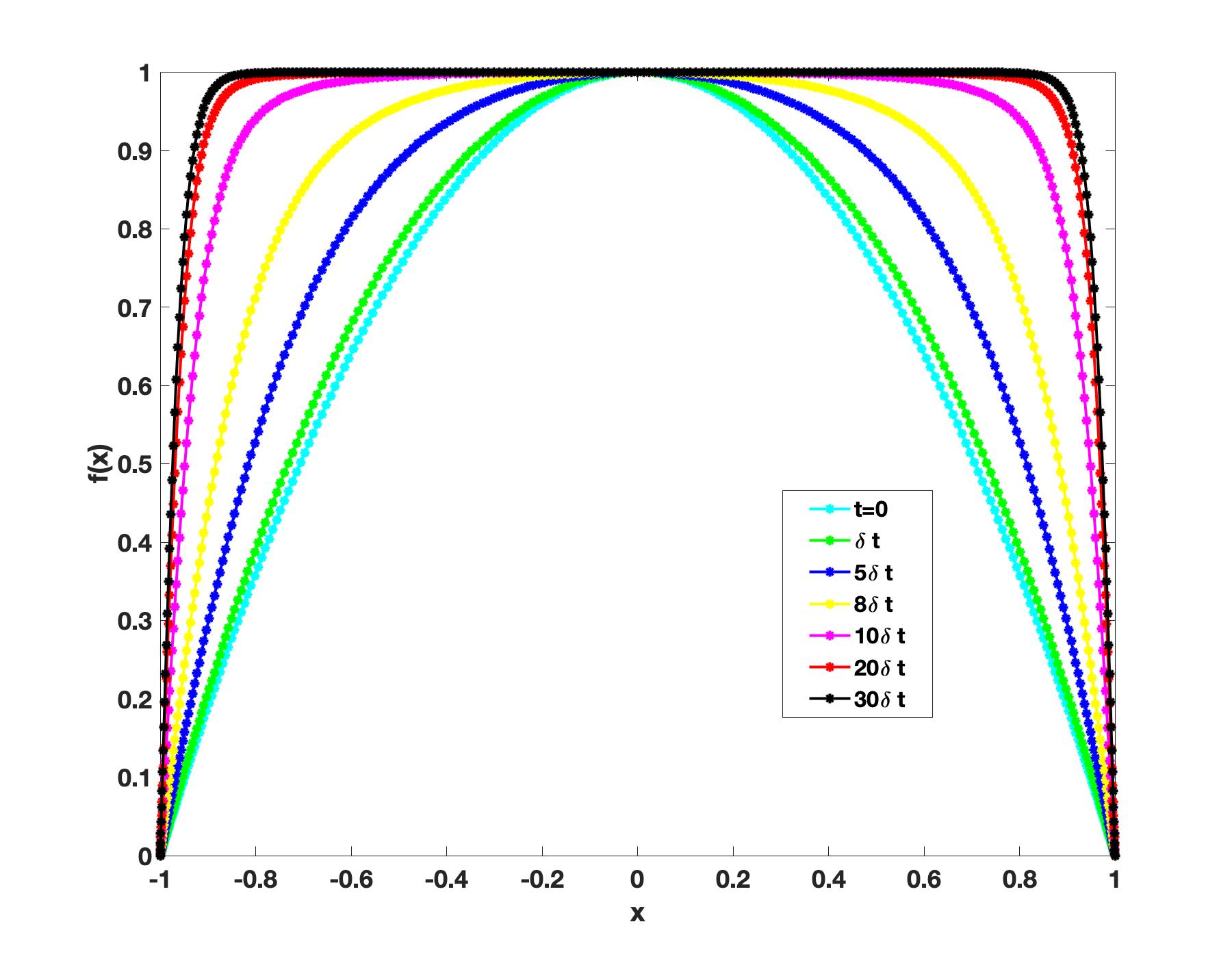}}
\subfigure[Mesh distribution for Flow dynamic approach.]{
\includegraphics[width=0.45\textwidth,clip==]{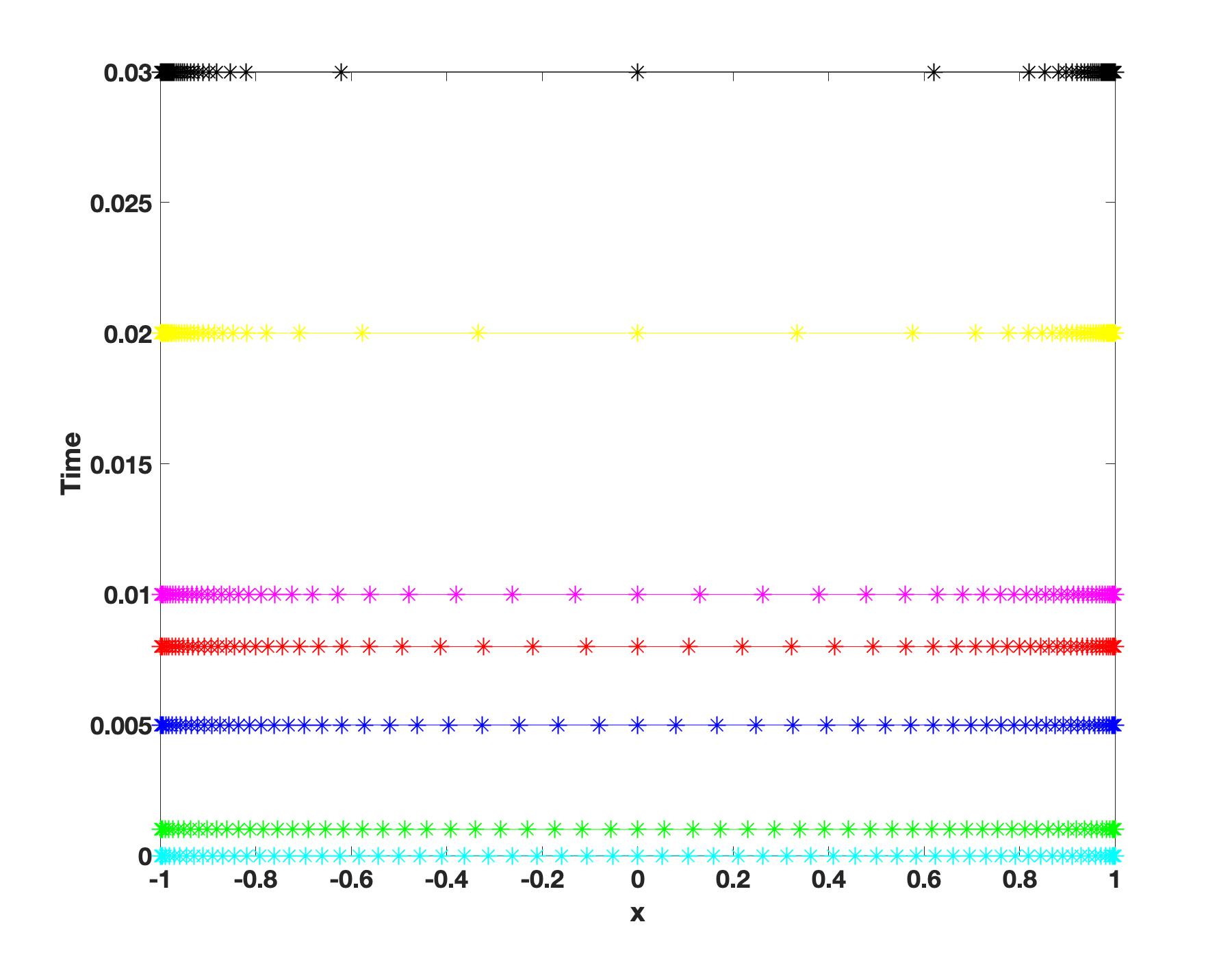}}
\caption{Capturing interface by  Lagrangian numerical method based on Variational Energetic Approach.}\label{interface_0}
\end{figure}

Next we examine what happens as we decrease the interfacial width.  It is expected that the solution,  in the limit of $\eps$ going to zero, behaves like a piecewise constant function with values $\pm 1$ in much of two bulk regions which are separated by a diffusive interfacial layer of thickness $O(\eps)$. 

We first use the Lagrangian scheme with the finite-element method in space.  In Fig.\,\ref{fem_steady}, we plot the results  for  $\eps^2=10^{-3}$ to $\eps^2=10^{-6}$ with $N=8,16,32,64$ points and initial condition is $f_0(x)=x$.  We observe that  almost all  points are concentrated at the interfacial region. The interface location is well captured even with only 8 points, although the value is a bit off due to the limited accuracy of finite-elements. We obtain similar results as we decrease $\eps$ further. This example shows the amazing ability of the flow dynamic approach in capturing thin interfaces of Allen-Cahn equations: the number of points needed to resolve the interface is independent of interfacial width!

\begin{figure}
\centering
\subfigure[ $\eps^2=10^{-3}$  and $N=8$.]{
\includegraphics[width=0.23\textwidth,clip==]{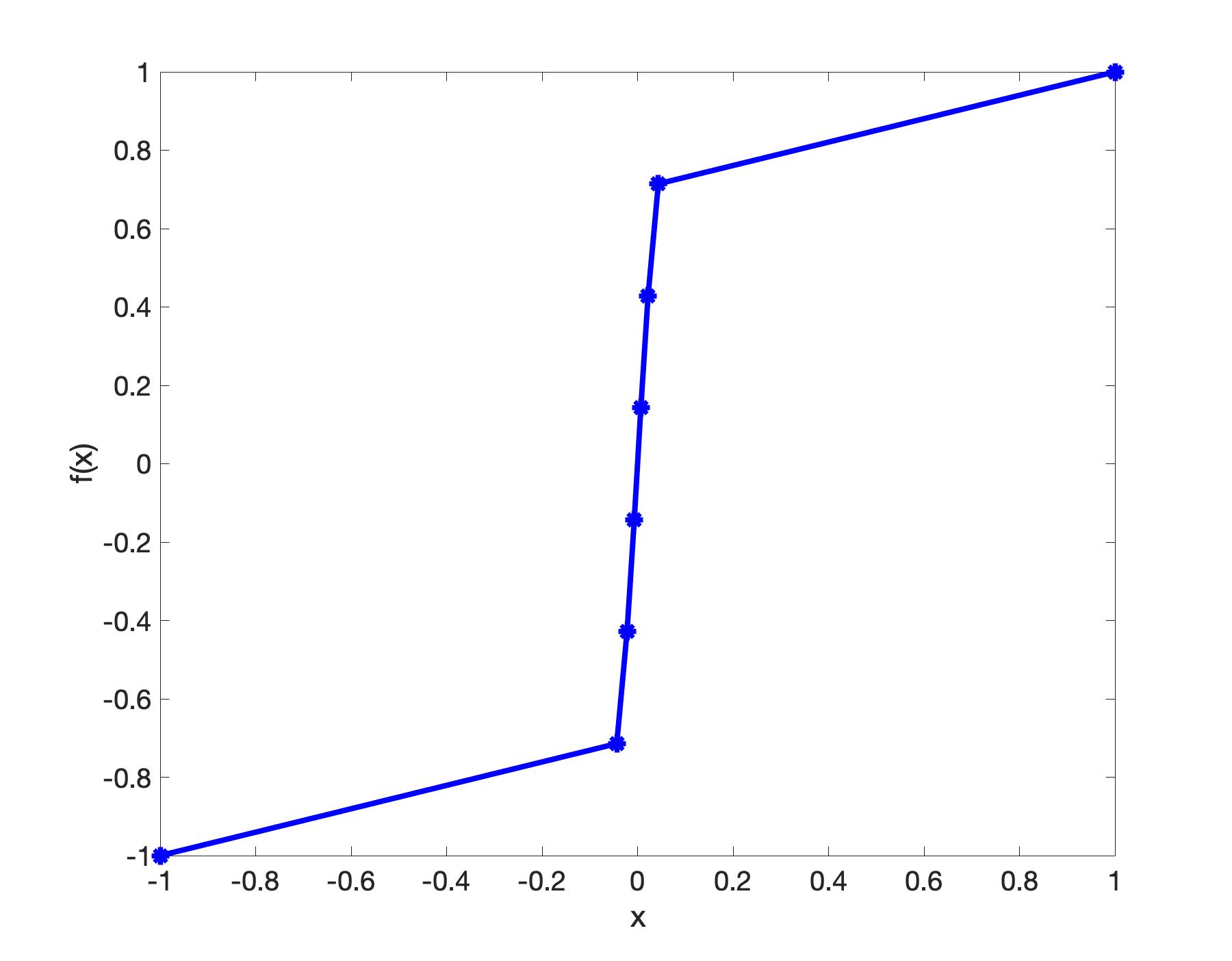}}
\subfigure[$\eps^2=10^{-3}$ and $N=16$ .]{
\includegraphics[width=0.23\textwidth,clip==]{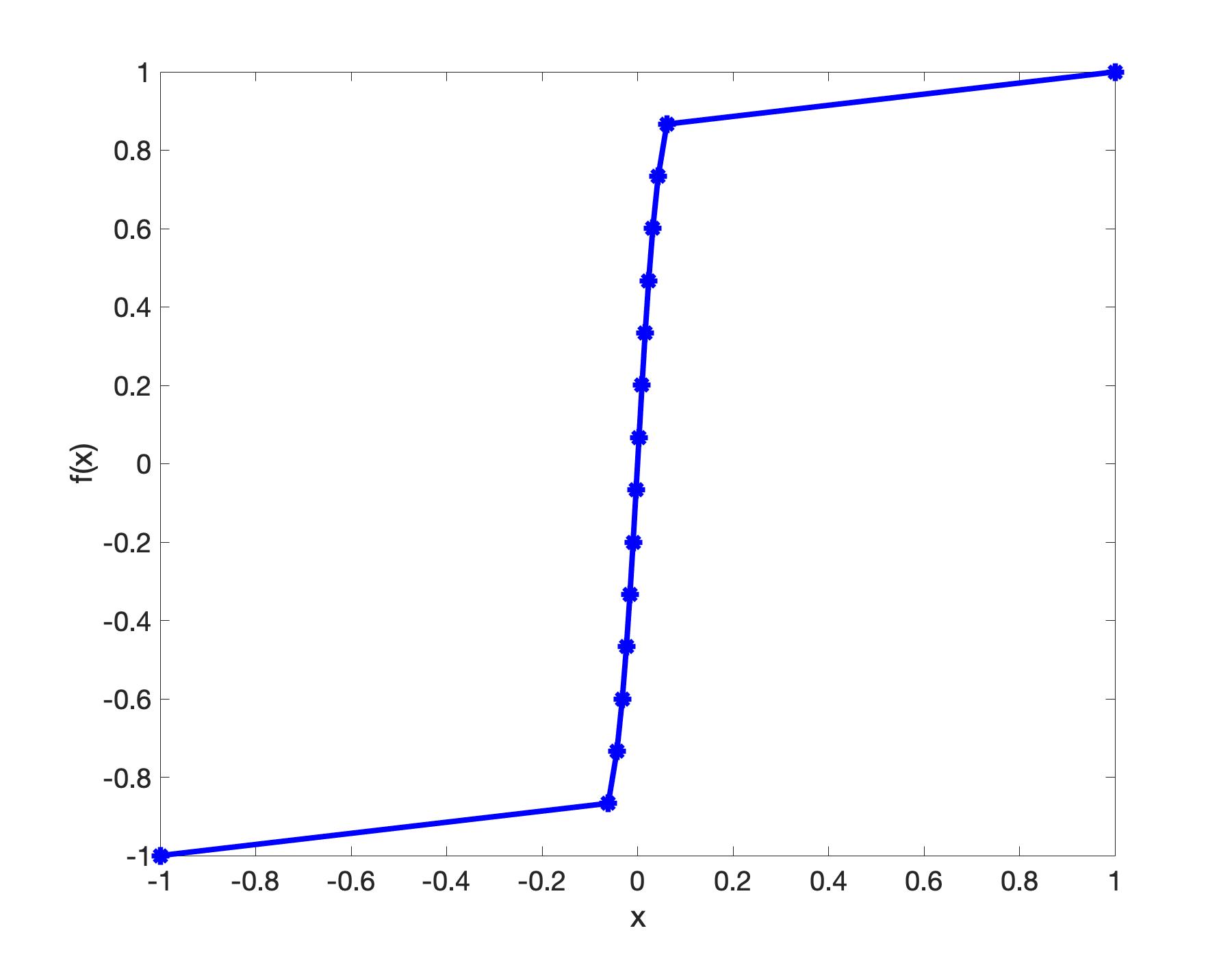}}
\subfigure[$\eps^2=10^{-3}$ and $N=32$]{
\includegraphics[width=0.23\textwidth,clip==]{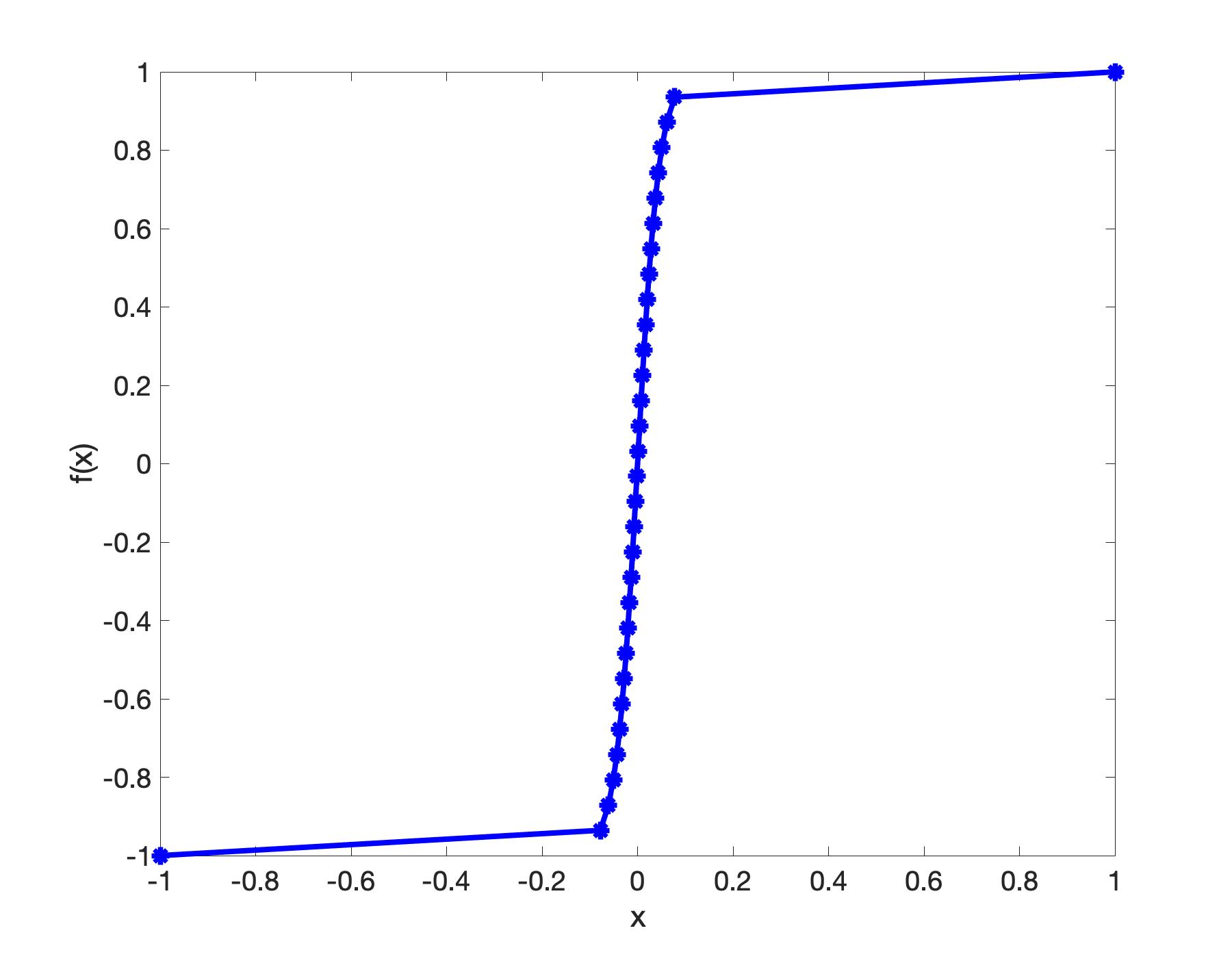}}
\subfigure[$\eps^2=10^{-3}$ and $N=64$]{
\includegraphics[width=0.23\textwidth,clip==]{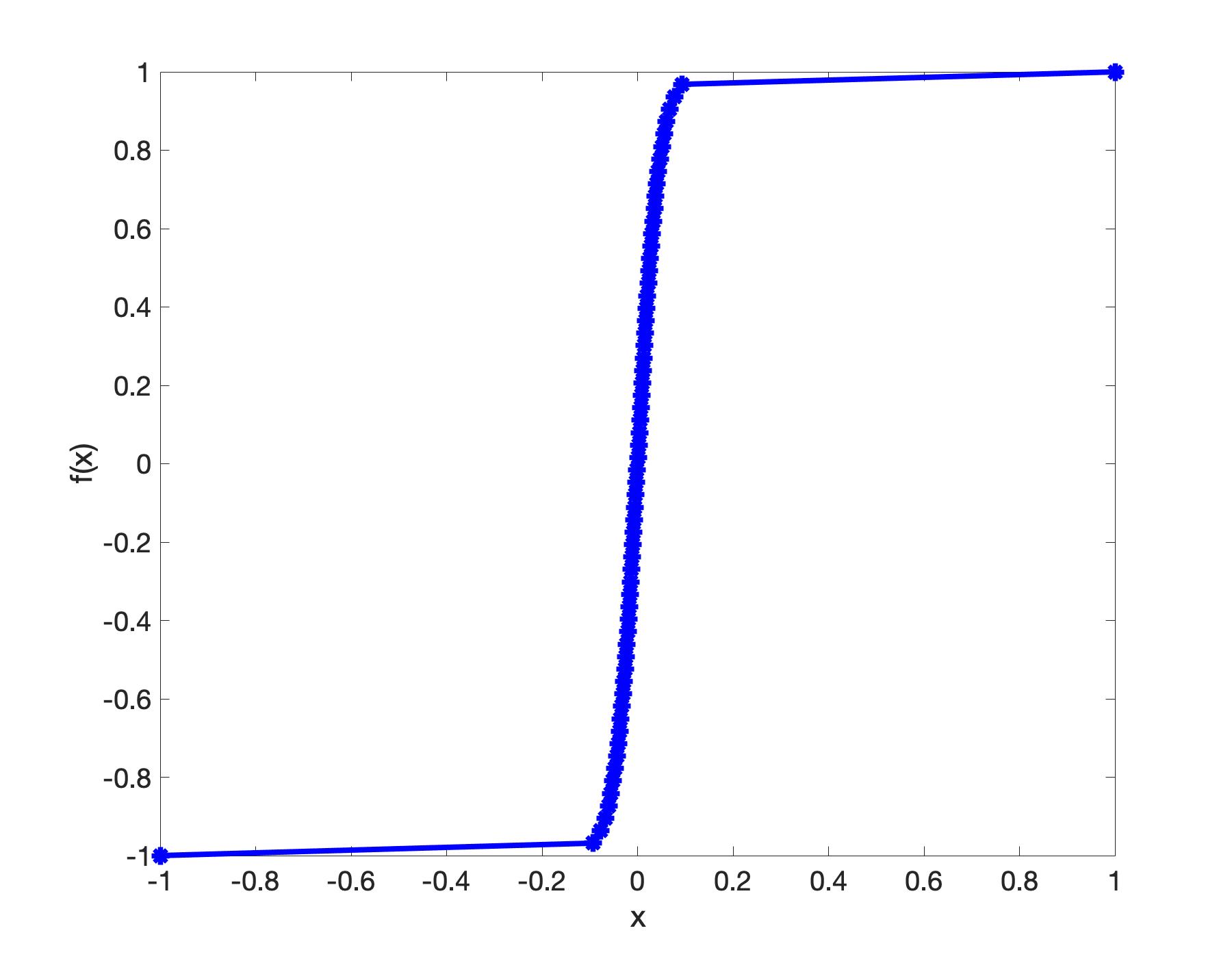}}
\subfigure[ $\eps^2=10^{-4}$  and $N=8$.]{
\includegraphics[width=0.23\textwidth,clip==]{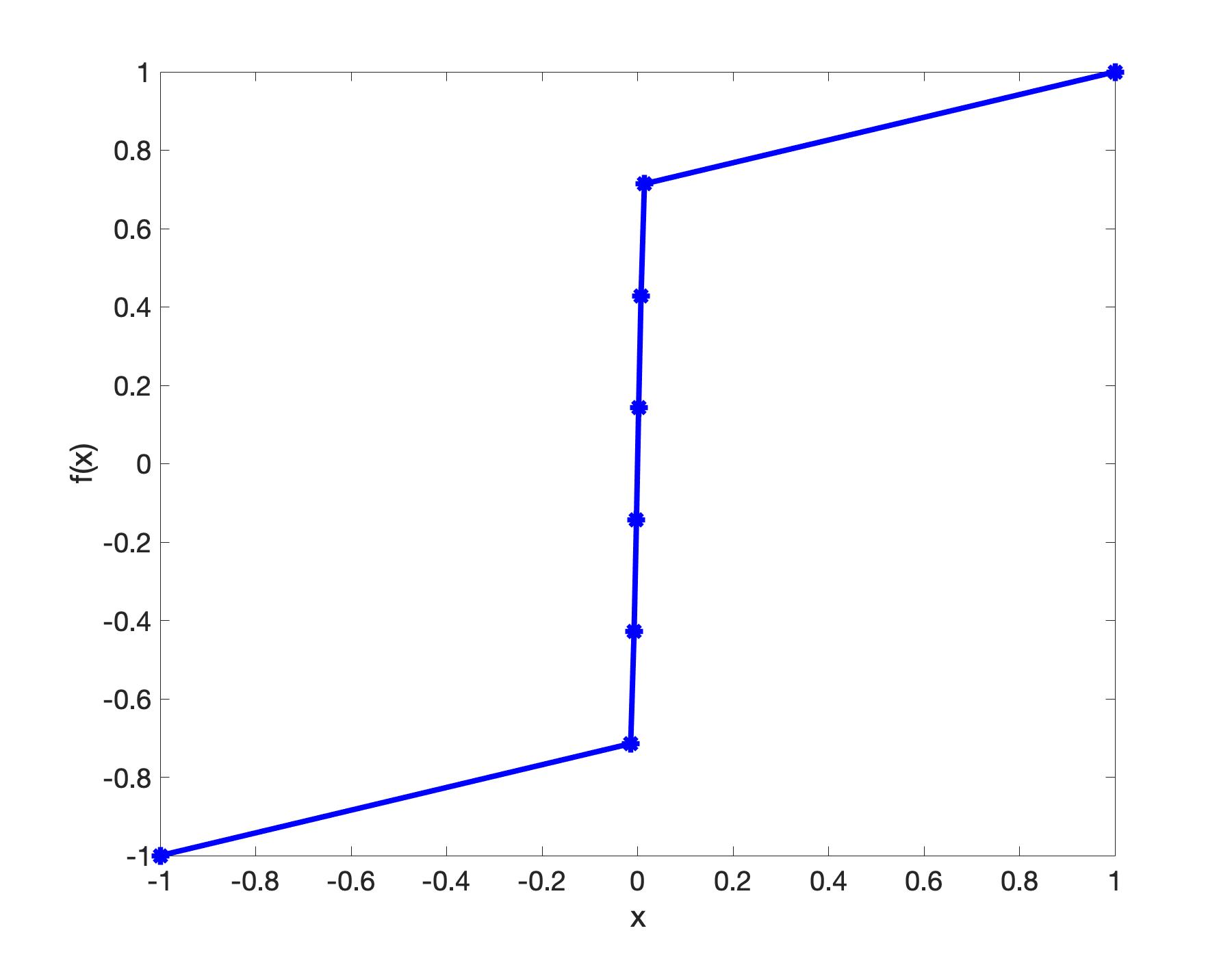}}
\subfigure[$\eps^2=10^{-4}$ and $N=16$ .]{
\includegraphics[width=0.23\textwidth,clip==]{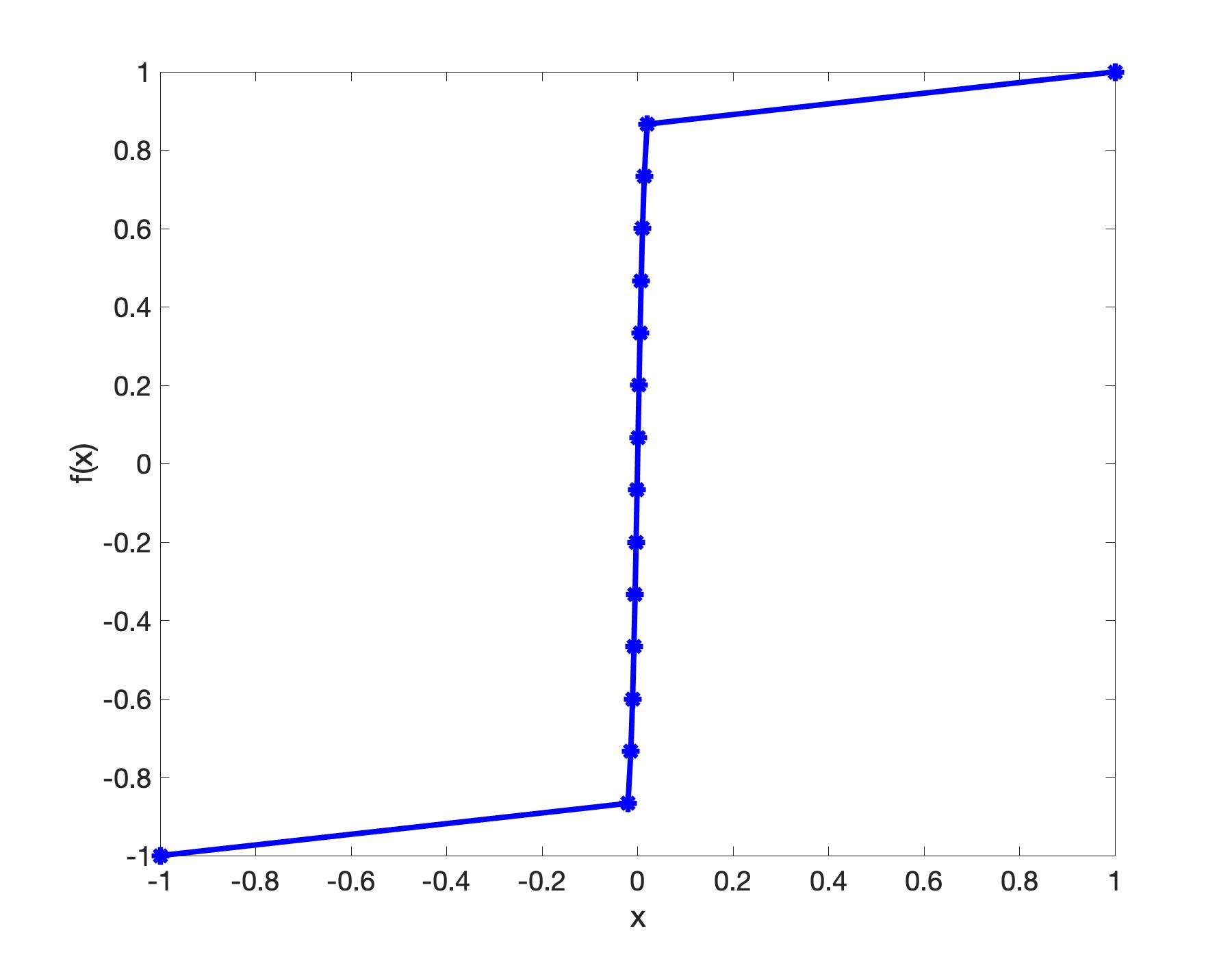}}
\subfigure[$\eps^2=10^{-4}$ and $N=32$]{
\includegraphics[width=0.23\textwidth,clip==]{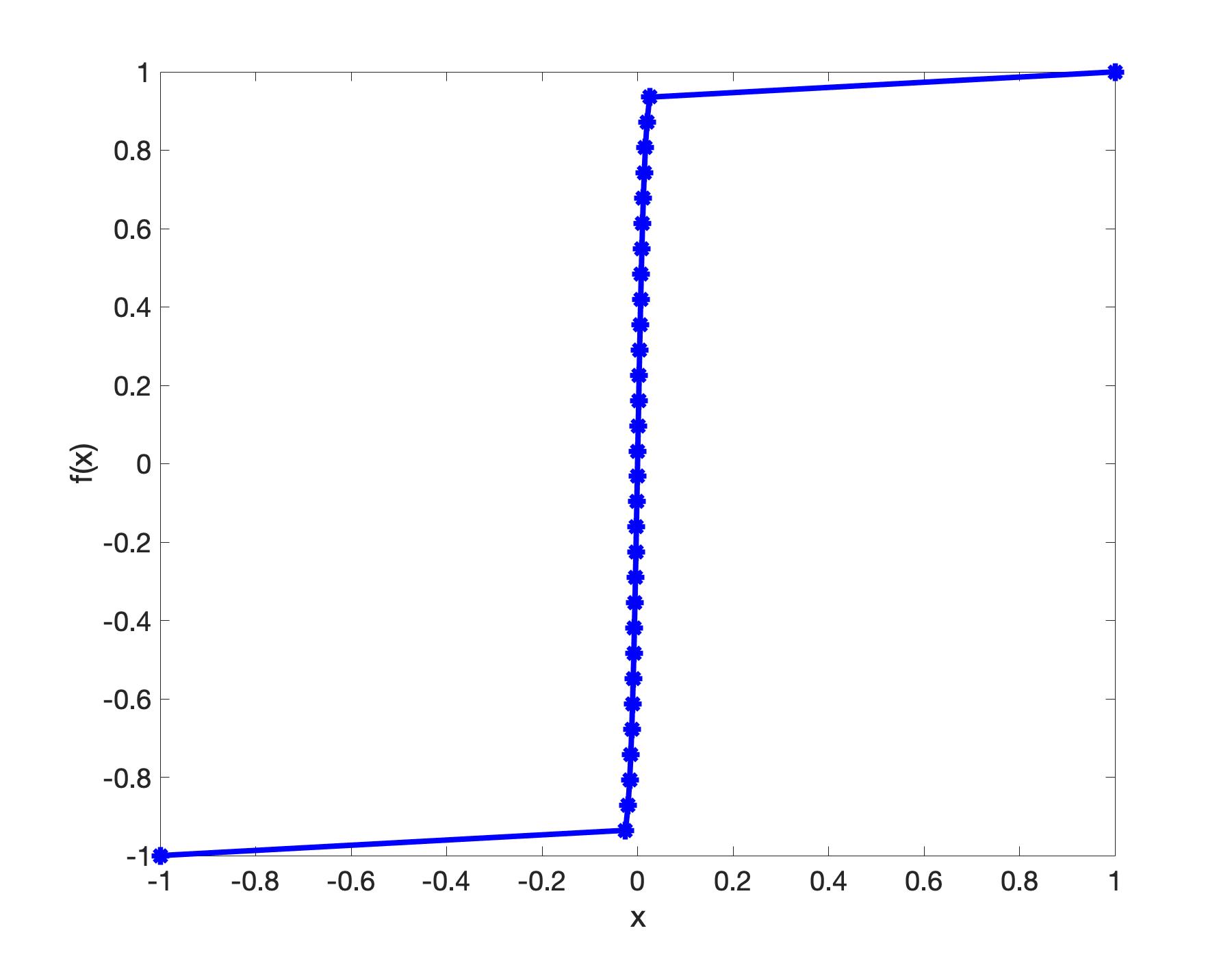}}
\subfigure[$\eps^2=10^{-4}$ and $N=64$]{
\includegraphics[width=0.23\textwidth,clip==]{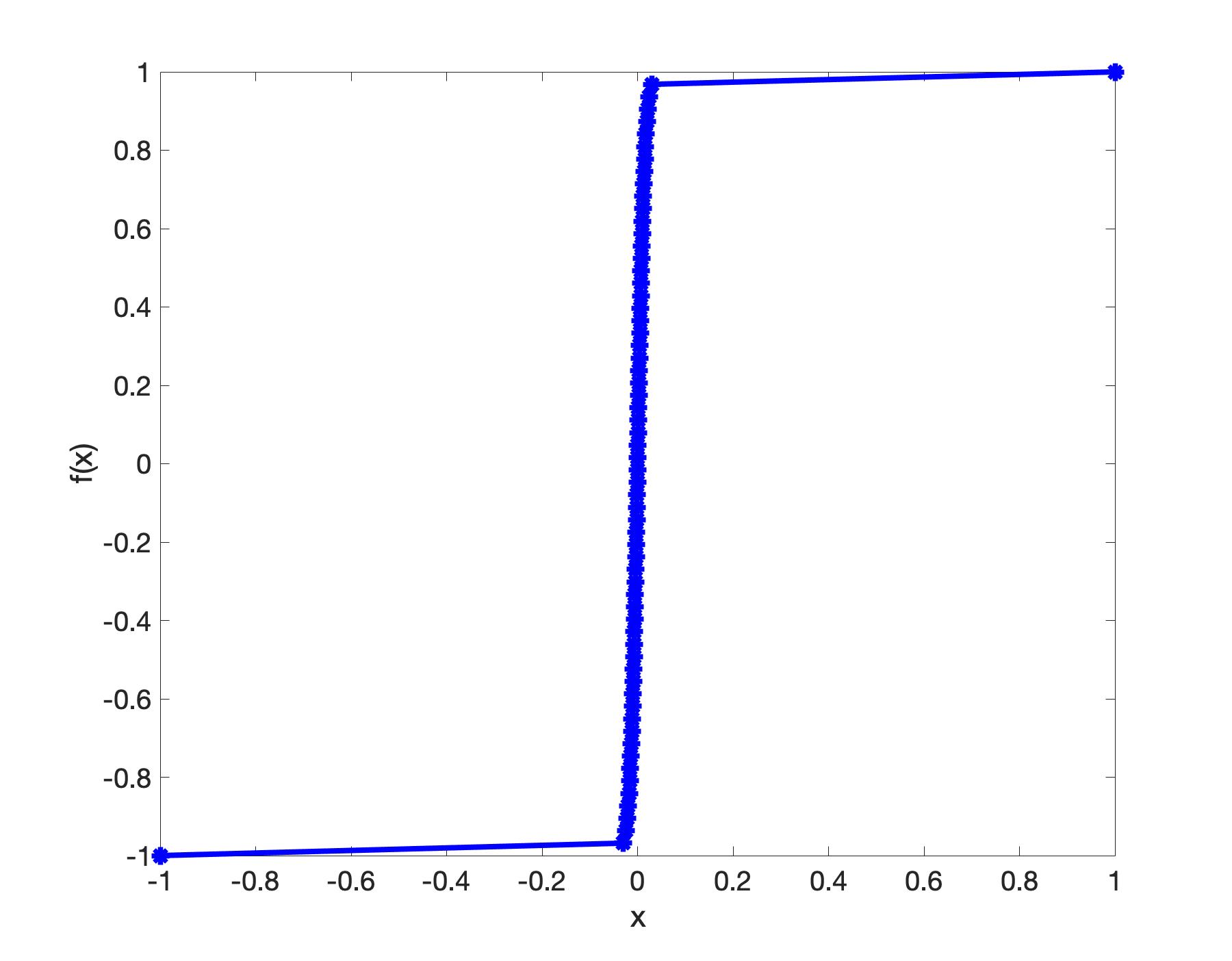}}
\subfigure[ $\eps^2=10^{-5}$  and $N=8$.]{
\includegraphics[width=0.23\textwidth,clip==]{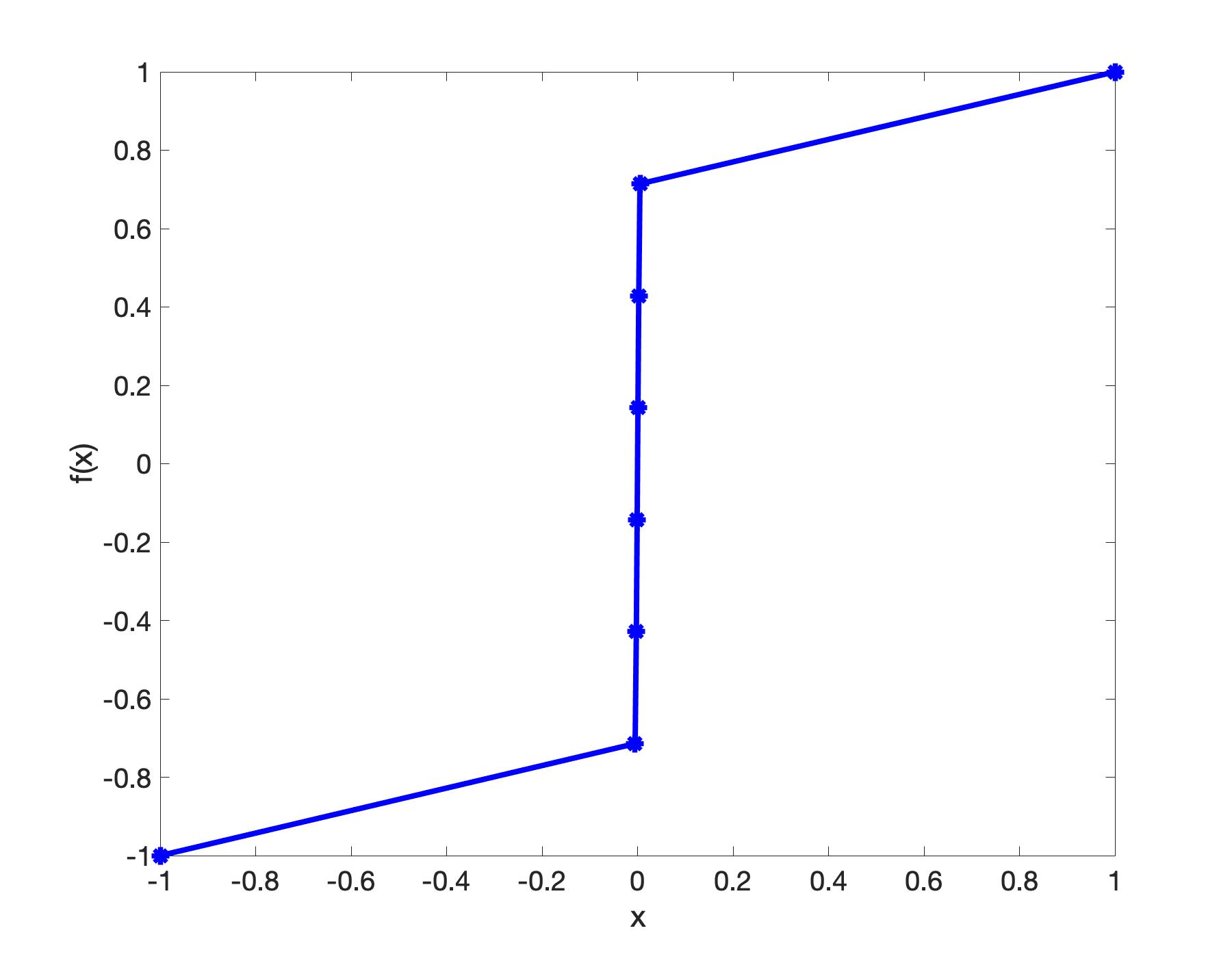}}
\subfigure[$\eps^2=10^{-5}$ and $N=16$ .]{
\includegraphics[width=0.23\textwidth,clip==]{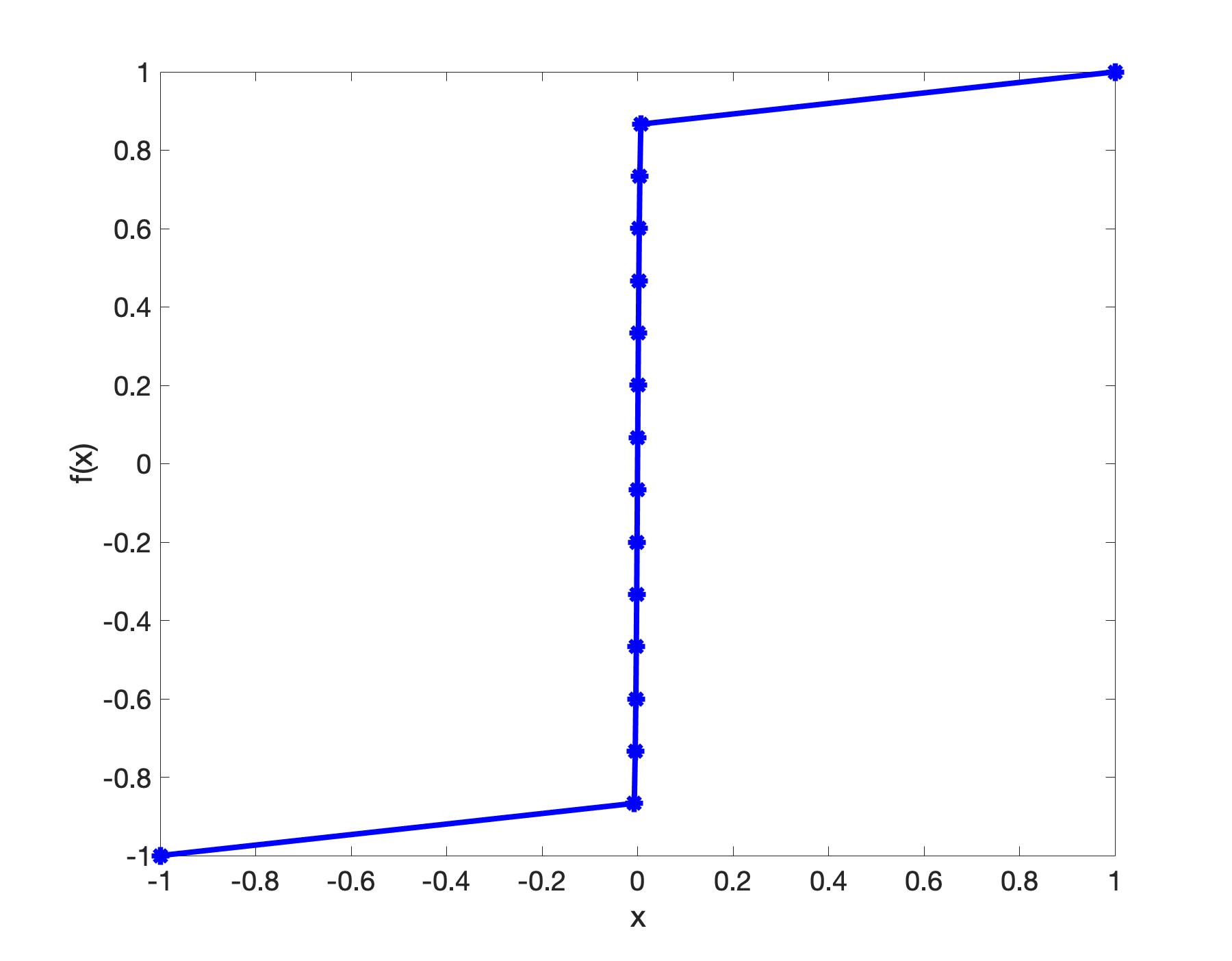}}
\subfigure[$\eps^2=10^{-5}$ and $N=32$]{
\includegraphics[width=0.23\textwidth,clip==]{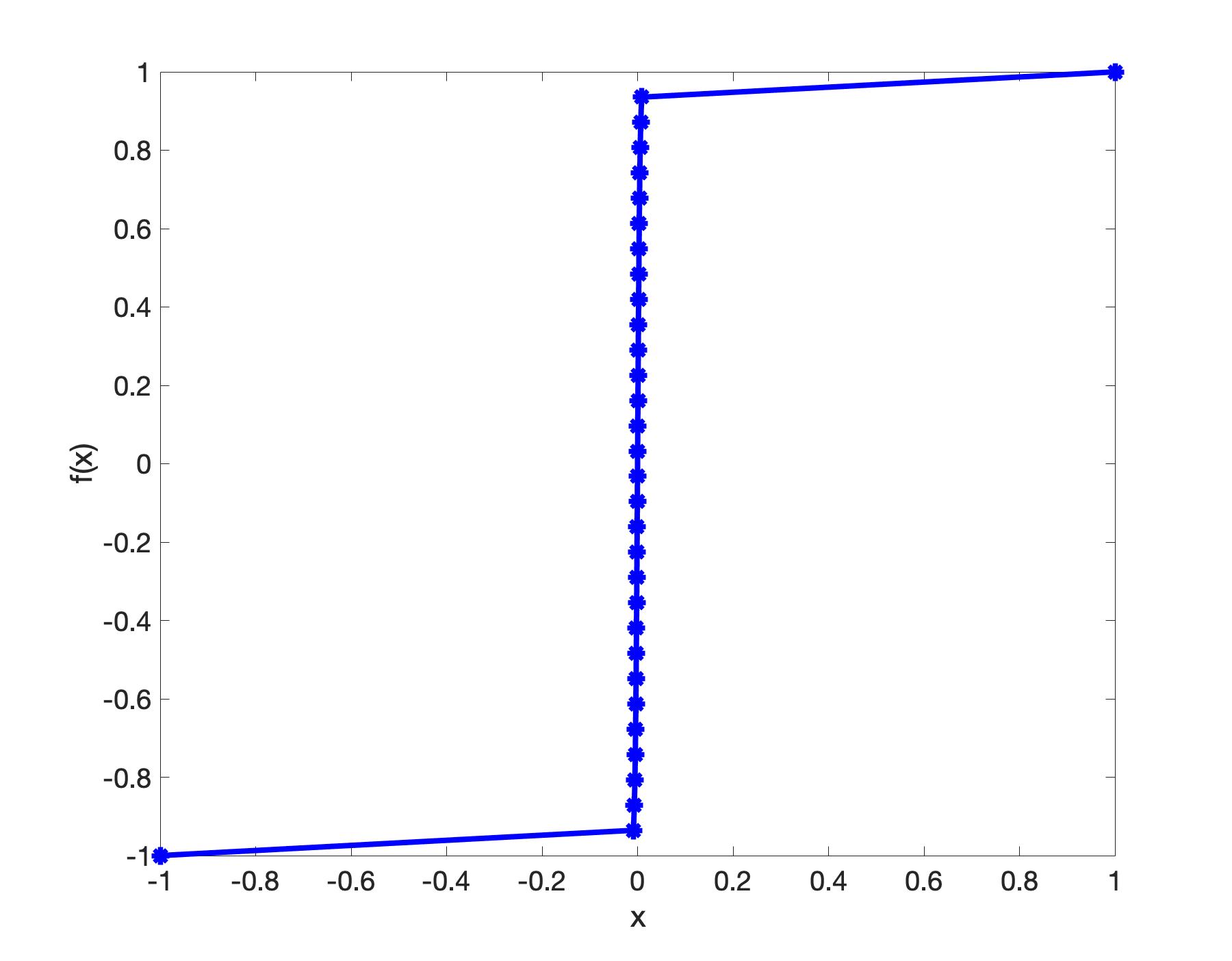}}
\subfigure[$\eps^2=10^{-5}$ and $N=64$]{
\includegraphics[width=0.23\textwidth,clip==]{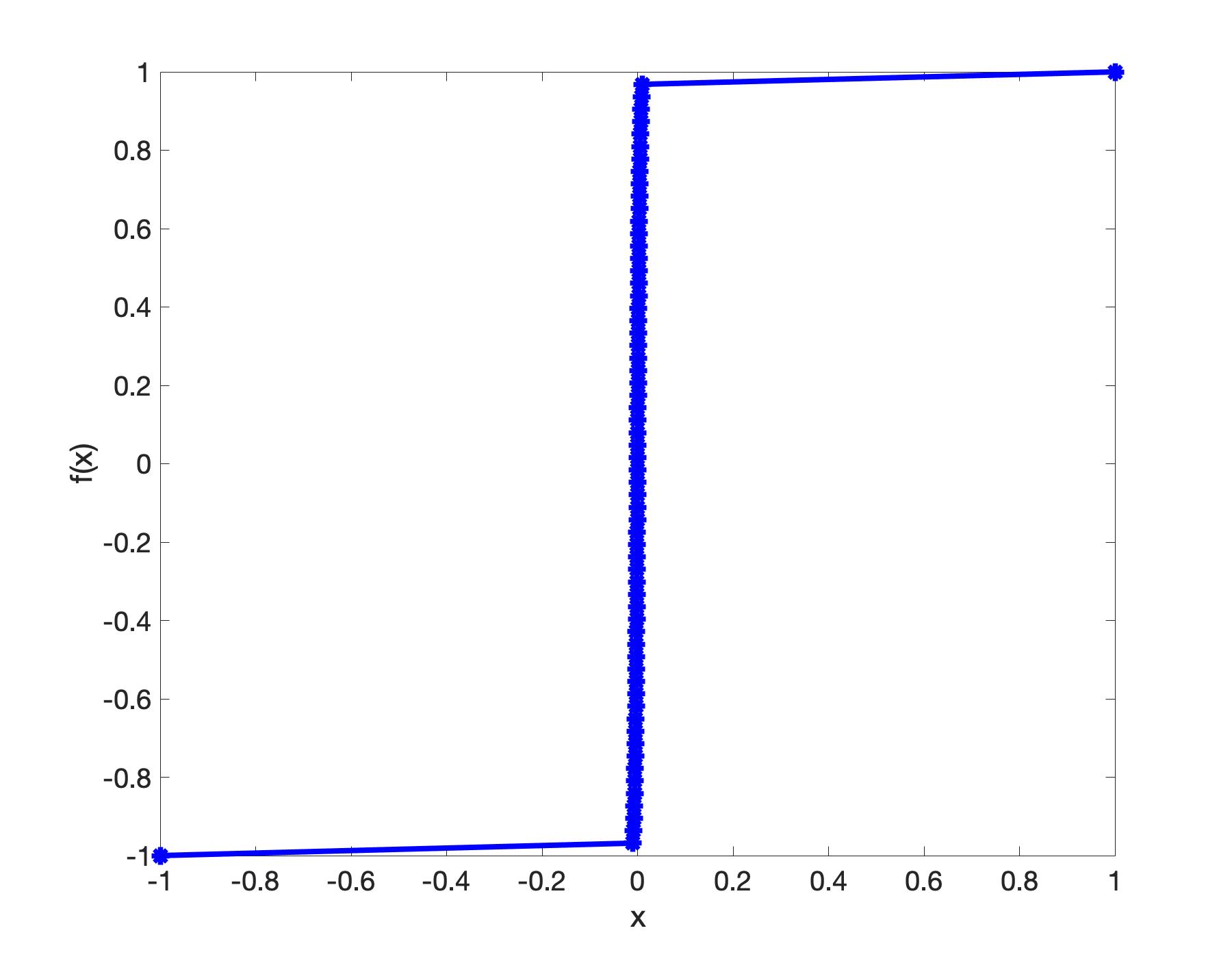}}
\subfigure[ $\eps^2=10^{-6}$  and $N=8$.]{
\includegraphics[width=0.23\textwidth,clip==]{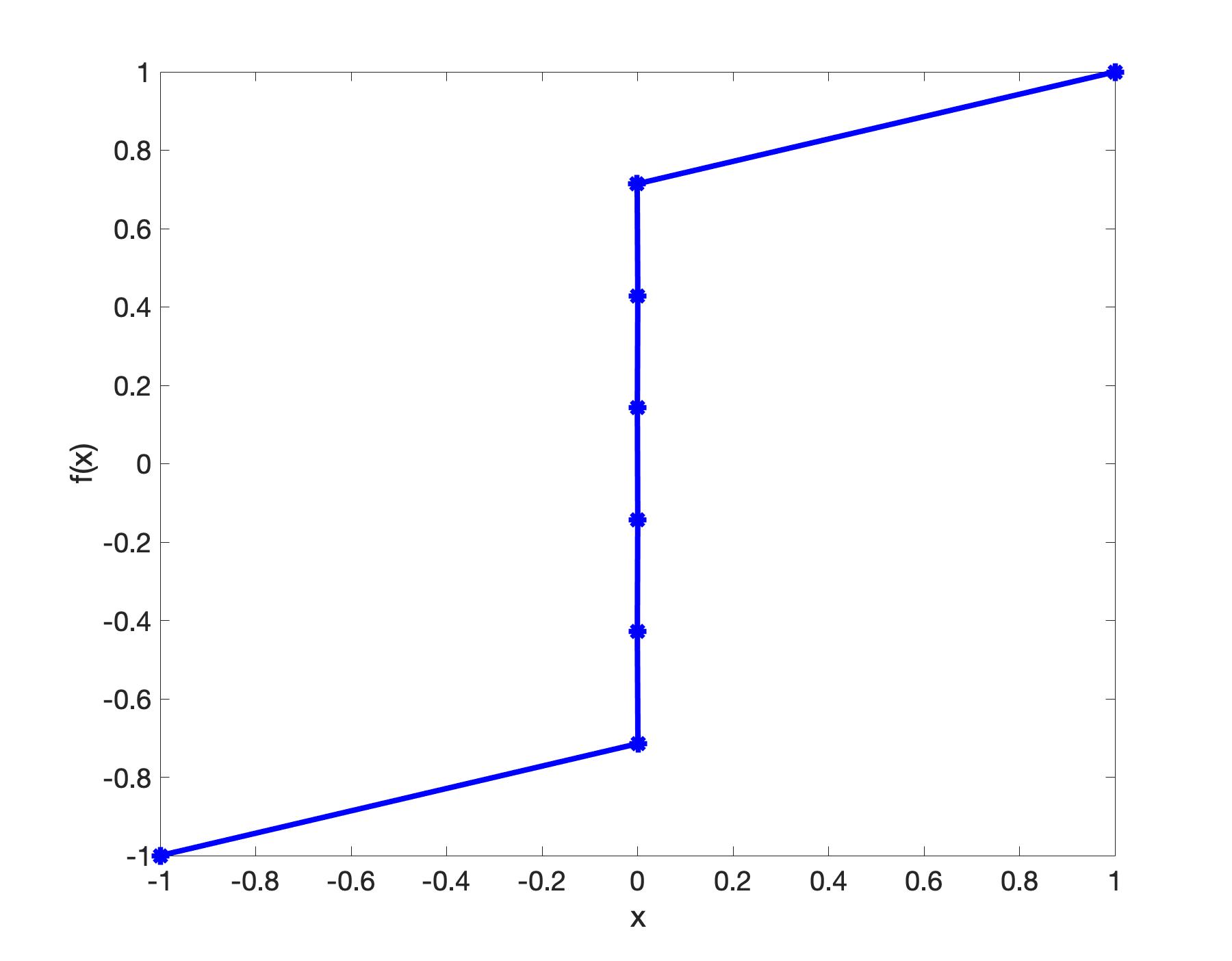}}
\subfigure[$\eps^2=10^{-6}$ and $N=16$ .]{
\includegraphics[width=0.23\textwidth,clip==]{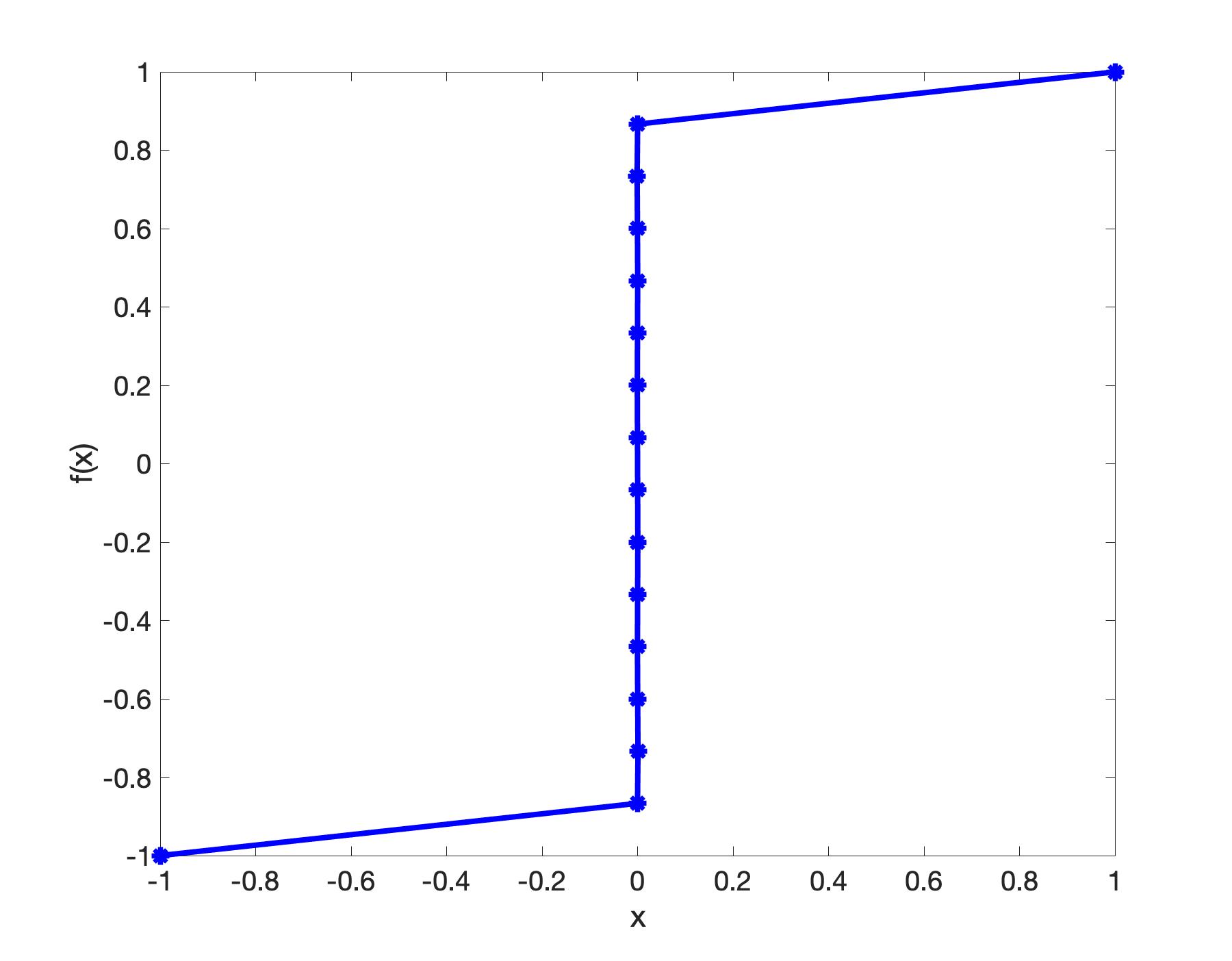}}
\subfigure[$\eps^2=10^{-6}$ and $N=32$]{
\includegraphics[width=0.23\textwidth,clip==]{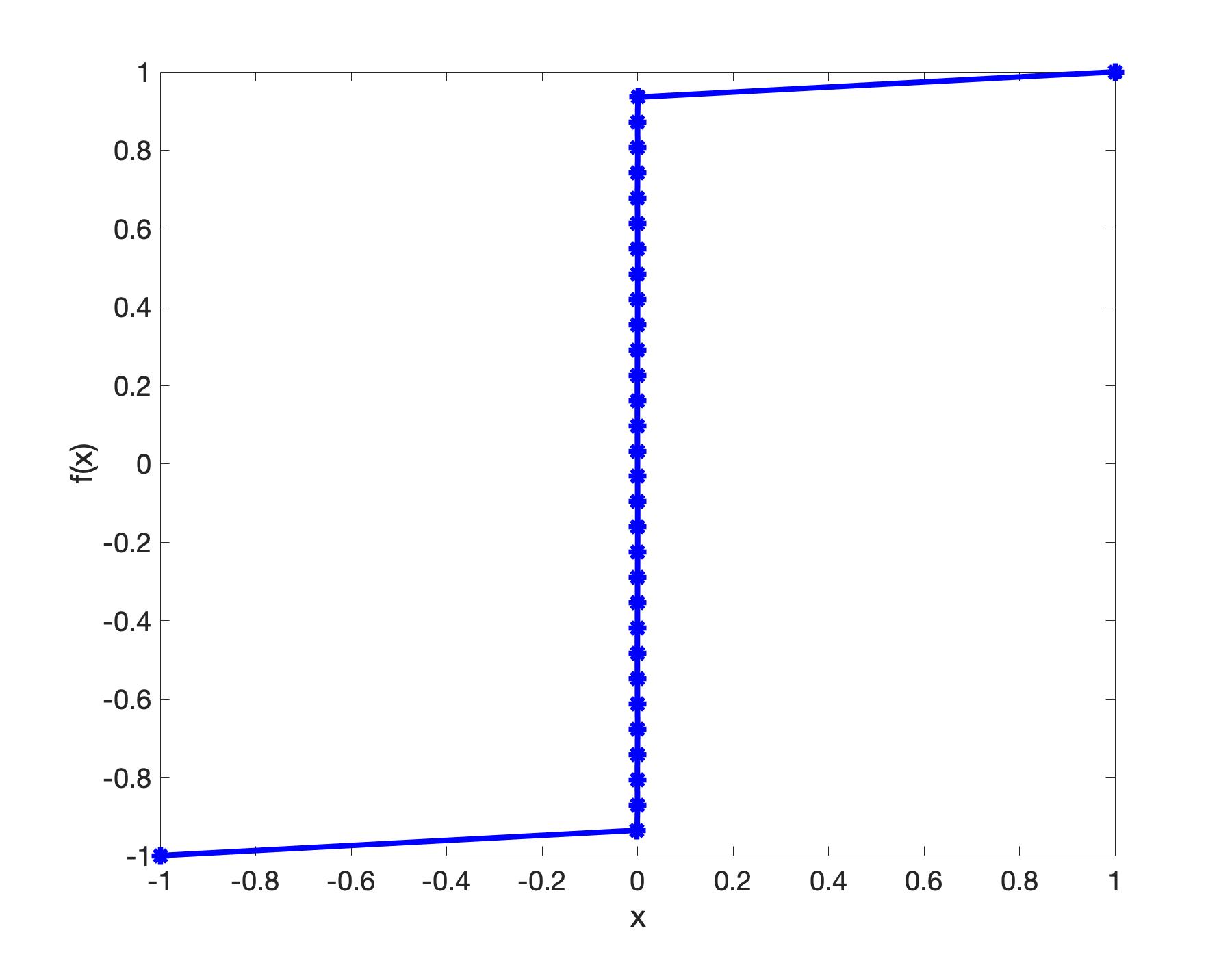}}
\subfigure[$\eps^2=10^{-6}$ and $N=64$]{
\includegraphics[width=0.23\textwidth,clip==]{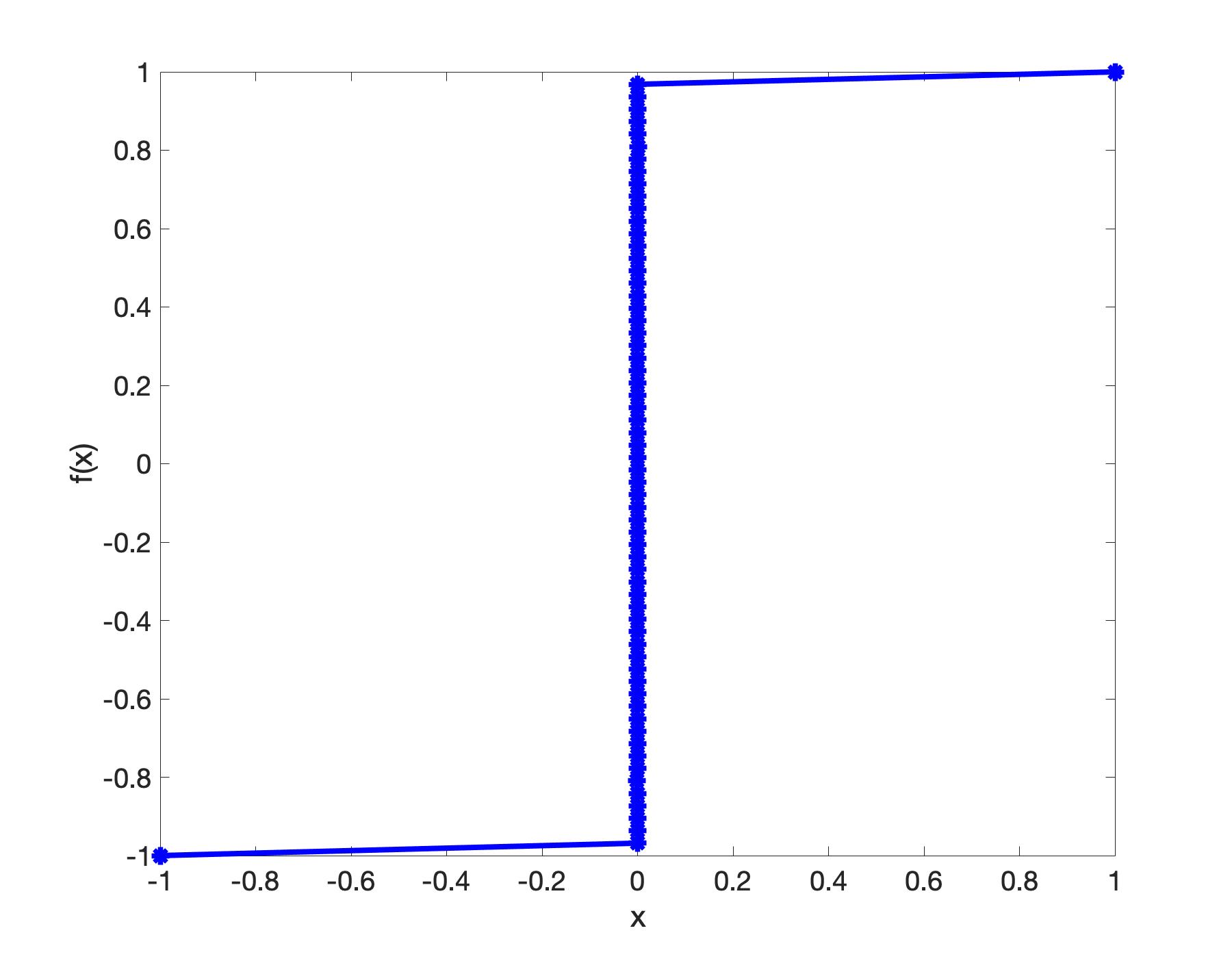}}
\caption{Approximate steady states of Allen-Cahn equation by using the Lagrangian method with  finite-element in space.}\label{fem_steady}
\end{figure}

To obtain better approximation for both location and values of the interface, it is natural to consider the Lagrangian scheme with spectral method in space. 
In Fig.\,\ref{spe1_steady} and  Fig.\,\ref{spe2_steady}, we plot the results for $\eps^2=10^{-3}$ and $\eps^2=10^{-5}$ with $N=8,16,32,64$, respectively.   We first look at the first and third column of the two figures. we observe that while most of the points are still located in the interfacial region, but the approximate solutions exhibit oscillations except at the finest resolution with $N=64$. This is a common phenomena with under-resolved spectral methods. Usually this can be fixed with a suitable filter to post-process the oscillatory approximate solutions \cite{vandeven1991family,hesthaven2008filtering}. 

Hence, in order to remove the oscillation, we use    an exponential  filter for  post-processing. More precisely, given approximate solution  $u_N=\sum\limits_{n=0}^N\hat{u}_nL_n(x)$ with $L_n(x)$ being the Legendre polynomial of degree $n$, we set the filtered solution to be
\begin{equation}
F_Nu_N=\sum\limits_{n=0}^N\sigma(\frac{n}{N})\hat{u}_nL_n(x),
\end{equation}
where $\sigma(\frac{n}{N})=\exp(-a(\frac{n}{N}))$, and $a=-\log(\eps_M)$ where $\eps_M$ is the machine accuracy.  The filtered results are presented in the second and fourth columns of Figs.\,\ref{spe1_steady} and  Fig.\,\ref{spe2_steady}. We observe that the filtered solutions are non-oscillatory and approximate the exact solutions much better than the finite-element methods. In fact, while the values with $N=8$ are still visibly different from the exact solution, excellent approximations are obtained with $N=16$ for both cases.

\begin{figure}
\centering
\subfigure[ $\eps^2=10^{-3}$  and $N=8$ without spectral filter.]{
\includegraphics[width=0.23\textwidth,clip==]{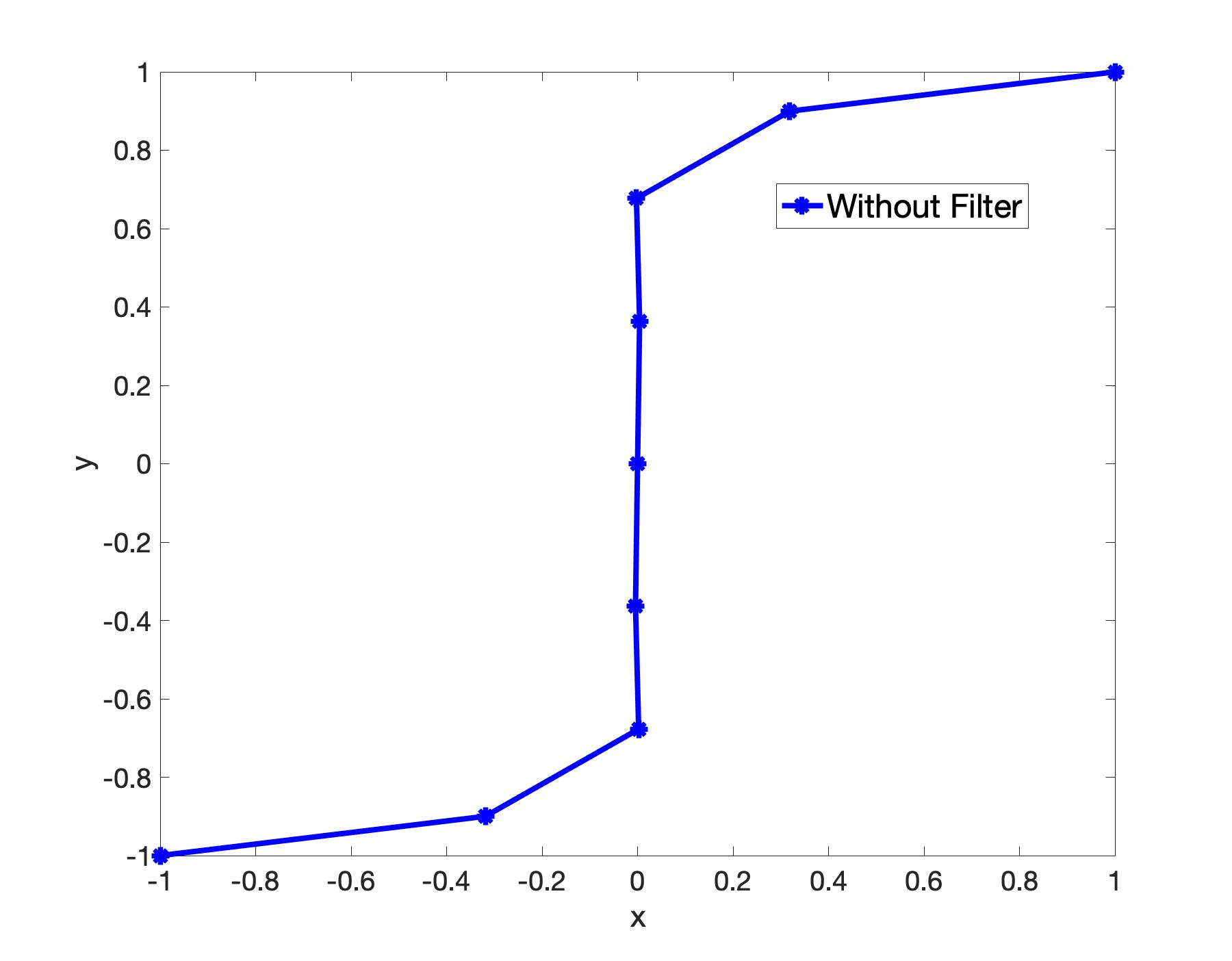}}
\subfigure[$\eps^2=10^{-3}$ and $N=8$ with spectral filter .]{
\includegraphics[width=0.23\textwidth,clip==]{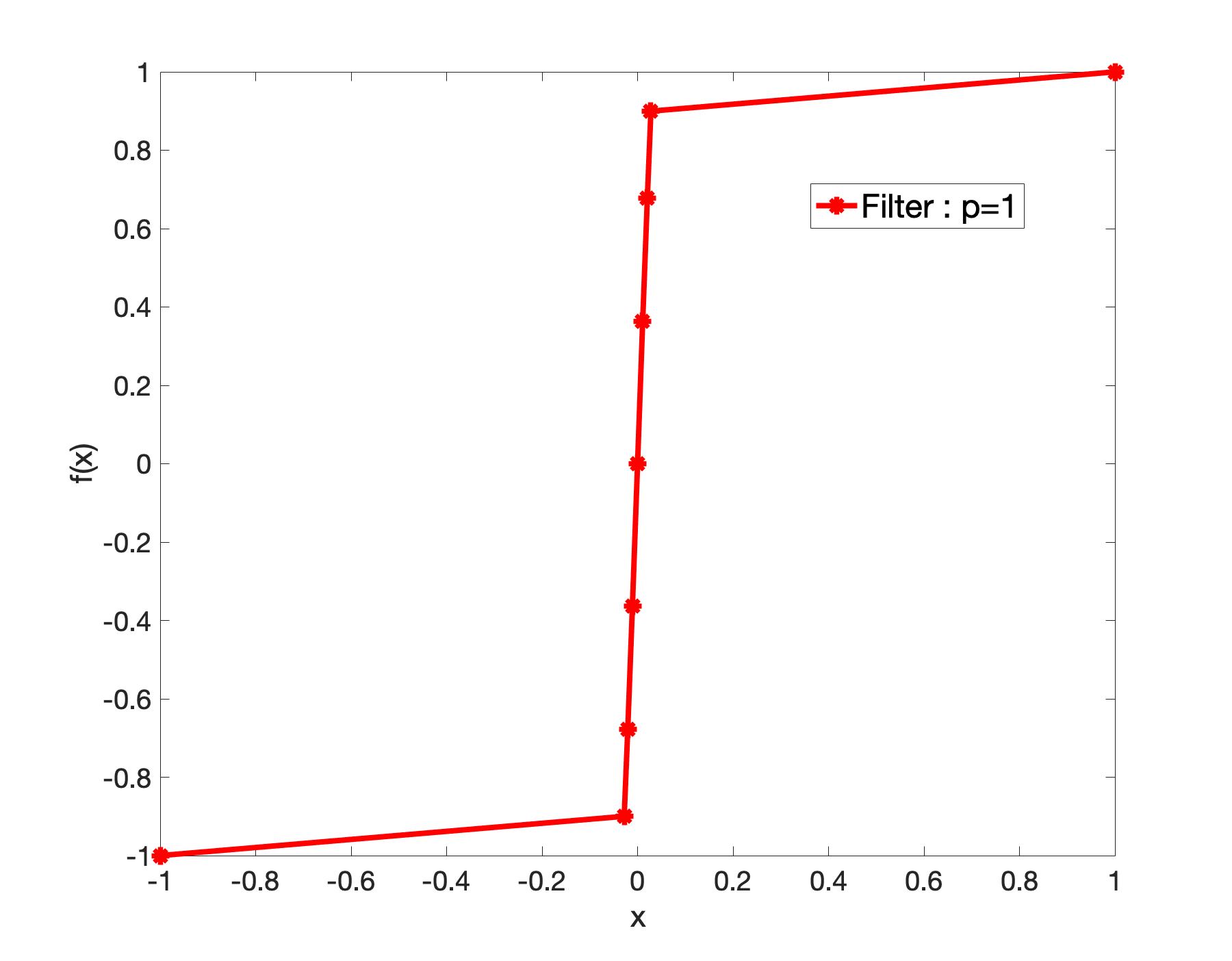}}
\subfigure[$\eps^2=10^{-3}$ and $N=16$ without spectral filter.]{
\includegraphics[width=0.23\textwidth,clip==]{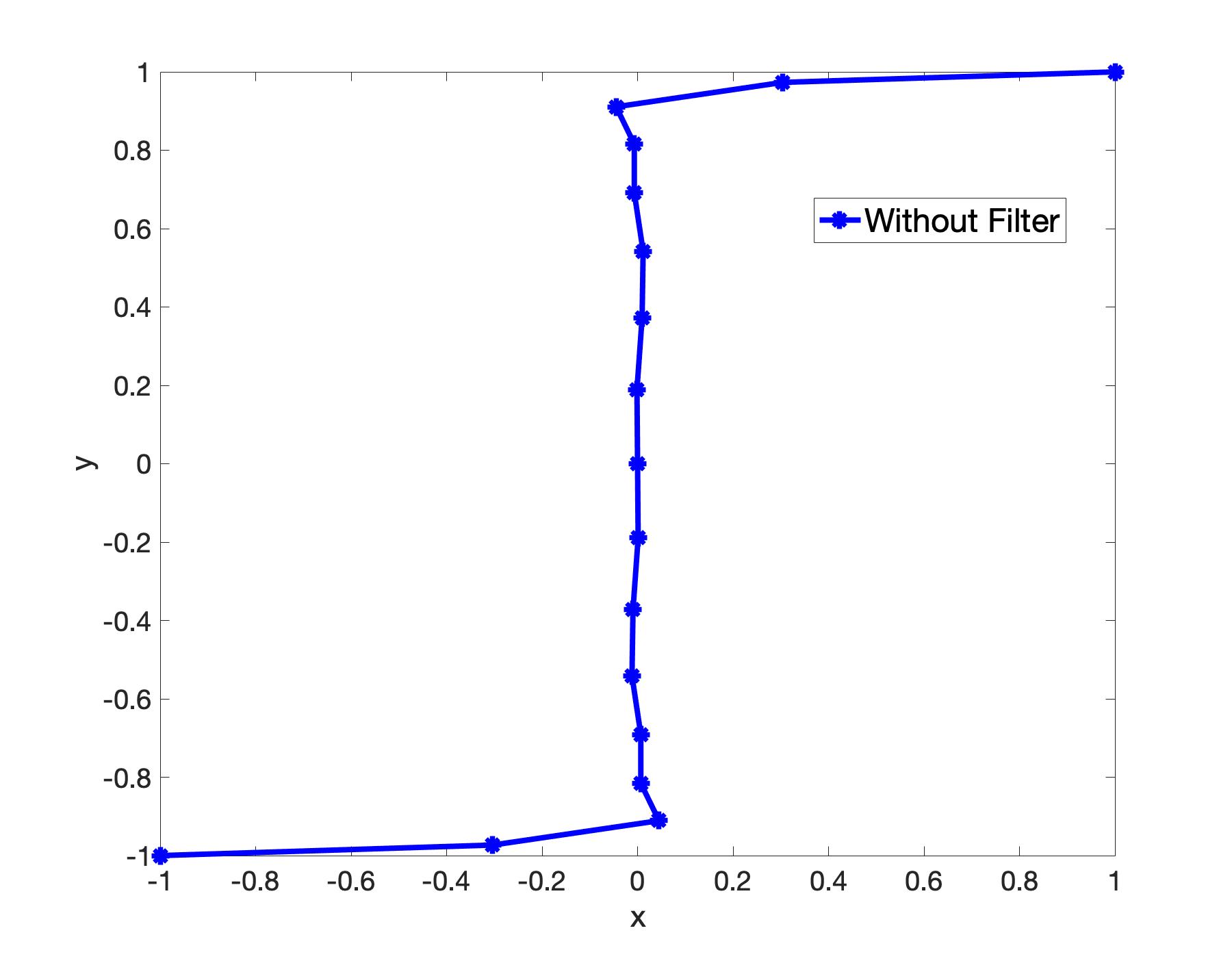}}
\subfigure[$\eps^2=10^{-3}$ and $N=16$ with spectral filter.]{
\includegraphics[width=0.23\textwidth,clip==]{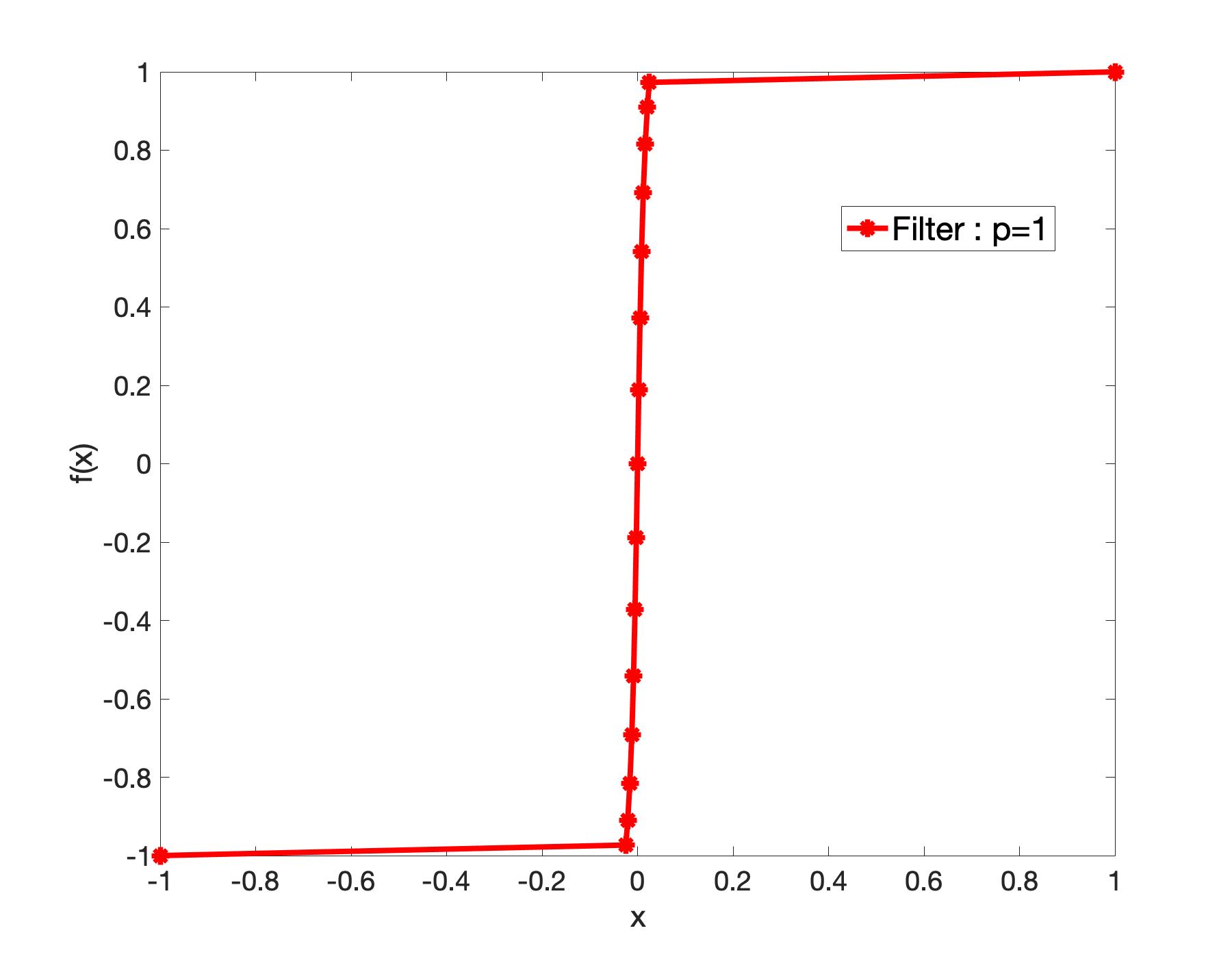}}
\subfigure[$\eps^2=10^{-3}$ and $N=32$ without spectral filter.]{
\includegraphics[width=0.23\textwidth,clip==]{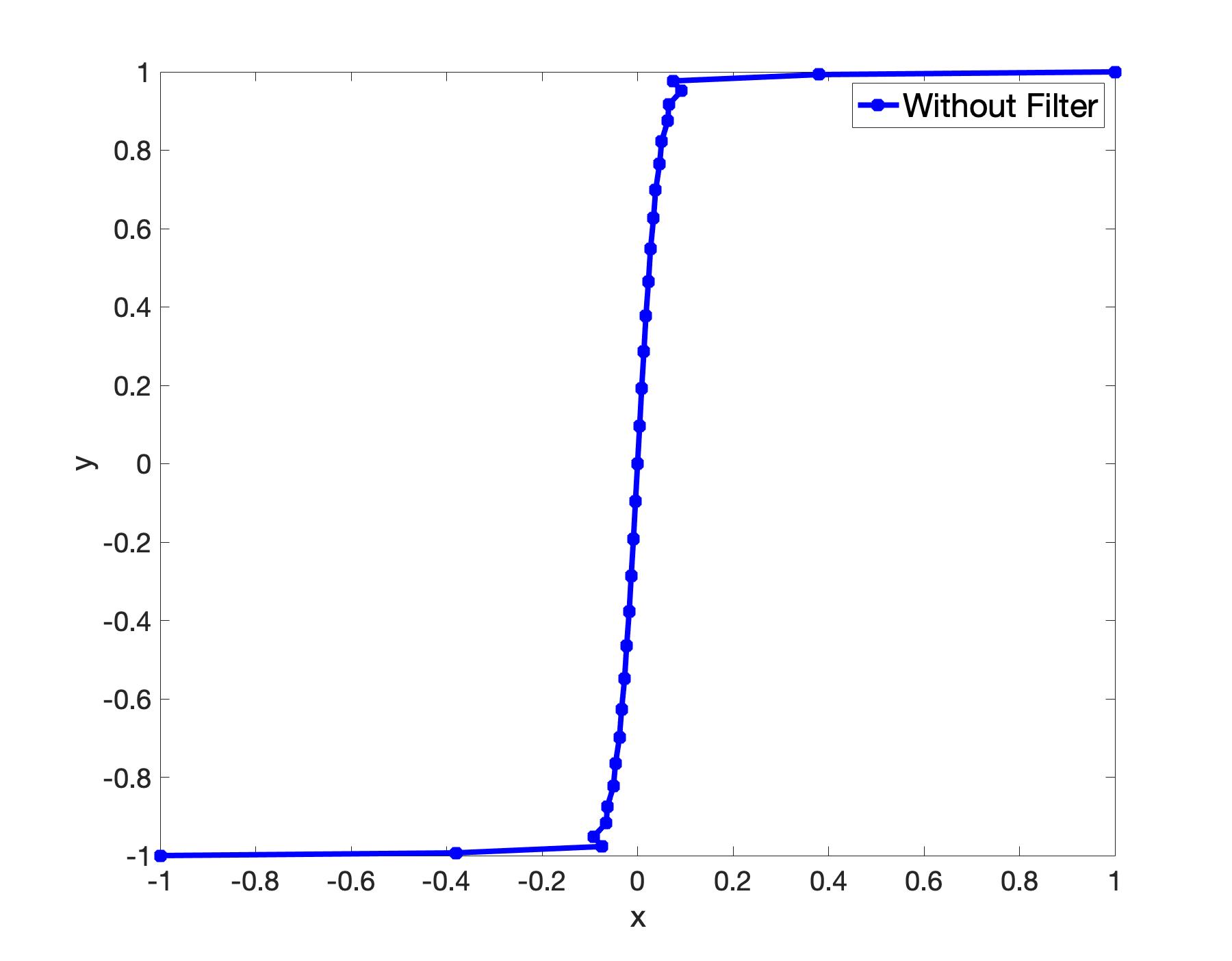}}
\subfigure[$\eps^2=10^{-3}$ and $N=32$ with spectral filter.]{
\includegraphics[width=0.23\textwidth,clip==]{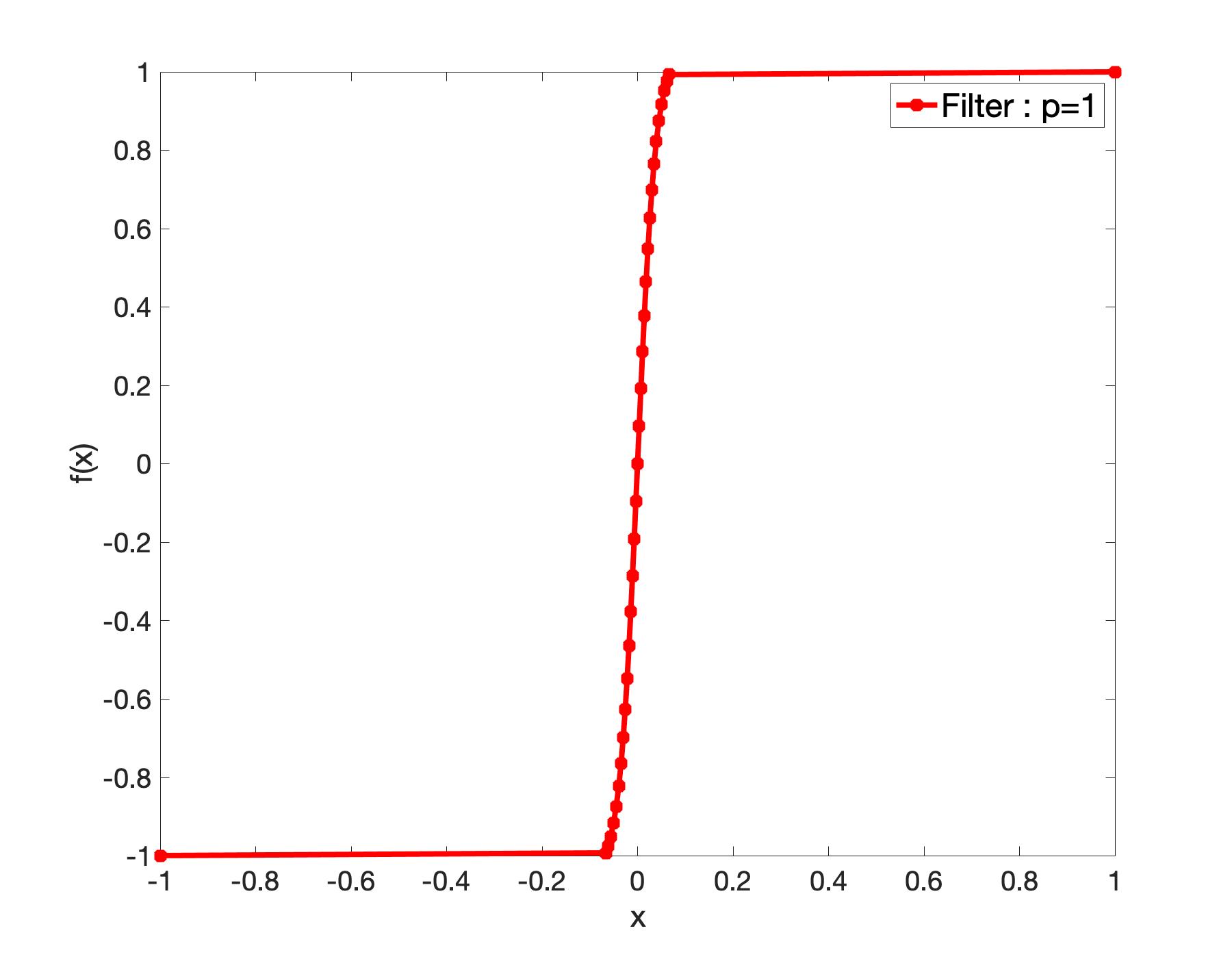}}
\subfigure[$\eps^2=10^{-3}$ and $N=64$ without spectral filter.]{
\includegraphics[width=0.23\textwidth,clip==]{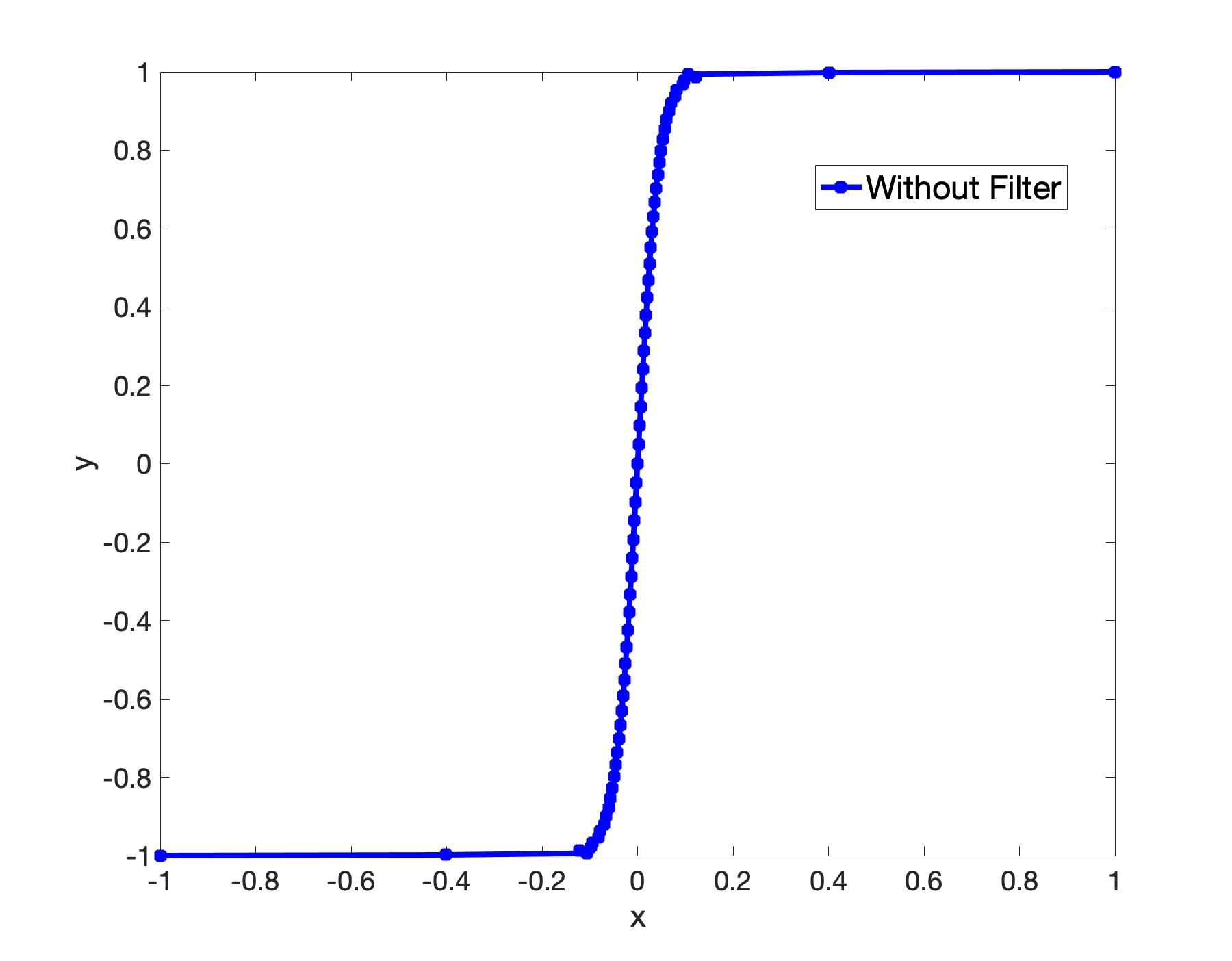}}
\subfigure[$\eps^2=10^{-3}$ and $N=64$ with spectral filter.]{
\includegraphics[width=0.23\textwidth,clip==]{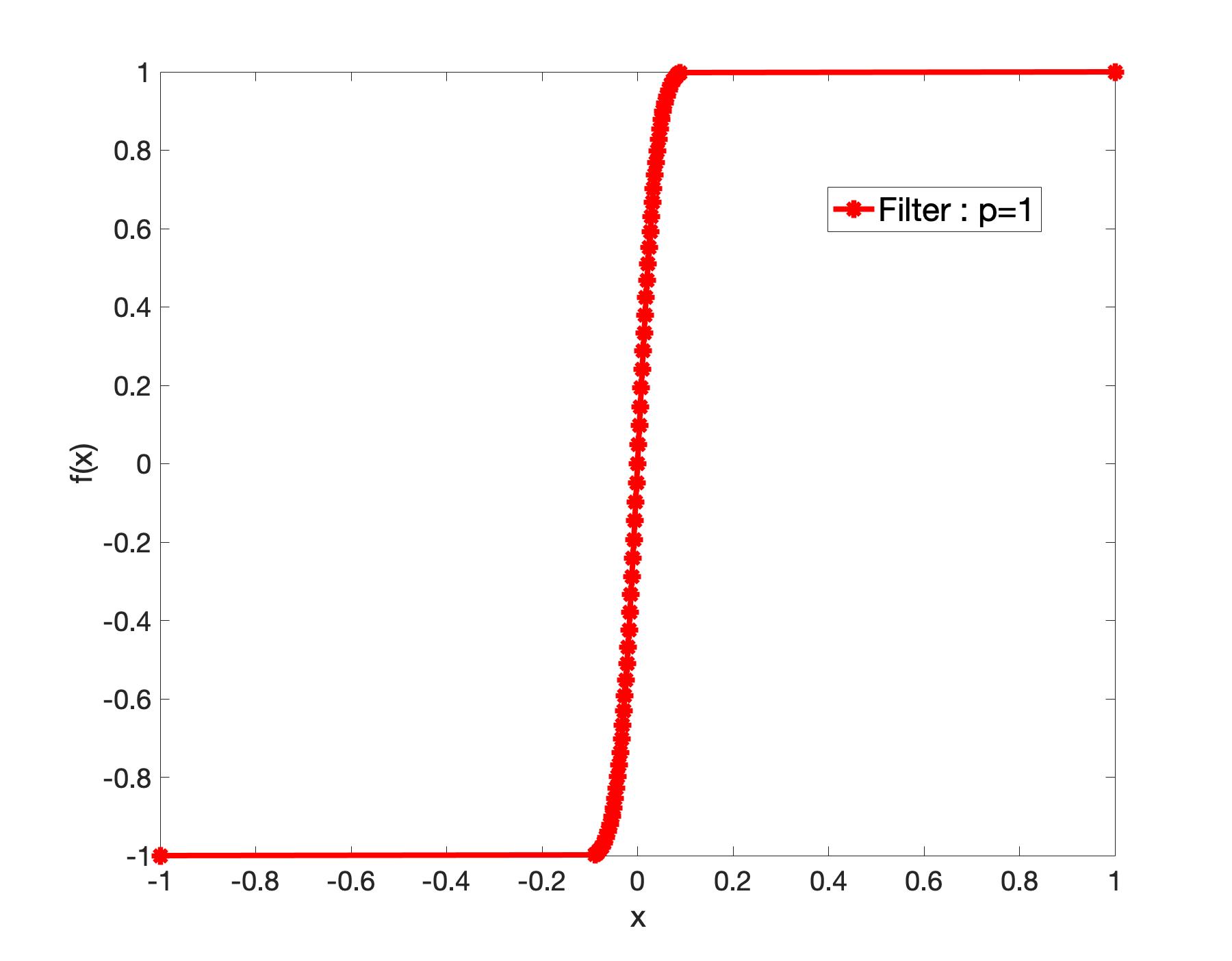}}
\caption{Approximate steady states of Allen-Cahn equation by the Lagrangian scheme  with   Legendre Spectral method in space for $\eps^2=10^{-3}$.  }\label{spe1_steady}
\end{figure}

\begin{figure}
\centering
\subfigure[ $\eps^2=10^{-5}$  and $N=8$ without spectral filter.]{
\includegraphics[width=0.23\textwidth,clip==]{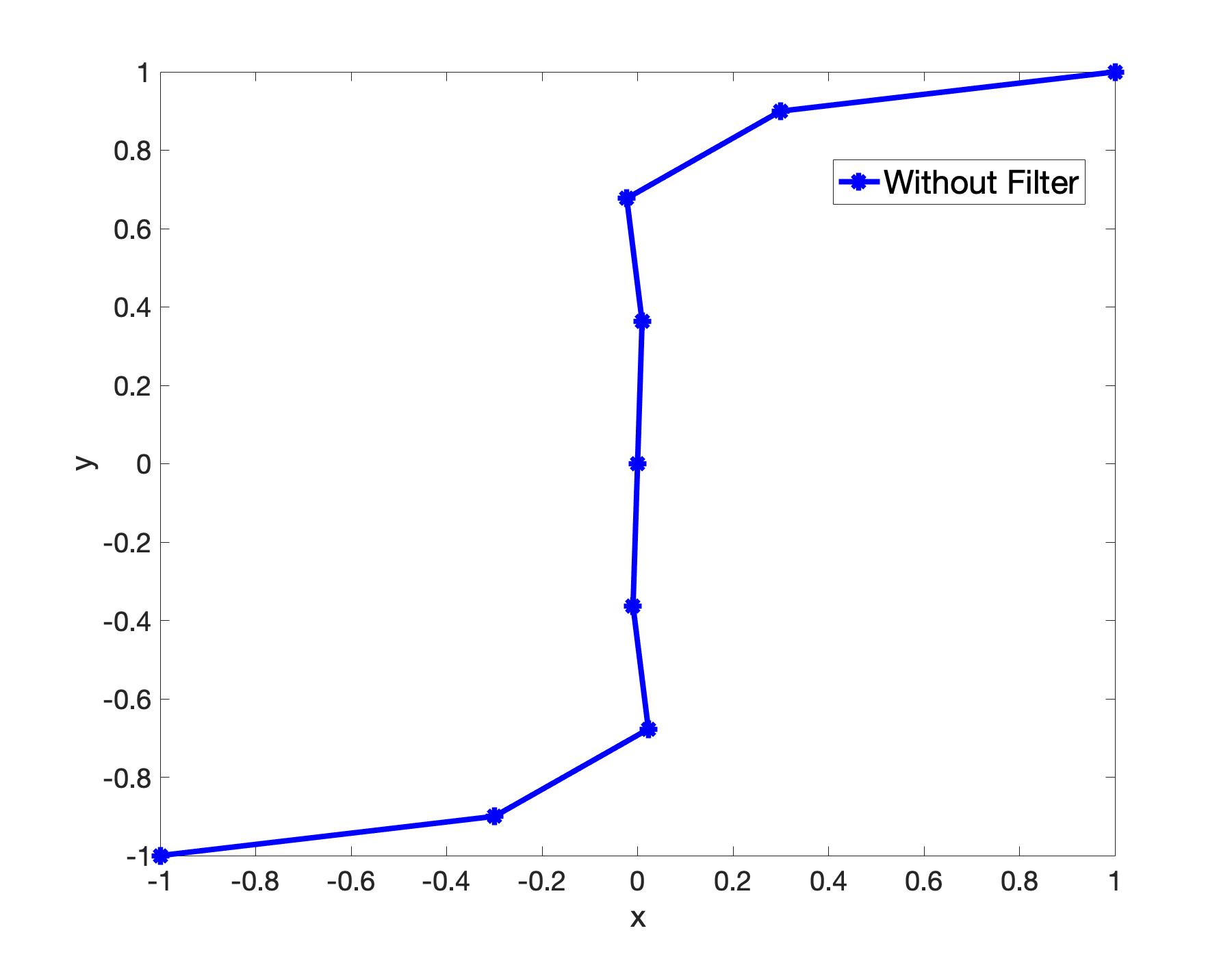}}
\subfigure[$\eps^2=10^{-5}$ and $N=8$ with spectral filter .]{
\includegraphics[width=0.23\textwidth,clip==]{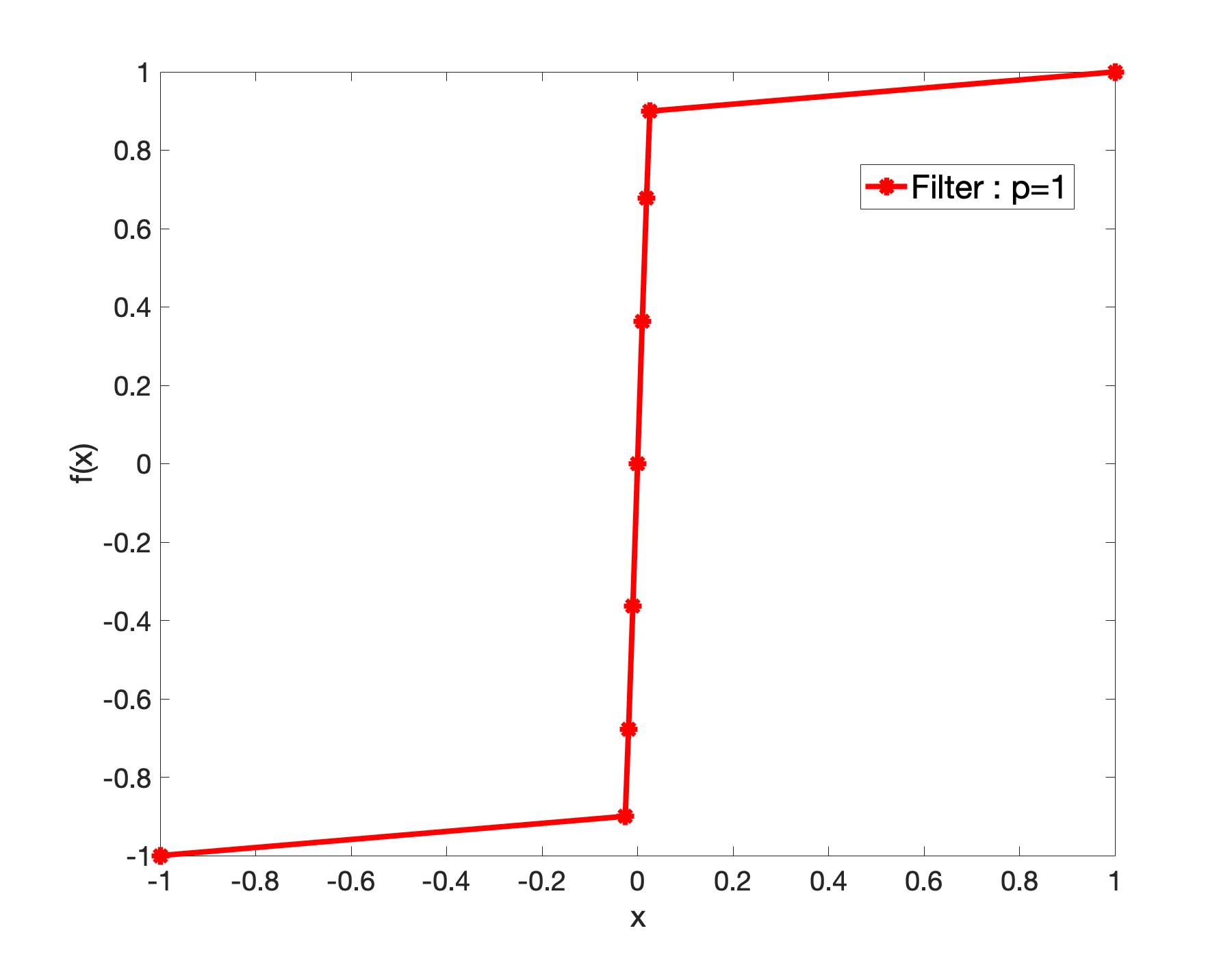}}
\subfigure[$\eps^2=10^{-5}$ and $N=16$ without spectral filter.]{
\includegraphics[width=0.23\textwidth,clip==]{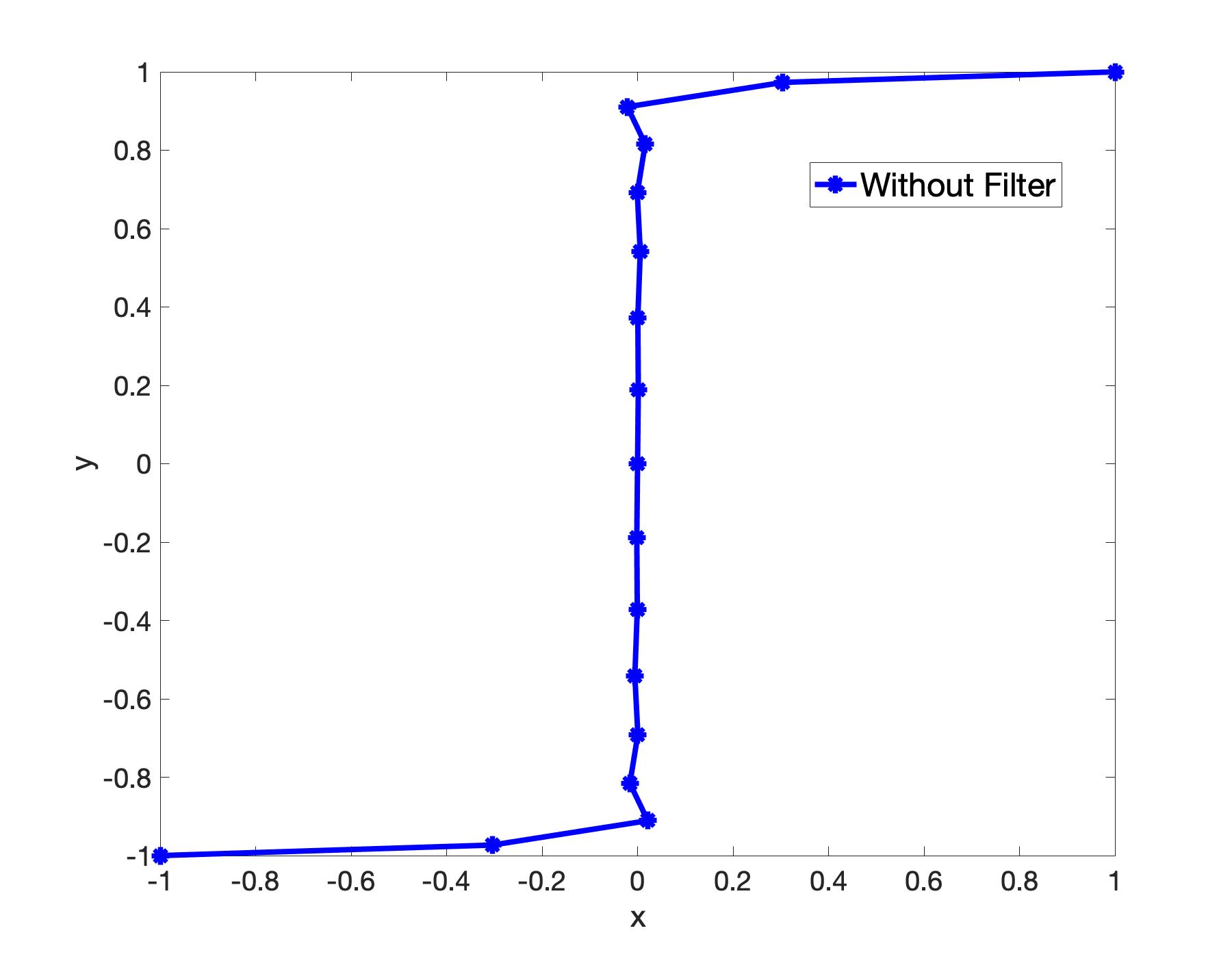}}
\subfigure[$\eps^2=10^{-5}$ and $N=16$ with spectral filter.]{
\includegraphics[width=0.23\textwidth,clip==]{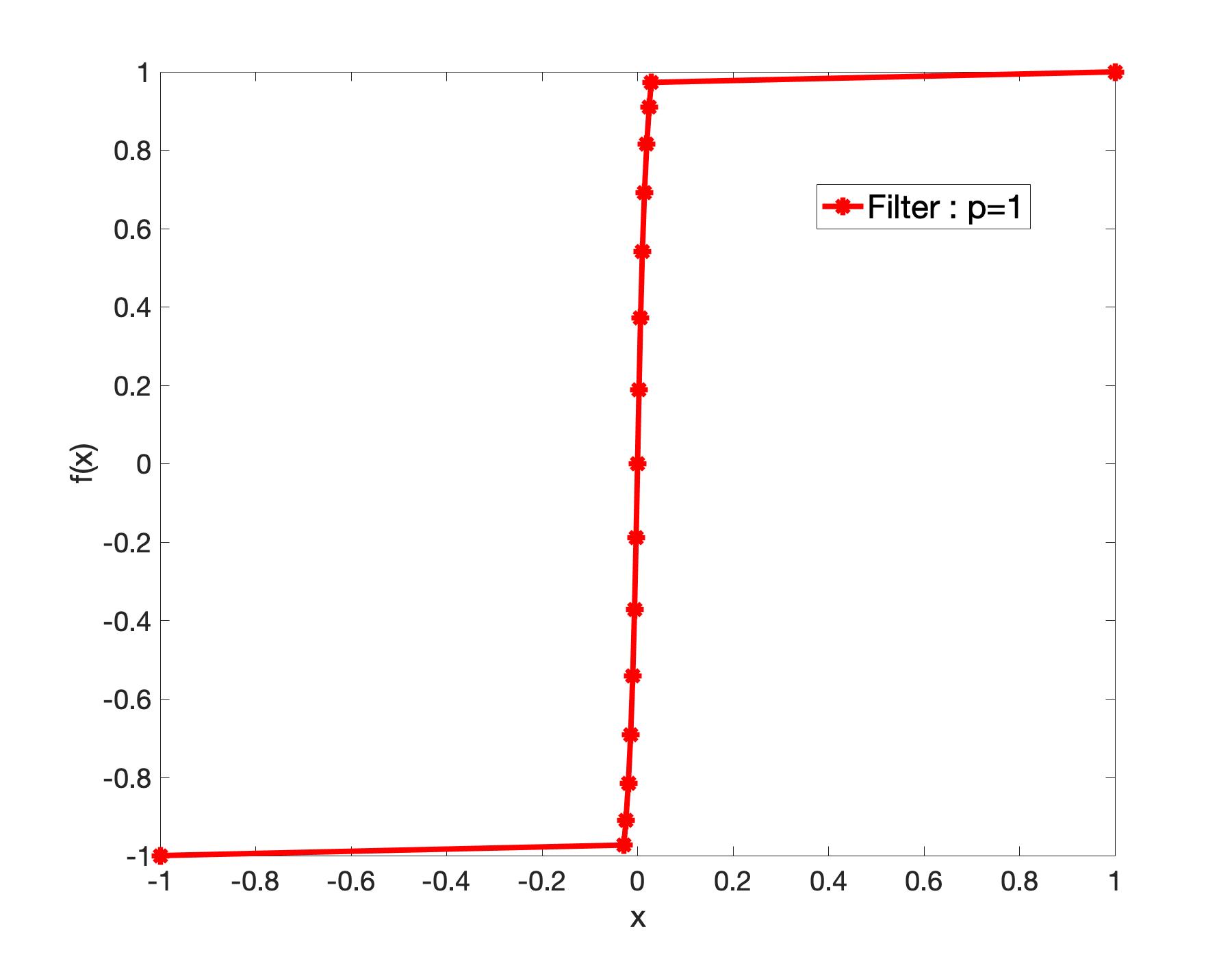}}
\subfigure[$\eps^2=10^{-5}$ and $N=32$ without spectral filter.]{
\includegraphics[width=0.23\textwidth,clip==]{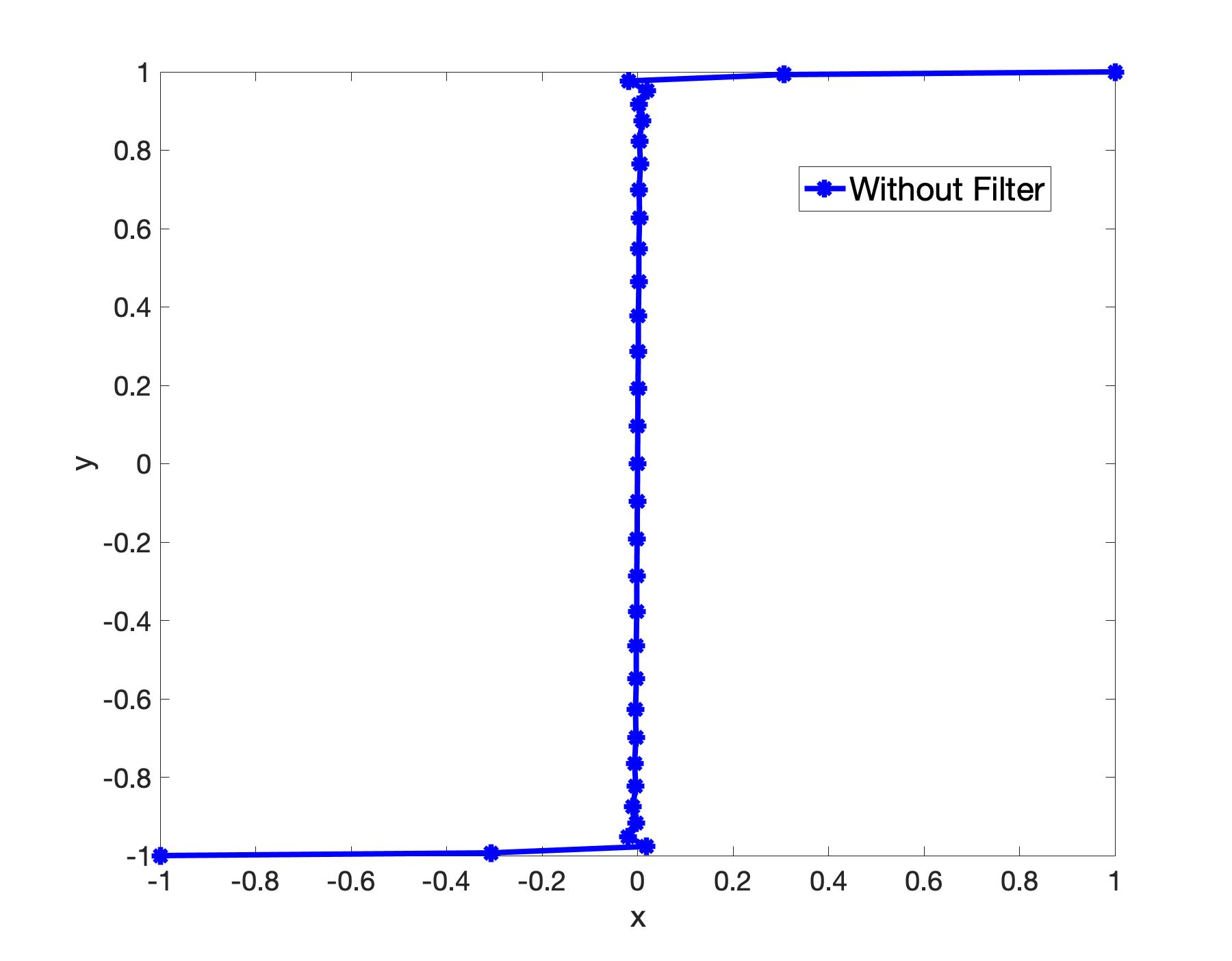}}
\subfigure[$\eps^2=10^{-5}$ and $N=32$ with spectral filter.]{
\includegraphics[width=0.23\textwidth,clip==]{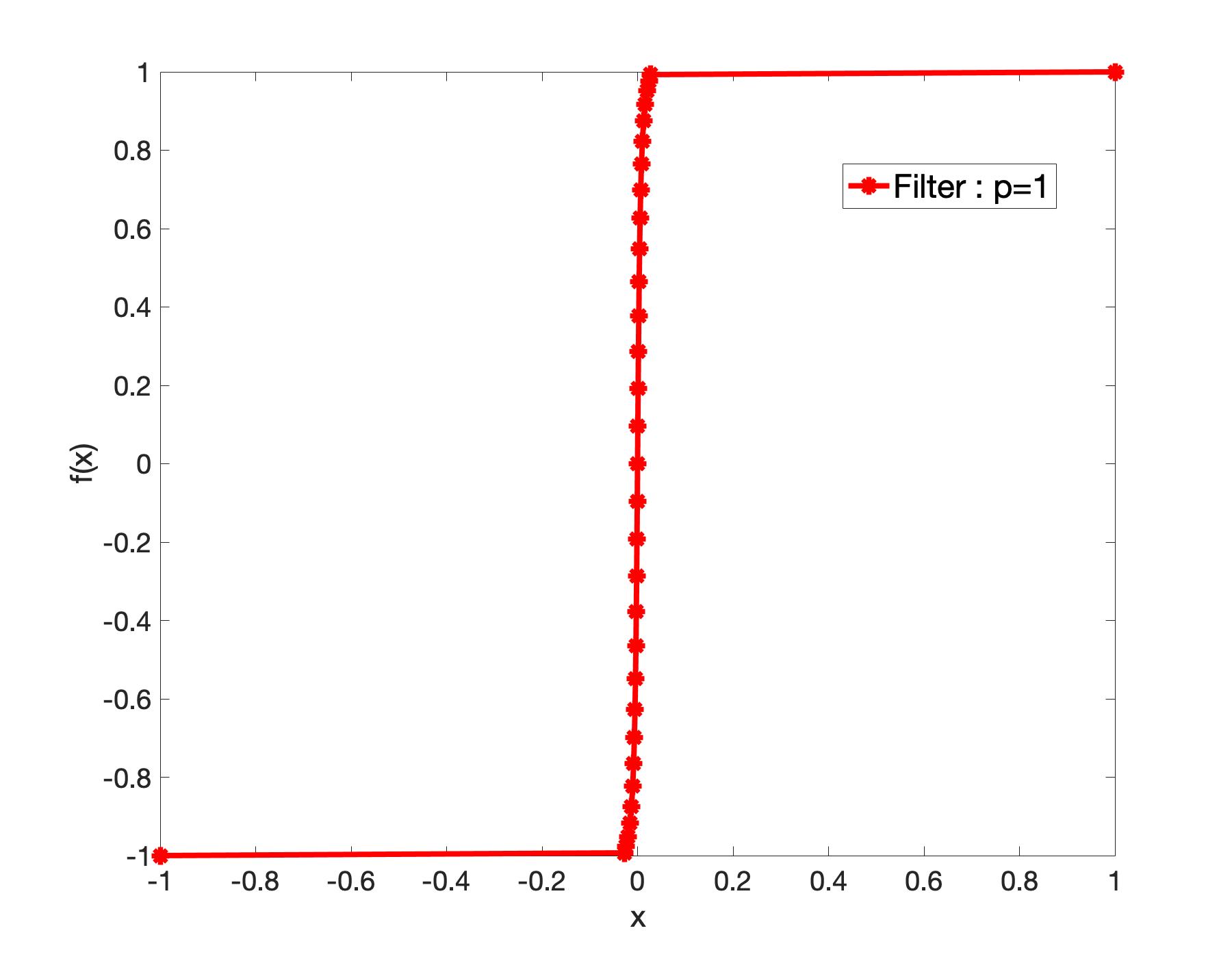}}
\subfigure[$\eps^2=10^{-5}$ and $N=64$ without spectral filter.]{
\includegraphics[width=0.23\textwidth,clip==]{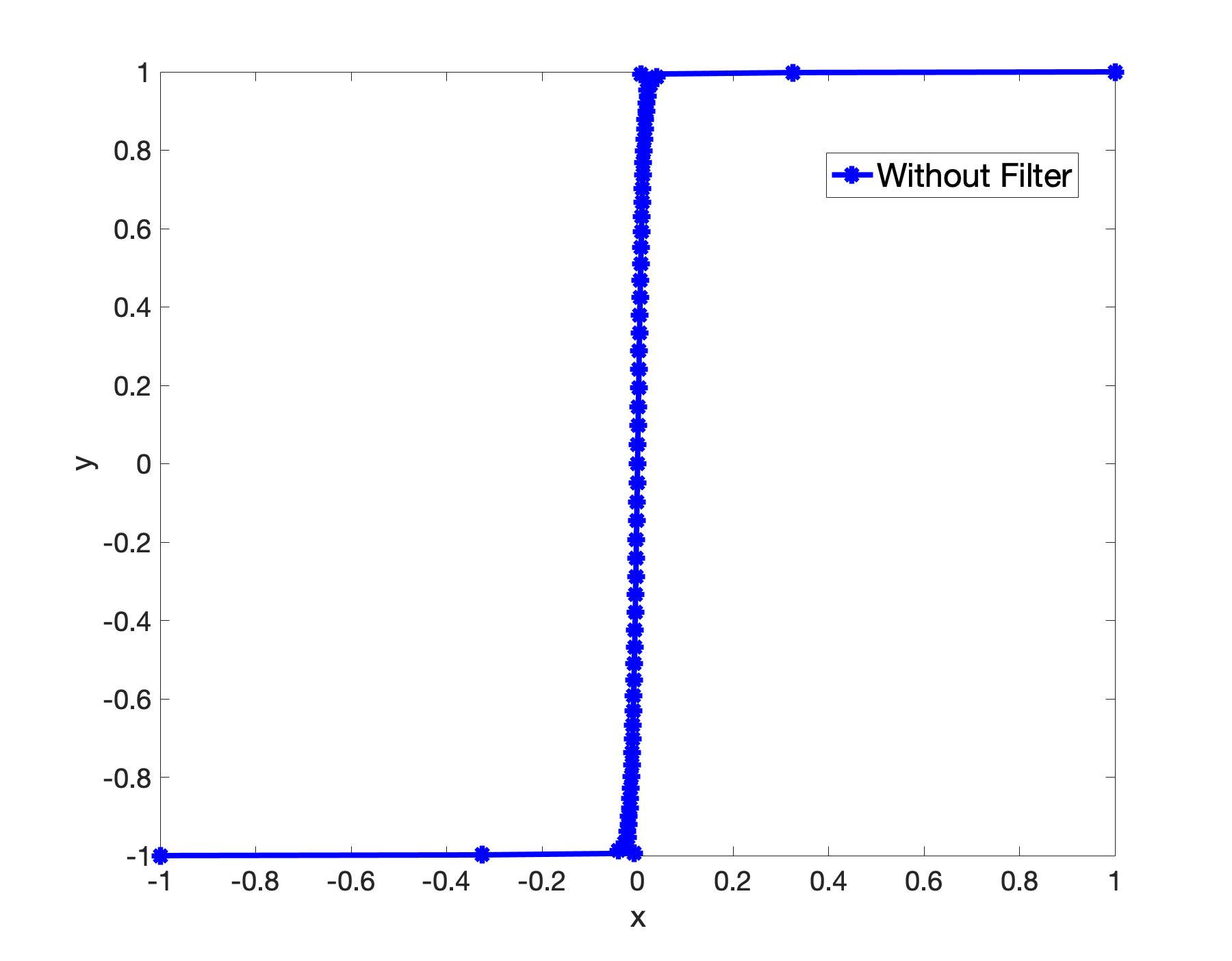}}
\subfigure[$\eps=10^{-5}$ and $N=64$ with spectral filter.]{
\includegraphics[width=0.23\textwidth,clip==]{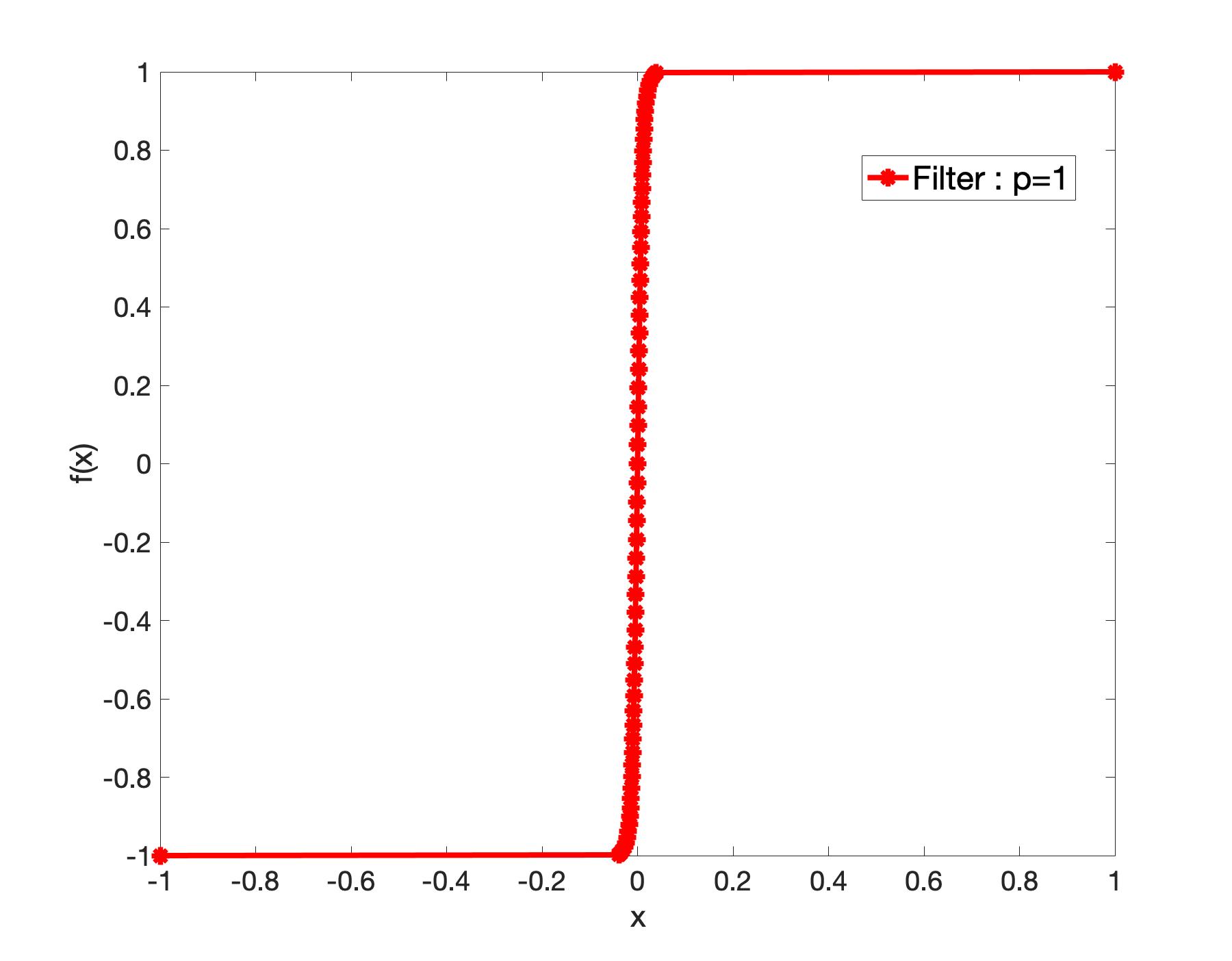}}
\caption{Approximate steady states of Allen-Cahn equation by the Lagrangian scheme  with   Legendre Spectral method in space for $\eps^2=10^{-5}$.}\label{spe2_steady}
\end{figure}

Next, we consider the generalized Allen-Cahn equation  \eqref{move:sharp:1} with an advection velocity $\bv\equiv 1$,  so  the interface will evolve and move to the right. We would like to see how our Lagrangian method performs with moving interfaces.
 In Fig.\,\ref{moving_4} we plot  the  interface profiles at various times computed by the Lagrangian scheme with spectral method and finite element method in space for the generalized Allen-Cahn equation \eqref{move:sharp:1} with     $\bv\equiv 1$. As a comparison, we also plotted results by using a semi-implicit method in Eulerian coordinate.  We observe that as the interface moves, our Lagrangian method can still capture the interface well with few points.

\begin{figure}
\centering
\subfigure[ Flow dynamic approach with Finite element method,  $\eps^2=0.001$, $N=8$.]{
\includegraphics[width=0.45\textwidth,clip==]{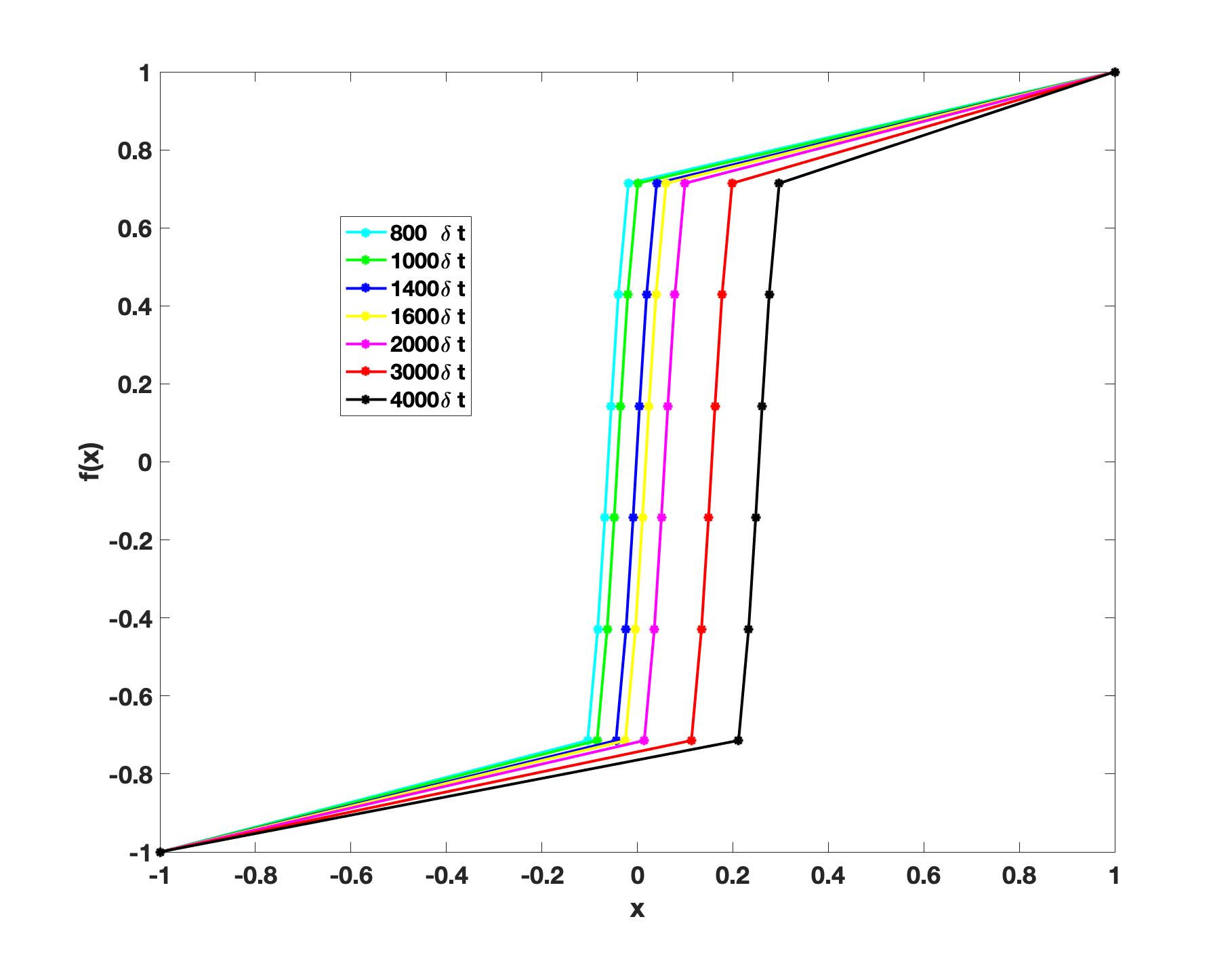}}
\subfigure[ Flow dynamic approach with Finite element method,  $\eps^2=0.001$, $N=16$.]{
\includegraphics[width=0.45\textwidth,clip==]{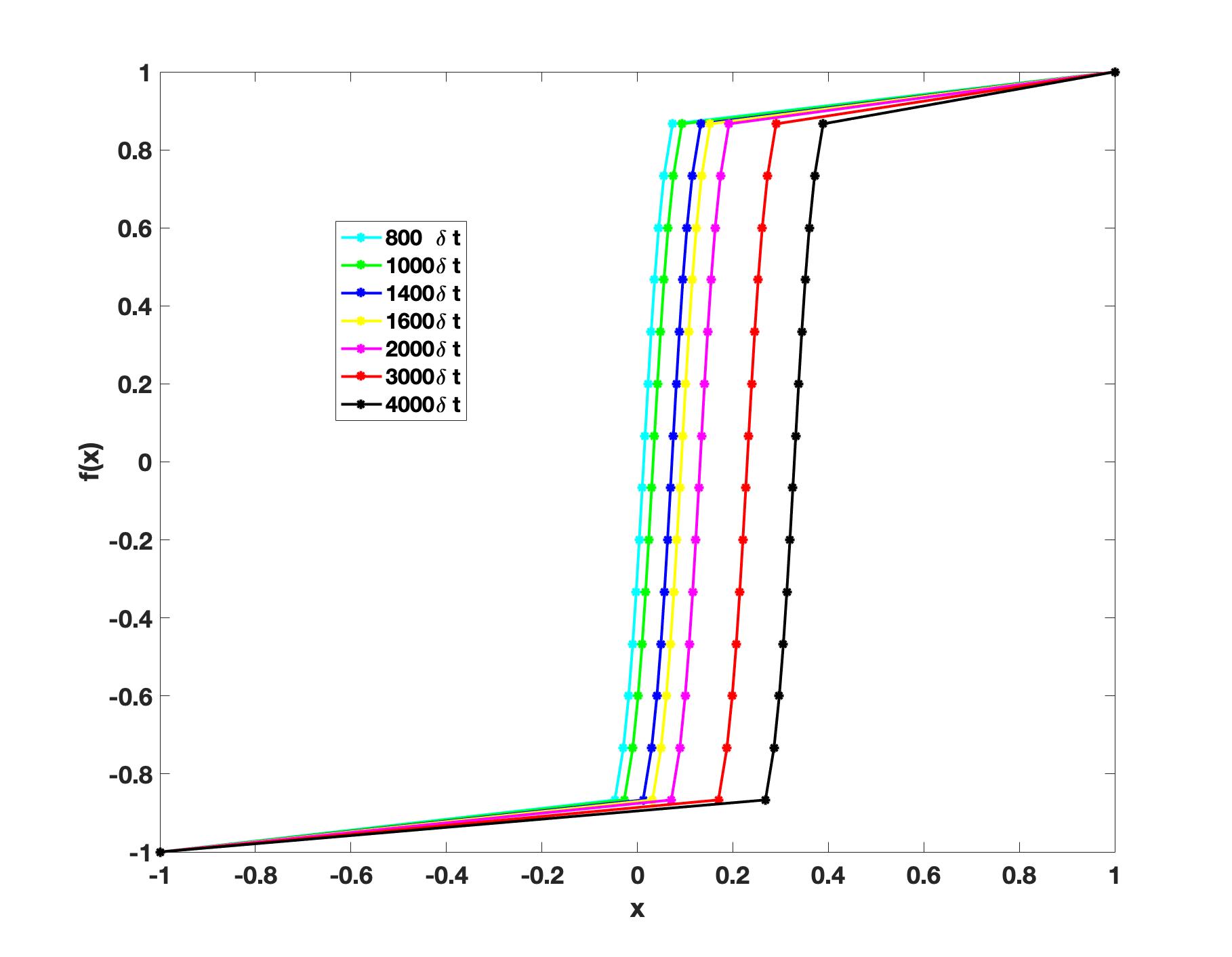}}
\subfigure[ Flow dynamic approach with Spectral method,  $\eps^2=0.001$, $N=64$.]{
\includegraphics[width=0.45\textwidth,clip==]{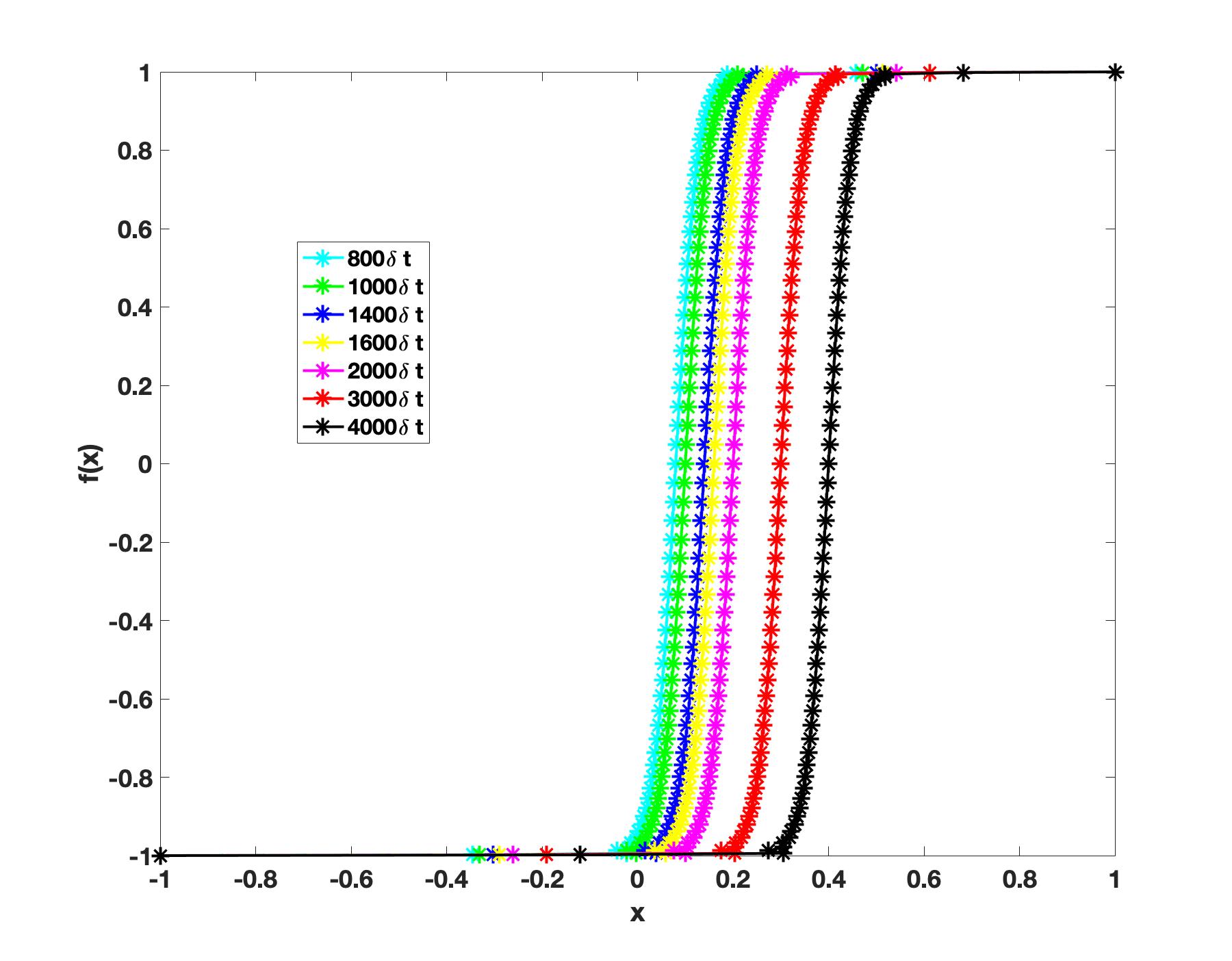}}
\subfigure[ Numerical method in Eulerian coordinate, $\eps^2=0.001$, $N=256$.]{
\includegraphics[width=0.45\textwidth,clip==]{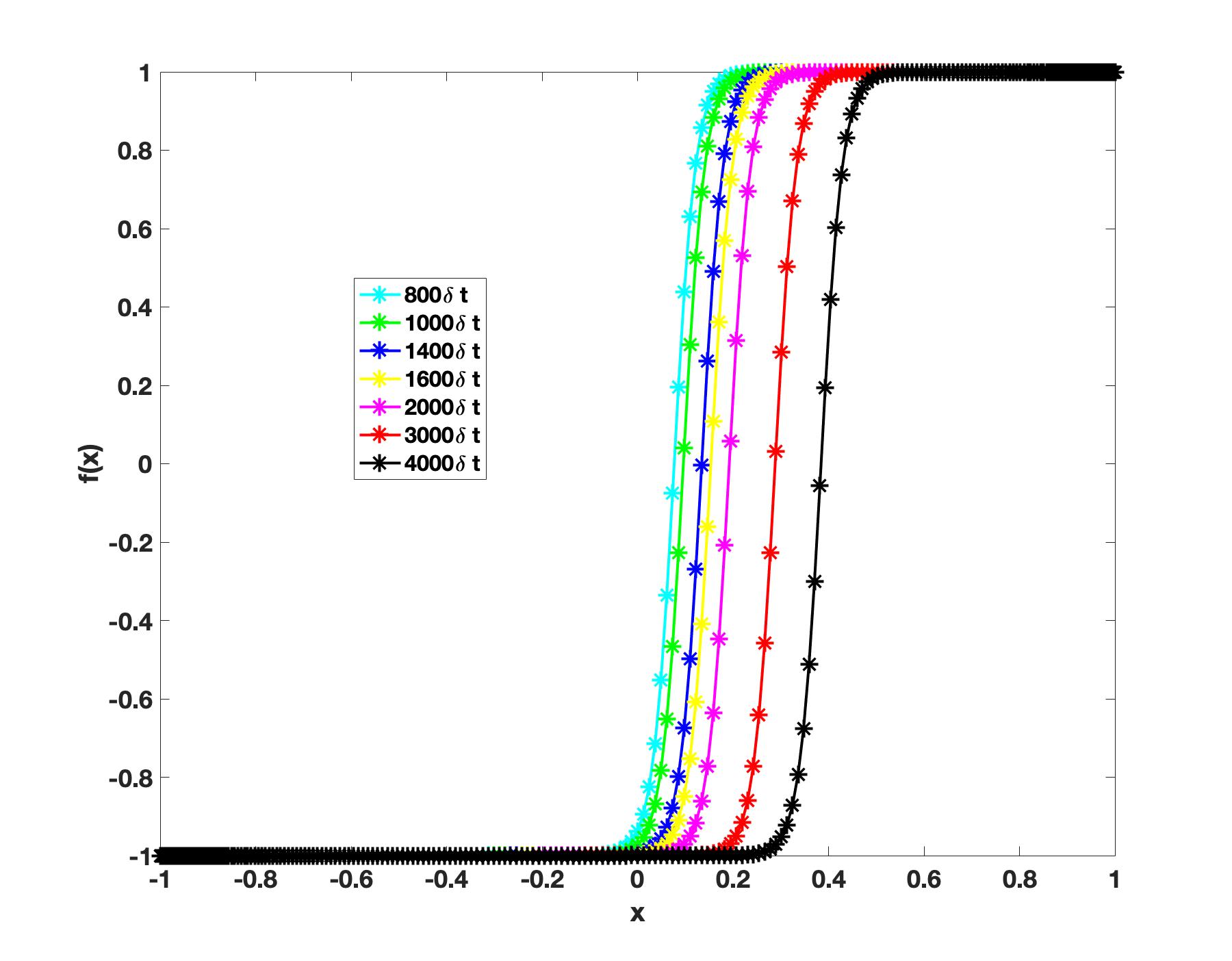}}
\caption{Approximate solutions for the generalized Allen-Cahn equation. }\label{moving_4}
\end{figure}

\subsection{Two dimensional axis-symmetric case}
As a final example, we  examine the performance of our flow dynamic approach for a two dimensional axis-symmetric case with  $\Omega=\{x^2+y^2<1\}$ and initial condition $f_0(x)=x$. 
More precisely, we solve \eqref{AC-axi} with $\eps^2=0.001$ using the Lagrangian scheme with a spectral method in space with $N=16,64$.  
Since \eqref{AC-axi} is axi-symmetric, we only plot the one-dimensional profiles in Fig.\,\ref{axis-symmetric}.

\begin{figure}
\centering
\subfigure[ Flow dynamic approach with spectral method of $N=16$.]{
\includegraphics[width=0.45\textwidth,clip==]{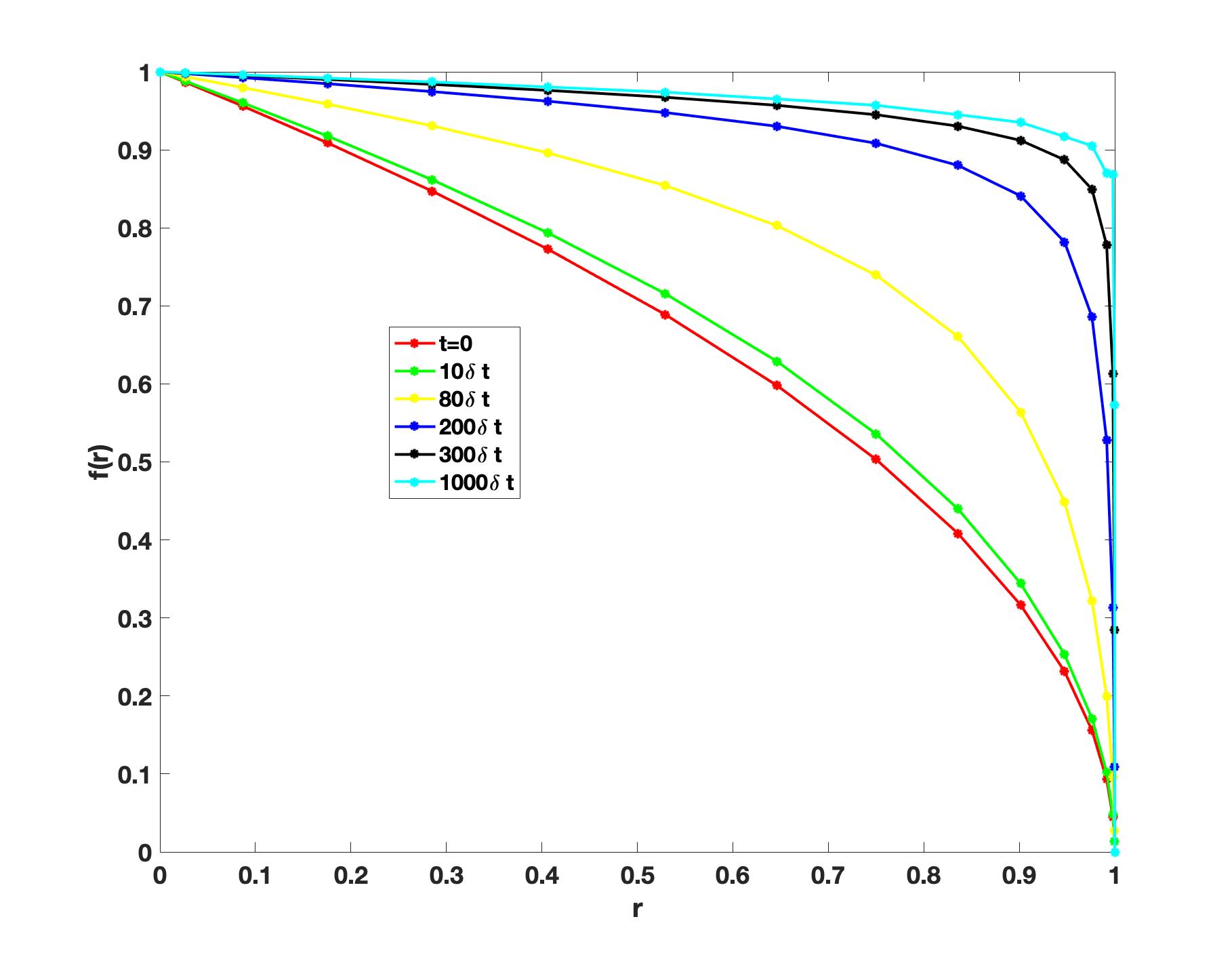}
\includegraphics[width=0.45\textwidth,clip==]{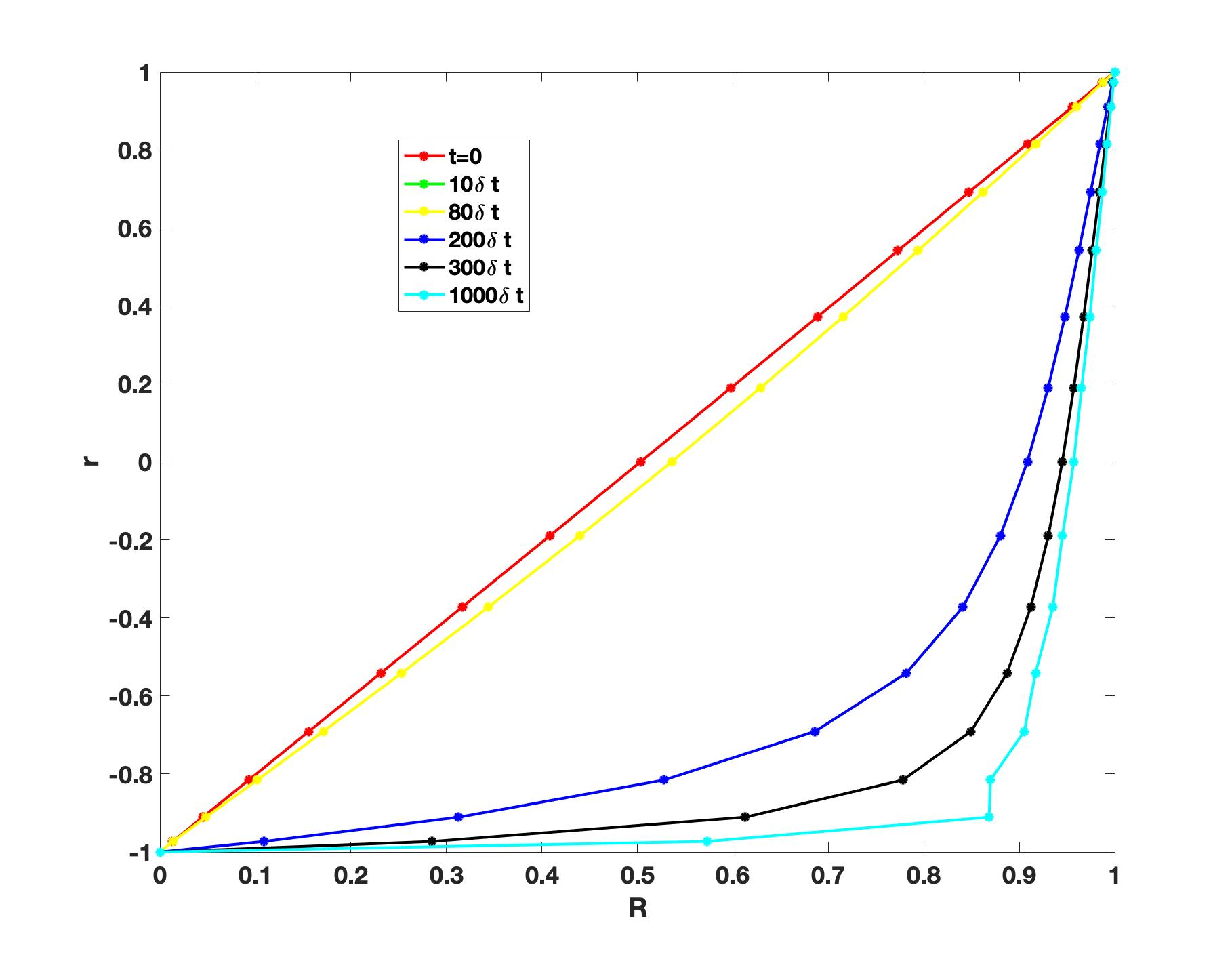}}
\subfigure[ Flow dynamic approach with spectral method of  $N=64$.]{
\includegraphics[width=0.45\textwidth,clip==]{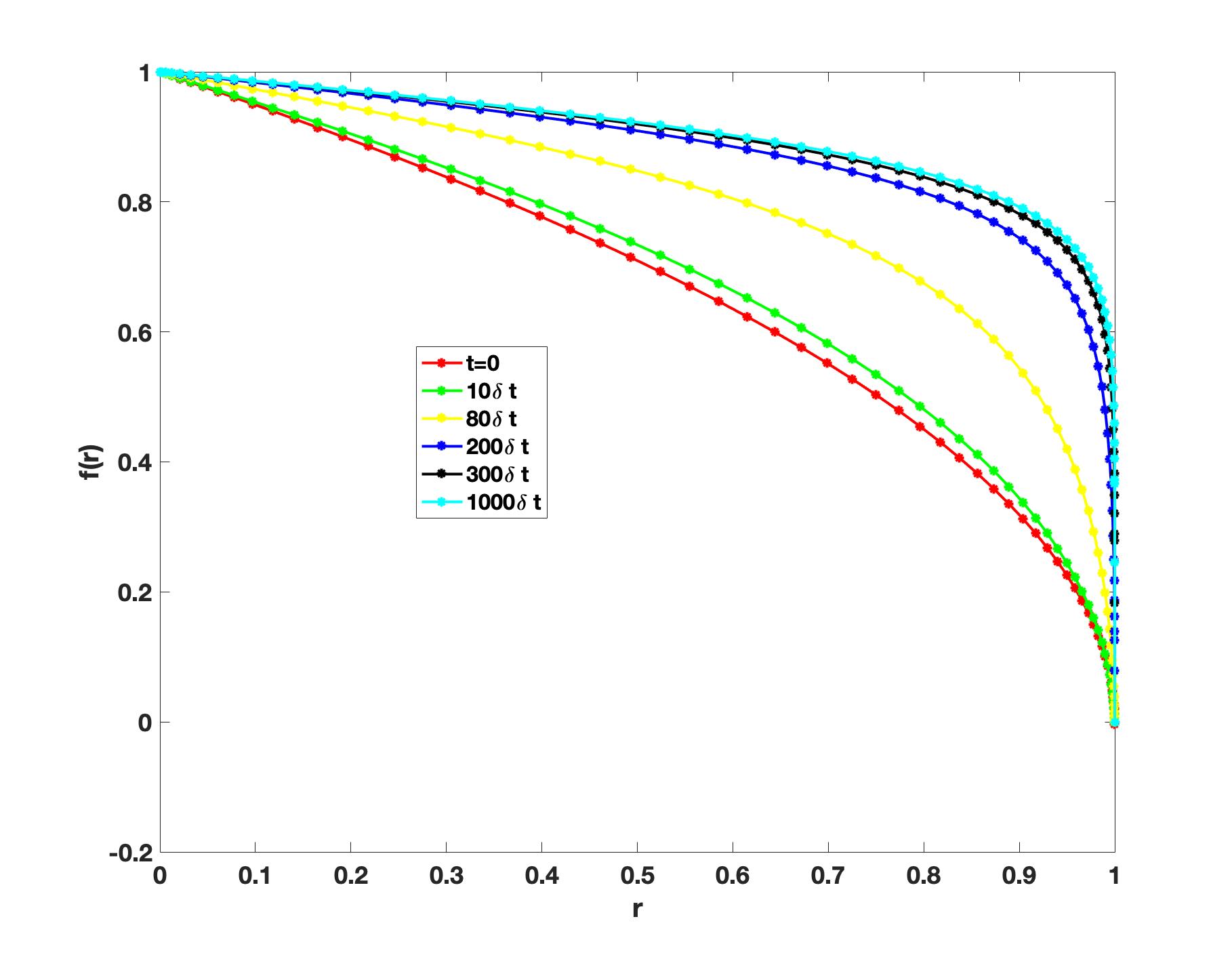}
\includegraphics[width=0.45\textwidth,clip==]{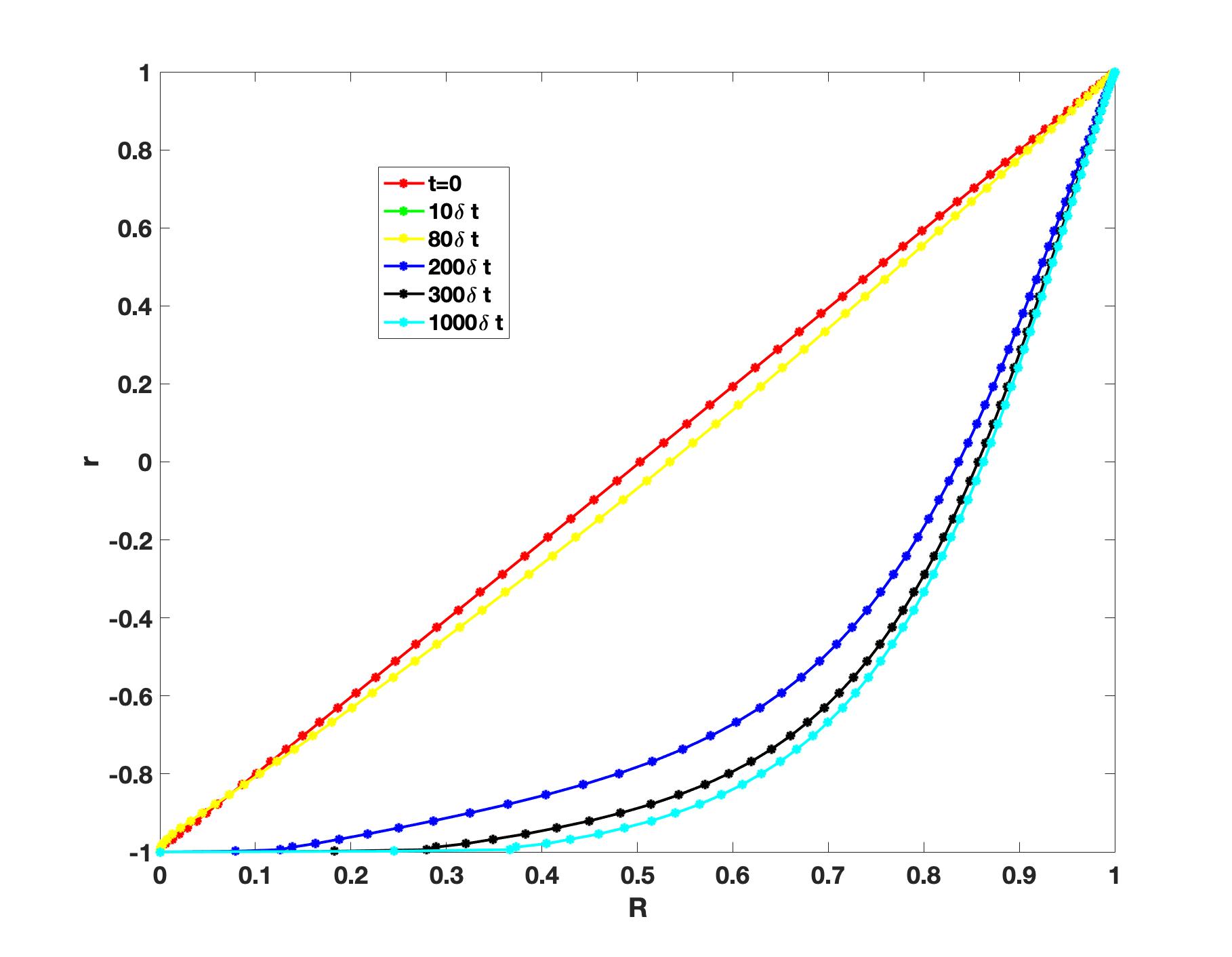}}
\caption{Axis-symmetric case computed by Lagrangian numerical method with $\eps^2=0.001$ and  time step $\delta t=10^{-4}$. }\label{axis-symmetric}
\end{figure}

\section{Concluding remarks}

We presented in this paper a new Lagrangian approach which can effectively capture the thin interface of  the Allen-Cahn type equations. Using the energetic variational approach, we introduced a transport equation and  
 reformulated the Allen-Cahn equation in Eulerian coordinates to a trajectory equation for the flow map in Lagrangian coordinates. We then developed effective energy stable schemes for the highly nonlinear trajectory equation, and presented ample numerical results to show the  effectiveness of this approach for interface capturing. 

The main advantage of the new approach is that meshes, in the Eulerian coordinate through the flow map,   automatically moves to the interfacial regions  so that only a few points are needed to resolve thin interfaces. In fact, the number of points required to resolve interfacial layers of width $\epsilon$ is independent of  $\epsilon$! 

 To fix the idea, we restricted ourselves to the one-dimensional case in this paper. In this case, the  assumption that the flow velocity satisfies the transport equation \eqref{transport} leads to a well-posed trajectory equation. But the transport equation \eqref{transport} is not a suitable choice for multi-dimensional cases as it will lead to a trajectory equation which is not well-posed.  However, the methodology introduced in this paper is still applicable for multi-dimensional cases and for other type of diffuse interface models such as Cahn-Hilliard models. The key is to use an alternative transport equation so that the resulting trajectory equation becomes well-posed.
 In a forthcoming paper, we shall apply the new Lagrangian approach introduced in this paper to multi-dimensional diffuse-interface models. 

\appendix
\section{Derivation by an energetic variational approach}
We shall use the energetic variational approach to derive the Allen-Cahn equation  \eqref{All:High}  using  flow map  \eqref{flow_map_AC} and kinematic relation \eqref{map:1}.

\subsection{\bf Energy dissipative law with flow map}
The energy dissipative law consisting the conservation function as well as the dissipation function plus kinematic relationship determine all the physical information for mathematical models.  So we combine original energy dissipative law \eqref{en:diss} with transport equation \eqref{map:1} together to define the singularity by using the Energetic Variational Approach.
If we plug the  kinematic equations \eqref{map:1} into the energy dissipative law \eqref{en:diss},  we can  derive a equivalent energy dissipative law with respect to flow map of equation \eqref{flow_map_AC} in Eulerian coordinate. For Allen-Cahn system \eqref{G_AC_Var:1}-\eqref{G_AC_Var:2}, we have the new energy dissipative law as

\begin{equation}\label{en:diss:flow1}
\left\{
\begin{aligned}
&\frac{d}{dt}\int_{\Omega_{\bx}}\frac 12|\Grad_{\bx} f|^2+\frac{1}{4\eps^2}(f^2-1)^2d\bx
=-\int_{\Omega_{\bx}}|\bu\cdot \Grad_{\bx} f|^2d\bx,\\
&f_t+(\bu\cdot\Grad_{\bx})f=0.
\end{aligned}
\right.
\end{equation}

Where the total free energy is $E^{total}:=\int_{\Omega_{\bx}}w(f)d\bx$ with free energy density $w:=\frac 12|\Grad_{\bx} f|^2+\frac{1}{4\eps^2}(f^2-1)^2$, and ${\bf \Delta}$ is represented as 
\begin{equation}\label{en:diss:flow3}
\begin{aligned}
{\bf \Delta} =\int_{\Omega_{\bx}}|\bu\cdot \Grad_{\bx} f|^2d\bx,
\end{aligned}
\end{equation}
which is dissipative term with respect to velocity $\bu$ and also can be regarded as entropy production from the Second Law of Thermodynamics.  In order to derive the constitution equation of Allen-Cahn equation in terms of force balance,   we need to introduce the framework  of Least Action principle and Maximum Dissipative principle.

\subsubsection{\bf Least Action Principle}
The  Least Action Principle \cite{abraham1978foundations,arnol2013mathematical} is interpreted as for a Hamiltonian system the trajectories of particles from position
$\bx(\bX,0)$ at time $t=0$ to position $\bx(\bX,T)$ at time $t=T$ are determined by the variational of Least Action function with respect to trajectory flow map. From energy dissipative law \eqref{en:diss},  for Allen-Cahn equation, the least action function is defined as
\begin{equation}
\begin{split}
A(\bx):&=-\int_0^{T}\mathcal{F}dt=-\int_0^{T}\int_{\Omega_{\bx}}w(\Grad_{\bx} f,f)d\bx dt\\&=-\int_0^{T}\int_{\Omega_{\bx}}\frac 12|\Grad_{\bx} f|^2+\frac{1}{4\eps^2}(f^2-1)^2d\bx dt,
\end{split}
\end{equation}
  where $\mathcal{F}$ is Helmholtz free energy and $w(\Grad_{\bx} f,f)=\frac 12|\Grad_{\bx} f|^2+\frac{1}{4\eps^2}(f^2-1)^2$.
Since from the kinematic relationship defined by equation \eqref{map:1}  in  Allen-Cahn system, we have the following equalities in Lagrange coordinate
\begin{equation}\label{traj:eq}
\begin{aligned}
f(\bx(\bX,t),t):=f(\bX,0)=f_0(\bX).
\end{aligned}
\end{equation}
By \eqref{traj:eq},  for Allen-Cahn system \eqref{G_AC_Var:1}-\eqref{G_AC_Var:2}, using deformation tensor $F$,  the action function is formulated as follows  in Lagrangian coordinate
\begin{equation}
A(\bx):=-\int_0^{T}\int_{\Omega_{\bX}}w(\Grad_{\bX} f_0(\bX)(\frac{\partial \bx}{\partial \bX})^{-1},f_0(\bX))det FdX dt.
\end{equation}
  Taking the variational derivative of action function $A(\bx)$ with respect to flow map $\bx \rightarrow \bx+\eps \by$ and combined with chain rule $\Grad_{\bx} f=\Grad_{\bX} f_0(\bX)(\frac{\partial \bx}{\partial \bX})^{-1}$, and notice the equality \eqref{traj:eq}.  Then we obtain

\begin{equation}
\begin{split}\label{LAF:eq}
\frac{\delta A}{\delta \bx}&=\frac{d}{d\eps}|_{\eps=0}\int_{\Omega_{\bx}}w(\Grad_{\bx} f(\bx+\eps \by),f(\bx+\eps \by))d\bx
\\&=\frac{d}{d\eps}|_{\eps=0}\int_{\Omega_{\bX}}w(\Grad_{\bX} f_0(\bX)(\frac{\partial (\bx+\eps \by)}{\partial \bX})^{-1},f_0(\bX))det(\frac{\partial (\bx+\eps 
\by)}{\partial \bX})d\bX
\\&=\int_{\Omega_{\bX}}\frac{\partial w(\Grad_{\bx} f,f)}{\partial \Grad_{\bx} f}(-F^{-1}\frac{\partial \by}{\partial \bX}F^{-1}\Grad_{\bX} f_0(\bX))det F+w(\Grad_{\bx} f,f)detF\cdot tr(F^{-T}\frac{\partial \by}{\partial \bX}) d\bX
\\&=\int_{\Omega_{\bX}}-\frac{\partial w(\Grad_{\bx} f,f)}{\partial \Grad_{\bx} f}\otimes \Grad_{\bx} f\frac{\partial \by}{\partial \bx}+w(\Grad_{\bx} f,f)\Grad_{\bx}\cdot \by d\bx
\\&=\int_{\Omega_{\bX}}\Grad_{\bx}\cdot(\frac{\partial w(\Grad_{\bx} f,f)}{\partial \Grad_{\bx} f}\otimes \Grad_{\bx} f-w(\Grad_{\bx} f,f)I)\by d\bx=\int_{\Omega_{\bx}}w_f\Grad_{\bx} f \by d\bx.
\end{split}
\end{equation}

Where $\frac{\delta w}{\delta f}=w_f=-\Delta_{\bx} f + \frac{1}{\eps^2}f(f^2-1)$,  $\frac{\delta w}{\delta f}$ is also called chemical potential and  $I$ is identity matrix.  According to Least Action Principle we have the conservative force as $F_{con}=\frac{\delta A}{\delta \bx}$ in Eulerian coordinate.

As a consequence,  we derive  that
\begin{equation}\label{con:force}
F_{con}=\frac{\delta A}{\delta \bx}=w_f\Grad_{\bx} f. 
\end{equation}
In order to derive the constitution equation, as we have computed the conservative force \eqref{con:force} from the Least Action principle, the dissipative force shall be obtained from the following Maximum dissipative principle.
\subsubsection{\bf Maximum Dissipative Principle}
The Maximum Dissipative Principle is also named as Onsager principle, .ie.  the dissipative force can be obtain by taking variational
of $\frac 12\bf \Delta$ with respect to velocity $\bu$. Since $\bf \Delta$ is said to be quadratic in the rates, so the force is linear
with respective rates.

\begin{equation}
F_{dis}=\frac{\delta\frac 12\bf \Delta}{\delta \bu}=
\bu\cdot\Grad_{\bx} f\Grad_{\bx} f.
\end{equation}

\subsection{\bf Force balance and constitution equation}
From Newton's force balance law,
\begin{equation}
F_{con}=F_{dis}
\end{equation}
We derive the constitution equation of Allen-Cahn equation in Eulerian coordinate in combination of conservative force and dissipative force, for system \eqref{G_AC_Var:1}-\eqref{G_AC_Var:2} 
\begin{equation}\label{All:force}
\begin{split}
&w_f\Grad_{\bx} f= \bu\cdot\Grad_{\bx} f\Grad_{\bx} f,\\
&w_f=-\Delta_{\bx} f + \frac{1}{\eps^2}f(f^2-1).
\end{split}
\end{equation}

\begin{remark}
For Allen-Cahn system \eqref{All:High}, taking inner product of \eqref{All:High} with $-\bu\cdot\Grad_{\bx} f$ and notice the equality $f_t=-\bu\cdot\Grad_{\bx} f$.  We can also derive the equivalent energy dissipative law \eqref{en:diss:flow1} in Eulerian coordinate. 
\end{remark}

\bibliographystyle{plain}

\bibliography{vesicle_ref}

\begin{thebibliography}{10}

\bibitem{abraham1978foundations}
Ralph Abraham, Jerrold~E Marsden, and Jerrold~E Marsden.
\newblock {\em Foundations of mechanics}, volume~36.
\newblock Benjamin/Cummings Publishing Company Reading, Massachusetts, 1978.

\bibitem{allen1979microscopic}
Samuel~M Allen and John~W Cahn.
\newblock A microscopic theory for antiphase boundary motion and its
  application to antiphase domain coarsening.
\newblock {\em Acta metallurgica}, 27(6):1085--1095, 1979.

\bibitem{arnol2013mathematical}
Vladimir~Igorevich Arnol'd.
\newblock {\em Mathematical methods of classical mechanics}, volume~60.
\newblock Springer Science \& Business Media, 2013.

\bibitem{bronsard1991motion}
Lia Bronsard and Robert~V Kohn.
\newblock Motion by mean curvature as the singular limit of ginzburg-landau
  dynamics.
\newblock {\em Journal of differential equations}, 90(2):211--237, 1991.

\bibitem{cao2002moving}
Weiming Cao, Weizhang Huang, and Robert~D Russell.
\newblock A moving mesh method based on the geometric conservation law.
\newblock {\em SIAM Journal on Scientific Computing}, 24(1):118--142, 2002.

\bibitem{chen2011mass}
Xinfu Chen, Danielle Hilhorst, and Elisabeth Logak.
\newblock Mass conserving allen--cahn equation and volume preserving mean
  curvature flow.
\newblock {\em Interfaces and Free Boundaries}, 12(4):527--549, 2011.

\bibitem{de2013non}
Sybren~Ruurds De~Groot and Peter Mazur.
\newblock {\em Non-equilibrium thermodynamics}.
\newblock Courier Corporation, 2013.

\bibitem{DuLiWa04}
Q.~Du, C.~Liu, and X.~Wang.
\newblock A phase field approach in the numerical study of the elastic bending
  energy for vesicle membranes.
\newblock {\em J. Comput. Phys.}, 198:450--468, 2004.

\bibitem{DuLiWa06}
Q.~Du, C.~Liu, and X.~Wang.
\newblock Simulating the deformation of vesicle membranes under elastic bending
  energy in three dimensions.
\newblock {\em J. Comput. Phys.}, 212:757--777, 2005.

\bibitem{Du.F19}
Qiang Du and Xiaobing Feng.
\newblock The phase field method for geometric moving interfaces and their
  numerical approximations.
\newblock {\em arXiv preprint arXiv:1902.04924}, 2019.

\bibitem{eisenberg2010energy}
Bob Eisenberg, Yunkyong Hyon, and Chun Liu.
\newblock Energy variational analysis of ions in water and channels: Field
  theory for primitive models of complex ionic fluids.
\newblock {\em The Journal of Chemical Physics}, 133(10):104104, 2010.

\bibitem{evans1992phase}
Lawrence~C Evans, H~Mete Soner, and Panagiotis~E Souganidis.
\newblock Phase transitions and generalized motion by mean curvature.
\newblock {\em Communications on Pure and Applied Mathematics},
  45(9):1097--1123, 1992.

\bibitem{feng2006spectral}
WM~Feng, Peng Yu, SY~Hu, Zi-Kui Liu, Qiang Du, and Long-Qing Chen.
\newblock Spectral implementation of an adaptive moving mesh method for
  phase-field equations.
\newblock {\em Journal of Computational Physics}, 220(1):498--510, 2006.

\bibitem{feng2004analysis}
Xiaobing Feng and Andreas Prohl.
\newblock Analysis of a fully discrete finite element method for the phase
  field model and approximation of its sharp interface limits.
\newblock {\em Mathematics of computation}, 73(246):541--567, 2004.

\bibitem{giga2017variational}
Mi-Ho Giga, Arkadz Kirshtein, and Chun Liu.
\newblock Variational modeling and complex fluids.
\newblock {\em Handbook of mathematical analysis in mechanics of viscous
  fluids}, pages 1--41, 2017.

\bibitem{greven2014entropy}
Andreas Greven, Gerhard Keller, and Gerald Warnecke.
\newblock {\em Entropy}, volume~47.
\newblock Princeton University Press, 2014.

\bibitem{gurtin2010mechanics}
Morton~E Gurtin, Eliot Fried, and Lallit Anand.
\newblock {\em The mechanics and thermodynamics of continua}.
\newblock Cambridge University Press, 2010.

\bibitem{hesthaven2008filtering}
Jan Hesthaven and Robert Kirby.
\newblock Filtering in legendre spectral methods.
\newblock {\em Mathematics of Computation}, 77(263):1425--1452, 2008.

\bibitem{huang1994moving}
Weizhang Huang, Yuhe Ren, and Robert~D Russell.
\newblock Moving mesh methods based on moving mesh partial differential
  equations.
\newblock {\em Journal of Computational Physics}, 113(2):279--290, 1994.

\bibitem{ilmanen1993convergence}
Tom Ilmanen et~al.
\newblock Convergence of the allen-cahn equation to brakke???s motion by mean
  curvature.
\newblock {\em J. Differential Geom}, 38(2):417--461, 1993.

\bibitem{katsoulakis1995generalized}
Markos Katsoulakis, Georgios~T Kossioris, and Fernando Reitich.
\newblock Generalized motion by mean curvature with neumann conditions and the
  allen-cahn model for phase transitions.
\newblock {\em The Journal of Geometric Analysis}, 5(2):255, 1995.

\bibitem{leslie1979theory}
Frank~M Leslie.
\newblock Theory of flow phenomena in liquid crystals.
\newblock In {\em Advances in liquid crystals}, volume~4, pages 1--81.
  Elsevier, 1979.

\bibitem{li2003thin}
Bo~Li and Jian-Guo Liu.
\newblock Thin film epitaxy with or without slope selection.
\newblock {\em European Journal of Applied Mathematics}, 14(06):713--743, 2003.

\bibitem{li2001moving}
Ruo Li, Tao Tang, and Pingwen Zhang.
\newblock Moving mesh methods in multiple dimensions based on harmonic maps.
\newblock {\em Journal of Computational Physics}, 170(2):562--588, 2001.

\bibitem{liu2003phase}
Chun Liu and Jie Shen.
\newblock A phase field model for the mixture of two incompressible fluids and
  its approximation by a fourier-spectral method.
\newblock {\em Physica D: Nonlinear Phenomena}, 179(3-4):211--228, 2003.

\bibitem{mackenzie2002moving}
JA~Mackenzie and ML~Robertson.
\newblock A moving mesh method for the solution of the one-dimensional
  phase-field equations.
\newblock {\em Journal of Computational Physics}, 181(2):526--544, 2002.

\bibitem{nesterov1994interior}
Yurii Nesterov and Arkadii Nemirovskii.
\newblock {\em Interior-point polynomial algorithms in convex programming},
  volume~13.
\newblock Siam, 1994.

\bibitem{onsager1931reciprocal}
Lars Onsager.
\newblock Reciprocal relations in irreversible processes. i.
\newblock {\em Physical review}, 37(4):405, 1931.

\bibitem{onsager1931reciprocal2}
Lars Onsager.
\newblock Reciprocal relations in irreversible processes. ii.
\newblock {\em Physical review}, 38(12):2265, 1931.

\bibitem{Shen94b}
J.~Shen.
\newblock Efficient spectral-{G}alerkin method {I}. direct solvers for second-
  and fourth-order equations by using {L}egendre polynomials.
\newblock {\em SIAM J. Sci. Comput.}, 15:1489--1505, 1994.

\bibitem{shen2009efficient}
Jie Shen and Xiaofeng Yang.
\newblock An efficient moving mesh spectral method for the phase-field model of
  two-phase flows.
\newblock {\em Journal of computational physics}, 228(8):2978--2992, 2009.

\bibitem{shen2010numerical}
Jie Shen and Xiaofeng Yang.
\newblock {Numerical approximations of Allen-Cahn and Cahn-Hilliard equations}.
\newblock {\em Discrete Contin. Dyn. Syst}, 28(4):1669--1691, 2010.

\bibitem{vandeven1991family}
Herv{\'e} Vandeven.
\newblock Family of spectral filters for discontinuous problems.
\newblock {\em Journal of Scientific Computing}, 6(2):159--192, 1991.

\bibitem{vazquez2007porous}
Juan~Luis V{\'a}zquez.
\newblock {\em The porous medium equation: mathematical theory}.
\newblock Oxford University Press, 2007.

\bibitem{xu2014energetic}
Shixin Xu, Ping Sheng, and Chun Liu.
\newblock An energetic variational approach for ion transport.
\newblock {\em arXiv preprint arXiv:1408.4114}, 2014.

\bibitem{zhang2009numerical}
Jian Zhang and Qiang Du.
\newblock Numerical studies of discrete approximations to the allen--cahn
  equation in the sharp interface limit.
\newblock {\em SIAM Journal on Scientific Computing}, 31(4):3042--3063, 2009.

\end{thebibliography}

\end{document}